\documentclass[11pt,reqno]{amsart}
\usepackage[margin=1in]{geometry}
\usepackage{color}
\usepackage{amsmath, amssymb, mathrsfs, mathtools, xcolor}
\usepackage{setspace}
\usepackage{enumerate}
\usepackage{comment}
\usepackage{cases}
\usepackage{enumitem}
\usepackage{latexsym}
\usepackage{hyperref}
\usepackage{bbm}
\usepackage{stmaryrd}
\usepackage{wasysym}
\usepackage{array}

\hypersetup{
  colorlinks   = true,
  urlcolor     = blue,
  linkcolor    = blue,
  citecolor    = blue
}

\numberwithin{equation}{section}
\newtheorem{thm}{Theorem}[section]
\newtheorem{prop}[thm]{Proposition}
\newtheorem{lem}[thm]{Lemma}
\newtheorem{cor}[thm]{Corollary}

\theoremstyle{definition}

\newtheorem{remark}[thm]{Remark}

\newcommand{\N}{\mathbb{N}}
\newcommand{\R}{\mathbb{R}}
\newcommand{\C}{\mathbb{C}}

\newcommand{\Z}{\mathbb{Z}}
\newcommand{\T}{\mathbb{T}}

\newcommand{\lb}{\langle}
\newcommand{\rb}{\rangle}
\newcommand{\supp}{\operatorname{supp}}
\renewcommand{\t}{\widetilde}

\renewcommand{\Re}{\operatorname{Re}}

\renewcommand{\hat}{\widehat}

\allowdisplaybreaks

\title[Nonlinear smoothing for dispersion generalized Benjamin-Ono]
{Nonlinear smoothing for the periodic dispersion generalized Benjamin-Ono equations with polynomial nonlinearity}
\author[W.~Shin]{Wangseok Shin}
\address{Department of Mathematics, University of Illinois}
\email{wshin14@illinois.edu}

\begin{document}

\begin{abstract}
We consider the periodic dispersion generalized Benjamin-Ono equations with polynomial nonlinearity. We establish the nonlinear smoothing properties of these equations, according to which the difference between the solution and the linear evolution is smoother than the initial data. In addition, we establish new local well-posedness results for these equations when the dispersion is sufficiently large. Our method also improves known local well-posedness results for a class of non-integrable fifth-order KdV equations.
\end{abstract}

\maketitle

\section{Introduction}
In this paper, we study the nonlinear smoothing properties of the periodic dispersion generalized Benjamin-Ono equations with polynomial nonlinearity given by
\begin{equation}\label{eq: dgbo}
\begin{cases}
\partial_{t} u + \partial_{x}D_{x}^{\alpha} u + \partial_{x}P(u)=0, \\
u(x,0) = g \in H^{s}(\T),
\end{cases}
(x,t) \in \T \times [-T,T], 
\end{equation}
where $P$ is a polynomial with real coefficients, $1<\alpha < 2$, $u=u(x,t)\in \R$, and $D_{x}^{\alpha}$ is the Fourier multiplier operator with symbol $|\xi|^{\alpha}$. The equation (\ref{eq: dgbo}) arises in physical contexts as a model for the vorticity waves in the coastal zone; see \cite{SV}. There are at least three conserved quantities for (\ref{eq: dgbo}),
\begin{align*}
    & \frac{d}{dt} \int_{\T} u(x,t)dx=0, & \text{(Mean conservation)}\\
    & \frac{d}{dt} \int_{\T} u(x,t)^2dx=0, & \text{(Mass conservation)}\\
    & \frac{d}{dt} \int_{\T} \frac{1}{2}\big|D^{\frac{\alpha}{2}}u(x,t)\big|^2 - F(u)(x,t)dx=0, & \text{(Energy conservation)}
\end{align*}
where
\begin{align*}
    F(x):=\int_{0}^{x}P(y)dy.
\end{align*}
Using the mean conservation, we will make the assumption 
\begin{align*}
     \int_{\T} u(x,t)dx=\int_{\T} g(x)dx=0
\end{align*}
throughout this paper. 

Nonlinear smoothing refers to the phenomenon in which the difference between the solution and the linear evolution obtains higher regularity than the initial data, i.e., for the linear evolution $S(t)g$ emanating from the initial data $g \in H^s$, $u(t)-S(t)g \in H^{s+a}$ holds for some $a>0$. This property has played a crucial role in proving various dynamical properties of nonlinear dispersive equations, for example, the polynomial-in-time growth bound of higher Sobolev norms \cite{B1}, global well-posedness below the energy space \cite{B2}, existence of a global attractor for the damped dispersive equations \cite{ET2}, nonlinear Talbot effect \cite{ET3}, and pointwise convergence of nonlinear flows \cite{CLS}.

For the KdV case ($\textnormal{deg}(P)=2$, $\alpha=2$), invoking the Duhamel formulation of the equation, a natural way to prove the nonlinear smoothing estimate would be to show the following:
\begin{align} \label{eq: Bourgain space estimate}
    \|\partial_{x}(u^2)\|_{X^{s+a,-\frac12}} \lesssim \|u\|_{X^{s,\frac12}}^2.
\end{align}
See subsection \ref{subsection: Basic notations} for the definition of $X^{s,b}$ spaces. Indeed, estimates similar to (\ref{eq: Bourgain space estimate}) have been used to prove nonlinear smoothing for various equations\footnote{Except for the KdV equation on $\R$, which we do not expect nonlinear smoothing; see \cite{IMT}.} on $\R^d$. Unfortunately, on $\T$, the estimate (\ref{eq: Bourgain space estimate}) fails for any $a>0$; see \cite[Remark 2]{ET3} for a counterexample. In \cite{ET1}, Erdo\u{g}an and Tzirakis observed that this difficulty can be avoided using the normal form transform combined with $X^{s,b}$ estimates, thereby proving nonlinear smoothing for the periodic KdV equation. Their method has also been successfully applied to various equations on $\T$, including the Zakharov system and the fractional Schr\"odinger equation \cite{ET2,ET4}. 

More recently, the nonlinear smoothing properties of the Benjamin-Ono equation, both on $\R$ and $\T$, have been investigated in \cite{C, GKT, IMOS, MosP}. In these works, it is shown that after applying Tao's gauge transform \cite{T}, solutions to the gauged Benjamin-Ono equation exhibit nonlinear smoothing. Note that these results do not imply nonlinear smoothing for the ungauged Benjamin-Ono equation, and it is likely that the original Benjamin-Ono equation lacks this property; see Appendix \ref{appendix: Upper bound of the smoothing exponent}.

We briefly review previous results regarding well-posedness of the initial value problem (\ref{eq: dgbo}) for $1<\alpha<2$. By local well-posedness, we mean the existence and uniqueness of the smooth solutions for initial data in $C^{\infty}(\T)$, and the continuous extendability of the data-to-solution map to $H^{s} \to C_{t}H^{s}_{x}$, at least locally in time. We say that the equation (\ref{eq: dgbo}) is globally well-posed if the local solutions admit global-in-time extensions.

\noindent\textbf{Case $\textnormal{deg}(P)=2$:} Herr \cite{H2} showed that when $1\leq \alpha<2$, the Piccard iteration based on the Duhamel formula
\begin{align} \label{eq: Herr Duhamel}
    u=e^{t\partial_{x}D^{\alpha}_{x}}g-\int_{0}^{t} e^{(t-t')\partial_{x}D^{\alpha}_{x}}[u \partial_{x}u]dt'
\end{align}
is not available for any Banach space embedded in $C_{t}H^{s}_{x}$. Molinet and Vento \cite{MV} showed that the Cauchy problem (\ref{eq: dgbo}) is well-posed for $s \geq 1-\frac{\alpha}{2}$ by combining the energy method and the Bourgain space estimates. Moreover, they showed that the solutions are unconditionally unique, in the sense that no other distributional solution exists in $L^{\infty}([-T,T],H^{s})$. Schippa \cite{S} used short-time Fourier restriction spaces to prove local well-posedness for $s>\frac{3}{2}-\alpha$ when $1<\alpha\leq\frac32$, and global well-posedness in $L^2(\T)$ when $\frac32<\alpha<2$. We note that well-posedness in the weakly dispersion regime $0<\alpha<1$ was examined in \cite{J}, where the author employed short-time analysis to establish local well-posedness for $s>\frac{3}{2}-\alpha$.

\noindent\textbf{Case $\textnormal{deg}(P) \geq 3$:} Kim and Schippa considered the generalized Benjamin-Ono equation ($\alpha=1$, $\textnormal{deg}(P) \geq 4$ case) in \cite{KS}, where they proved local well-posedness for $s\geq \frac34$ and existence of weak solutions for $s>\frac12$. Molinet and Tanaka \cite{MT1, MT2} proved that the equation (\ref{eq: dgbo}) is unconditionally well-posed for $s \geq 1-\frac{\alpha}{4}$ with $s>\frac12$. More precisely, they proved this result for a more general model
\begin{align*}
    \partial_{t} u + L_{\alpha+1} u +\partial_{x}f(u)=0
\end{align*}
where $L_{\alpha+1}$ is a Fourier multiplier operator that generalizes $\partial_{x}D_{x}^{\alpha}$, and $f:\R \to \R$ is a power series with an infinite radius of convergence.

\subsection{Main results} We summarize our main results below. As in \cite{CKSTT}, we introduce the gauge transform 
\begin{align*}
    \mathcal{G} u(x,t):=u\left(x - \int_{0}^{t}\int_{\T}P'(u)(x',t')dx' dt' , t\right),
\end{align*}
where $P'(x)=\frac{dP}{dx}$. This transform is invertible:
\begin{align*}
    \mathcal{G}^{-1} v(x,t)=v\left(x+\int_{0}^{t}\int_{\T}P'(v)(x',t')dx' dt' , t\right).
\end{align*}
Note that if $u$ is a mean-zero solution to the equation (\ref{eq: dgbo}) with $\textnormal{deg}(P)=2$, then $\mathcal{G}u=u$. 

\begin{thm} \label{thm: smoothing}
For $1 < \alpha < 2$ and $d \geq 2$, let
\begin{align*}
    s(d,\alpha):=
    \begin{cases}
    1-\frac{\alpha}{2}, & d=2, \\
    1-\frac{\alpha}{4}, & d=3, \\
    \frac{1}{2}+\frac{1}{3\alpha}, & d\geq 4,
    \end{cases}
    \quad
    a(d,\alpha):=
    \begin{cases}
    2s+\alpha-2, & d=2, \\
    2s+\frac{\alpha}{2}-2, & d=3,\\
    \frac{3\alpha}{\alpha+1}s-\frac{3\alpha+2}{2(\alpha+1)}, & d \geq 4.
    \end{cases}
\end{align*}
For $s> s(d,\alpha)$, let $u \in C([-T,T],H^{s})$ be a mean-zero solution to the equation (\ref{eq: dgbo}) with $\textnormal{deg}(P)=d$. Then we have
    \begin{align} \label{eq: nonlinear smoothing}
        \mathcal{G} u(t)-e^{t\partial_{x}D^{\alpha}_{x}}g \in C([-T,T],H^{s+a})
    \end{align}
for any $0\leq a < a(d,\alpha)$ with $a \leq \alpha-1$.
\end{thm}

When $\alpha=2$, the regularity threshold $s(d,2)$ in Theorem \ref{thm: smoothing} does not match the results in \cite{ET1, OS}, where the smoothing results are proved in the KdV case. Theorem \ref{thm: lwp} below gives nonlinear smoothing in lower regularity that is consistent with the KdV case; however, it does not cover the full range of $1< \alpha < 2$.

\begin{thm} \label{thm: lwp}
Let $\frac32<\alpha < 2$. Assume one of the following conditions:
\begin{enumerate}[label=(\alph*)]
    \item $\textnormal{deg}(P)=2$ and $s>\frac{1}{2}-\frac{\alpha}{2}$.
    \item $\textnormal{deg}(P)=3$ and $s>\frac12$. 
    \item $\textnormal{deg}(P)\geq 4$ and $s>\frac13+\frac{1}{3\alpha}$. 
\end{enumerate}
Then the Cauchy problem associated with the equation (\ref{eq: dgbo}) is locally well-posed. Moreover, the smoothing estimate (\ref{eq: nonlinear smoothing}) holds for some $a>0$.\footnote{The upper bound of the smoothing exponent $a$ can be explicitly computed, but is a bit complicated. We do not exactly quantify it here.}
\end{thm}

\begin{remark}
    The proof of Theorem \ref{thm: lwp} (a) easily extends to the $\alpha>2$ case with slight modifications. In this case, we obtain local well-posedness for $s>\frac12-\frac{\alpha}{2}$. This improves the results in \cite{Hi,LS}, where the authors proved local well-posedness for $s\geq-\frac{\alpha}{4}$. 
\end{remark}
\begin{remark}
    Theorem \ref{thm: lwp} (a) improves the well-posedness result in \cite{S} for $\alpha>\frac32$. Indeed, this is the first result to show that the equation (\ref{eq: dgbo}) with $P(x)=x^2$ is well-posed in negative Sobolev spaces. Moreover, the flow map of (a) is real-analytic, as the solutions are obtained via the Piccard iteration, and the gauge $\mathcal{G}$ in this case is the identity map. On the other hand, the analyticity fails for $s<\frac{1}{2}-\frac{\alpha}{2}$; see the proof of \cite[Theorem 1.2]{Kato}. In this sense, the regularity threshold in Theorem \ref{thm: lwp} is sharp up to the boundary.
\end{remark}
\begin{remark}
    Theorem \ref{thm: lwp} (b) and (c) partially improve the result in \cite{MT2} for $\alpha>\frac32$. We note that when $\alpha=1$ and $P(x)=x^3$, it is known that the equation (\ref{eq: dgbo}) is locally well-posed in $H^{\frac12}(\T)$, and the uniqueness is unconditional in $H^{\frac12+}(\T)$; see \cite{GLM, K}. However, these results rely on a gauge transform which does not directly extend to the $1<\alpha<2$ case.
\end{remark}

When we have nonlinear smoothing at the level of conservation law, iteration of the smoothing estimate gives polynomial growth bounds for higher Sobolev norms:
\begin{cor}\label{cor: polynomial growth} 
For $1<\alpha < 2$ and $s>\frac{\alpha}{2}$, let $u$ be a global solution to the equation (\ref{eq: dgbo}) with $\textnormal{deg}(P)=2$. Then the map $t\mapsto \|u(t)\|_{H^s}$ has at most polynomial growth. The same conclusion holds for the $\textnormal{deg}(P)\geq 3$ case, if $\frac{\alpha}{2}>s(\textnormal{deg}(P),\alpha)$ and $\sup_{t \in \R} \|u(t)\|_{H^{\alpha/2}}<\infty$.
\end{cor}
For the proof of Corollary \ref{cor: polynomial growth}, see \cite{ET1}. For the conditions that guarantee the global-in-time bound $\sup_{t \in \R} \|u(t)\|_{H^{\alpha/2}}<C(\|g\|_{H^{\alpha/2}})$, see \cite[Remark 1.6]{MT1}. Also note that we do not need the mean-zero assumption for this corollary since $u-\int g dx$ solves (\ref{eq: dgbo}) with $P$ replaced with $P\big(\;\cdot\;+ \int g dx\big)$.

Other possible applications of the smoothing estimate (\ref{eq: smoothing}) include the existence of a global attractor for the forced and damped version of (\ref{eq: dgbo}), or estimates of the fractional dimension of the graphs of the solutions. See \cite{Mc1, EG} for relevant studies in this direction.

\subsection{Basic notations} \label{subsection: Basic notations}

\begin{itemize}
  \item For $A, B \geq 0$, we write $A \lesssim B$ when there is a constant $C=C(s,\alpha,P)>0$ such that $A \leq CB$ for $s,\alpha,P$ given in (\ref{eq: dgbo}). We write $A \approx B$ when $A \lesssim B$ and $B \lesssim A$ both hold. The complement of $A \lesssim B$ is denoted by $A \gg B$.
  \item For $a\in \R$, the expression $a+$ means $a+\epsilon$ for all $0<\epsilon\ll 1$.  Similarly, $a-$ means $a-\epsilon$ for all $0<\epsilon\ll 1$.
  \item Let $\T=[0,2\pi)$. Define the Fourier transform for $2\pi$-periodic functions as
\begin{align*}
    \hat{f}(\xi)=\hat{f}_\xi=\mathcal{F}_{x}[f](\xi)=\frac{1}{2\pi}\int_{\T}e^{-i \xi x}f(x)dx.
\end{align*}
    If $u(x,t)$ is a function on $\T \times \R$, we define its space-time Fourier transform $\hat{u}(\xi,\tau)$ by
\begin{align*}
    \hat{u}(\xi,\tau)=\mathcal{F}_{x,t}[u](\xi,\tau)=\frac{1}{2\pi}\int_{\T \times \R}e^{-i(x\xi+t\tau)}u(x,t)dx dt.
\end{align*} 
We write $\mathcal{F}^{-1}f=f^{\vee}$ for the inverse Fourier transform.
  \item For $\alpha \in \R$, we define $D^{\alpha}f$ by
  \begin{align*}
      \mathcal{F}[D^{\alpha}f](\xi):=|\xi|^{\alpha}\hat{f}(\xi).
  \end{align*}
  \item Let $\Z_{\ast}:=\Z \setminus \{0\}$. Define the Sobolev space $H^{s}(\T)$ by the norm
\begin{align*}
 \| f \|_{H^{s}(\T)}^2=|\hat{f}(0)|^2+\sum_{\xi \in \Z_{\ast}}|\xi|^{2s}|\hat{f}(\xi )|^{2}.
\end{align*}
    \item For $T>0$ and a Banach space $X$, we write $L^{p}_{T}X$ for the space $L^{p}([-T,T],X)$.
    \item Let $\omega(\xi)=|\xi|^{\alpha}\xi$ and for $n \geq 1$, define the $n$'th order resonance function $\Omega_{n}$ by
\begin{align*}
    \Omega_{n}(\xi_1,\dots,\xi_{n}):=\omega\left(\sum_{i=1}^{n}\xi_i \right)-\sum_{i=1}^{n}\omega(\xi_i).
\end{align*}
    \item For $g \in H^s(\T)$, the free evolution $S(t)g=e^{t\partial_{x}D^{\alpha}_{x}}g$ associated with (\ref{eq: dgbo}) is defined by $$\hat{S(t)g}(\xi )=e^{it\omega(\xi)}\hat{g}(\xi).$$
    \item For $s,b \in \R$, the space $X^{s,b}$ is defined by the norm
    \begin{align*}
        \|u\|_{X^{s,b}}=\left\|\lb \xi \rb^{s}\lb \tau-\omega(\xi) \rb^{b}\hat{u}(\xi,\tau)\right\|_{L^{2}_{\xi,\tau}}.
    \end{align*}
    \item For a dyadic integer $N \in \{2^n:n\in\N\cup\{0\}\}$, $P_N$ denotes the standard Littlewood-Paley projector on $\T$.
    \item Let $\mathcal{A}$ be a subset of the frequency space. We write $\mathbbm{1}_{A}$ for the indicator function of $\mathcal{A}$. Denote $\mathbb{P}_{\mathcal{A}}$ with the frequency projection
    \begin{align*}
        \mathbb{P}_{\mathcal{A}}f:=\mathcal{F}^{-1}[ \mathbbm{1}_{\mathcal{A}}\hat{f}\ ].
    \end{align*}
    \item For a finite sequence of real numbers $\{a_i\}_{1\leq i \leq N}$, let $\{a^*_i\}_{1\leq i \leq N}$ be a rearrangement of $\{a_i\}_{1\leq i \leq N}$ such that $|a^*_1| \geq |a^*_2| \geq \dots \geq |a^*_N|$.
    \item For $f:\Z^{n} \to \C$, $k \in \N$ and $1\leq j\leq n$, define the \textit{elongation} $\mathbf{X}^{k}_{j}(f):\Z^{n+k-1} \to \C$ by
\begin{align*}
    \mathbf{X}^{k}_{j}(f)(\xi_1,\dots,\xi_{n+k-1})=f(\xi_1,\dots,\xi_{j-1},\xi_{j}+\cdots+\xi_{j+k-1},\xi_{j+k},\dots,\xi_{n+k-1}).
\end{align*}
    \item Let $\mathcal{A} \subseteq \Z^{n}_{\ast}$. For $P,Q:\mathcal{A}\to \C$, we write
\begin{align*}
    \sum_{x \in \mathcal{A}}^{\ast}\frac{P(x)}{Q(x)}:= \sum_{\substack{x\in \mathcal{A} \\ Q(x) \neq 0}}\frac{P(x)}{Q(x)}.
\end{align*}
    \item For $n \in \N$, the symmetric group $\operatorname{Sym}\{1,\dots,n\}$ is denoted by $S_{n}$.
    \item For $\sigma \in S_n$ and $f: \Z^{n} \to \C$, we write 
    \begin{align*}
        \sigma \cdot f(\xi_1,\dots,\xi_n):=f(\xi_{\sigma(1)},\dots,\xi_{\sigma(n)}).
    \end{align*}
    \item We write $\{a_1^{n_1},\dots,a_m^{n_m}\}$ for the multiset $\{a_1,\dots,a_1,a_2,\dots,a_2,\dots,a_m,\dots a_m\}$, where the element $a_i$ is repeated $n_i$ times.
    \item For a multiset $\mathcal{A}=\{a_1^{n_1},\dots,a_m^{n_m}\}$, we define its set of permutations $\textnormal{Perm}(\mathcal{A})$ as the set of all bijections from $\{1,2,\dots,n_1+\cdots+n_m\}$ onto $\mathcal{A}$.
    \item For $d \in \N$, we write $[d]$ for the set $\{1,2,\dots,d\}$.
    \item The numbers $\nu_n$ and $\nu_{n,m}$ defined in (\ref{eq: nu n}) and (\ref{eq: nu nm}) will be repeatedly used with the same notations throughout this paper.
    \item For $1 \leq p_1, \dots,p_n \leq \infty$ with $\frac{1}{p_1}+\cdots+\frac{1}{p_n}=1$, by the $(p_1,\dots,p_n)$-H\"older inequality, we mean
    \begin{align*}
        \left |\int f_1\cdots f_n dx \right | \leq \|f_1\|_{L^{p_1}}\cdots \|f_n\|_{L^{p_n}}.
    \end{align*}
\end{itemize}

\subsection{Comments on the proof} \label{subsection: Comments on the proof}
The main ingredient of our proof is normal form reduction, which has been introduced in \cite{Sh, BIT}. For a brief overview of our method, let us look at the $P(x)=x^2$ case. Consider the \textit{interaction variable} $v:=S(-t)u$. Then on the Fourier side, $v$ satisfies the equation 
\begin{align} \label{eq: interaction k=2}
    \partial_{t}\hat{v}_{\xi} =\frac{\xi}{i}\sum_{\xi_1+\xi_2=\xi}e^{-i\Omega_{2}(\xi_1,\xi_2)t}\hat{v}_{\xi_1}\hat{v}_{\xi_2}.
\end{align}
Then we perform \textit{differentiation by parts} to the summand of the right-hand side:
\begin{align*}
    e^{-i\Omega_{2}(\xi_1,\xi_2)t}\hat{v}_{\xi_1}\hat{v}_{\xi_2}=\partial_{t}\left( \frac{i}{\Omega_{2}(\xi_1,\xi_2)}e^{-i\Omega_{2}(\xi_1,\xi_2)t}\hat{v}_{\xi_1}\hat{v}_{\xi_2}\right)-\frac{i}{\Omega_{2}(\xi_1,\xi_2)}e^{-i\Omega_{2}(\xi_1,\xi_2)t}\partial_{t}\left(\hat{v}_{\xi_1}\hat{v}_{\xi_2}\right).
\end{align*}
Now we have $\Omega_2(\xi_1,\xi_2)$ in the denominator that can be utilized to weaken the derivative in the nonlinearity. This mechanism can be viewed as a smoothing effect caused by rapid temporal oscillation in the phase space. Roughly speaking, this process smoothes the nonlinearity out by $\partial^{-\alpha+1}_{x}$. For $1<\alpha<2$, this gain is not enough to recover one derivative loss in $\partial_x P(u)$, necessitating us to repeat this process at least $\frac{1}{\alpha-1}$ times. The major challenge in this analysis is to estimate various multilinear terms produced by the iteration. This requires delicate cancellation properties in the frequency domain, which is a main component of our proof.

The proof of Theorem \ref{thm: lwp} is based on the idea developed by McConnell \cite{Mc2}, where the author proved local well-posedness for the fifth-order KdV equation in $H^{s}(\T)$ for $s>\frac{35}{64}$. One of the main ideas of this work is to perform normal form reduction on the original equation, and then apply Bourgain's $X^{s,b}$ method to the reduced equation. Essentially this argument is a Piccard iteration. Nevertheless, the aforementioned Herr's negative result \cite{H2} does not affect applying this method to (\ref{eq: dgbo}), since we won't be directly iterating (\ref{eq: Herr Duhamel}). The difficulty in this analysis is that the `boundary terms' created in the normal form reduction cannot be appropriately estimated in $X^{s,\frac12}$. To remedy this, \cite{Mc2} adopted the space $X^{s,\frac14}\cap W^s$, where $W^s$ is the space defined by the norm $\|\lb \xi \rb^{s}\hat{u}\|_{\ell^{2}_{\xi}L^{1}_{\tau}}$. If we directly apply the same approach to the equation (\ref{eq: dgbo}), we would need the space $X^{s,1-\frac{1}{\alpha}}$. However, the reduced modulation exponent $1-\frac{1}{\alpha}$ causes many difficulties by limiting the full strength of Strichartz estimates. In the present work, rather than utilizing this space, we use the alternative space $Y^{s}$ defined by the norm
\begin{align*}
    \|u\|_{Y^{s}}:=\|\mathbb{P}_{D_1}u\|_{X^{s-1+\frac{\alpha}{2}-\epsilon,\frac12}}+\|\mathbb{P}_{D_2}u\|_{X^{s,\frac12-\epsilon}}+\|u\|_{W^s},
\end{align*}
where $\mathbb{P}_{D_1}$ and $\mathbb{P}_{D_2}$ are the spacetime frequency projections defined in Section \ref{section: Low regularity well-posendess and smoothing}. The main advantage of the space $Y^s$ compared to the space $X^{s,1-\frac{1}{\alpha}}$ is the flexible usability of Strichartz estimates, which usually require modulations close to $\frac12$. This method also improves the local well-posedness result in \cite{Mc2}; see Appendix \ref{appendix: The fifth-order KdV equation}.

\subsection{Outline of the paper} In Section \ref{section: Analysis of the resonance functions}, we discuss some important properties of the resonance functions. In Section \ref{section: Normal form reduction}, we perform the normal form reduction. In Section \ref{section: Preliminary lemmas}, we present some basic facts required to prove the results in Section \ref{section: Pointwise estimates}. In Section \ref{section: Pointwise estimates} we estimate the size of the symbols of the multiplier operators defined in Section \ref{section: Normal form reduction}. In Section \ref{section: Smoothing estimates}, we establish some smoothing estimates for the Sobolev norms of the relevant terms, and then prove Theorem \ref{thm: smoothing}. In Section \ref{section: Low regularity well-posendess and smoothing}, we prove Theorem \ref{thm: lwp} using Bourgain's $X^{s,b}$ spaces. In Appendix \ref{appendix: Upper bound of the smoothing exponent}, we demonstrate the upper bound of the smoothing exponent. Finally in Appendix \ref{appendix: The fifth-order KdV equation}, we outline how to extend our method to a class of fifth-order KdV equations.

\section{Analysis of the resonance functions}\label{section: Analysis of the resonance functions}

Below, we record some basic estimates for the resonance functions $\Omega_n$. 
When $\alpha=2$, we have the following formulas:
\begin{align*}
    &\Omega_{2}(\xi_1,\xi_2)=3\xi_1 \xi_2 (\xi_1+\xi_2), \\
    &\Omega_{3}(\xi_1,\xi_2,\xi_3)=3(\xi_1+\xi_2)(\xi_2+\xi_3)(\xi_1+\xi_3).
\end{align*}
Exact algebraic identities like these are no longer available for $1<\alpha<2$. Nevertheless, we still have the following simple asymptotic expressions, which are valid for both small and large frequencies:
\begin{lem}\cite[Lemma 5.1]{G} \label{lem: phase function asymptotics 1}
For $\alpha \geq 1$ and $\xi_1, \xi_2 \in \Z$, we have 
\begin{align*}
    |\Omega_{2}(\xi_1,\xi_2)| \approx |\xi_1^*|^{\alpha}\min\left(|\xi_1+\xi_2|, |\xi_2^*| \right).
\end{align*}
\end{lem}
For $n\geq 2$, define
    \begin{align*}
        \rho_n(\xi_1,\dots,\xi_n):=\min\left(\lb \xi_1^*+\xi_2^* \rb, \lb \xi_2^*+\cdots+\xi_n^* \rb \right).
    \end{align*}

\begin{lem}\cite[Lemma 4.1]{GH} \label{lem: phase function asymptotics 2}
    For $\alpha \geq 1$ and $\xi_1, \xi_2, \xi_3 \in \Z$, we have
    \begin{multline*}
        |\Omega_{3}(\xi_1,\xi_2,\xi_3)|
        \approx \max(|\xi_1+\xi_2|,|\xi_2+\xi_3|,|\xi_1+\xi_3|)^{\alpha-1}\\ \times \operatorname{med}(|\xi_1+\xi_2|,|\xi_2+\xi_3|,|\xi_1+\xi_3|)\min(|\xi_1+\xi_2|,|\xi_2+\xi_3|,|\xi_1+\xi_3|).
    \end{multline*}
In particular, for $\xi_1,\xi_2,\xi_3\in\Z_{\ast}$ such that
\begin{align*}
    (\xi_1,\xi_2,\xi_3) \notin \{\xi_1+\xi_2+\xi_3=\xi_i \textnormal{ for some } 1\leq i \leq 3\}\cup \{|\xi_1|\approx|\xi_2|\approx|\xi_3|\},
\end{align*}
we have $|\Omega_{3}(\xi_1,\xi_2,\xi_3)|\gtrsim |\xi_1^*|^{\alpha}\rho_3(\xi_1,\xi_2,\xi_3)$.
\end{lem}

For $n \geq 4$, there is no known factorization representation of $\Omega_{n}$. Instead, we have the following decomposition as in \cite[Proposition 2]{OS}:
\begin{lem} \label{lem: higher-order resonances}
    For $n\geq 4$, let $\xi=\xi_1+\dots+\xi_{n}$. Then at least one of the following holds true:
    \begin{enumerate}[label=(\alph*)]
        \item $|\xi_1^*|\gg|\xi_2^*|$ and $|\Omega_{n}(\xi_1,\dots,\xi_{n})|\gtrsim |\xi_1^*|^{\alpha}\lb\xi-\xi_1^*\rb$.
        \item $|\xi_1^*|\gg|\xi_2^*|$ and $|\xi_{3}^*|^{\alpha}|\xi_4^*| \gtrsim |\xi_1^*|^{\alpha} \lb\xi-\xi_1^*\rb$.
        \item $|\xi_1^*|\gg|\xi_2^*|$ and $\xi=\xi_1^*$.
        \item $|\xi_1^*|\approx|\xi_2^*|$ and $|\Omega_{n}(\xi_1,\dots,\xi_{n})|\gtrsim |\xi_1^*|^{\alpha}\rho_n(\xi_1,\dots,\xi_n)$.
        \item $|\xi_1^*|\approx|\xi_2^*|$ and $|\xi_3^*|\gtrsim |\xi|$.
    \end{enumerate}
\end{lem}
\begin{proof}
 Without loss of generality assume that $|\xi_1| \geq |\xi_2| \geq \dots \geq |\xi_{n}|$. Suppose for a contradiction that the followings hold:
    \begin{enumerate}[label=(\Roman*)]
        \item $|\Omega_{n}(\xi_1,\dots,\xi_{n})|\ll |\xi_{1}|^{\alpha}\min\big(|\xi_1+\xi_2|,|\xi-\xi_1|\big)$, 
        \item $|\xi_{3}|^{\alpha}|\xi_4| \ll |\xi_1|^{\alpha}\min\big(|\xi_1+\xi_2|,|\xi-\xi_1|\big)$, 
        \item $|\xi| \gg |\xi_3|$,
        \item $\xi \neq \xi_1$.
    \end{enumerate}
    First observe that (III) and (IV) imply $\min\big(|\xi_1+\xi_2|,|\xi-\xi_1|\big) \neq 0$. Let $\widetilde{\xi}_{i}:=\sum_{j=i}^{n}\xi_j$. Then
    \begin{align} \label{eq: omega n}
    \begin{split}
        \Omega_{n}(\xi_1,\dots,\xi_{n}) 
        &= \omega(\xi)-\sum_{i=1}^{n}\omega(\xi_i) \\
        &= 
        \begin{multlined}[t][0.7\linewidth]
            (\omega(\xi)-\omega(\xi_1)-\omega(\xi_2)-\omega(\widetilde{\xi}_{3})) \\
            + (\omega(\widetilde{\xi}_{3})-\omega(\xi_3)-\omega(\widetilde{\xi}_{4})) + \cdots +(\omega(\widetilde{\xi}_{n-1})-\omega(\xi_{n-1})-\omega(\xi_{n}))
        \end{multlined} \\
        &=\Omega_{3}(\xi_1,\xi_2,\widetilde{\xi}_{3})+\sum_{i=3}^{n-1}\Omega_{2}(\xi_i,\widetilde{\xi}_{i+1}).
    \end{split}
    \end{align}
    By Lemma \ref{lem: phase function asymptotics 2}, we have
    \begin{align*}
        |\Omega_{3}(\xi_1,\xi_2,\widetilde{\xi}_{3})| \gtrsim |\xi_1|^{\alpha-1} \max(|\xi_1+\xi_2|,|\xi-\xi_1|) \min(|\xi_1+\xi_2|,|\xi-\xi_1|).
    \end{align*}
    Note that if $|\xi_1+\xi_2| \ll |\xi_1|$, then by (III) we have $|\xi|\approx|\xi_1+\xi_2| \ll |\xi_1|$, hence $|\xi-\xi_1| \approx |\xi_1|$. Therefore in any case, we have $\max(|\xi_1+\xi_2|,|\xi-\xi_1|) \approx |\xi_1|$ and $|\Omega_{3}(\xi_1,\xi_2,\widetilde{\xi}_{3})| \gtrsim |\xi_1|^{\alpha}\min\big(|\xi_1+\xi_2|,|\xi-\xi_1|\big)$. This implies by assumption (I) and (\ref{eq: omega n}) that we have
    \begin{align*}
        \left| \sum_{i=3}^{n-1}\Omega_{2}(\xi_i,\widetilde{\xi}_{i+1}) \right| \gtrsim |\xi_1|^{\alpha}\min\big(|\xi_1+\xi_2|,|\xi-\xi_1|\big).
    \end{align*}
    By Lemma \ref{lem: phase function asymptotics 1}, for $3 \leq i \leq n$,
    \begin{align*}
        |\Omega_{2}(\xi_i,\widetilde{\xi}_{i+1})|\approx \max(|\xi_i|,|\widetilde{\xi}_{i}|,|\widetilde{\xi}_{i+1}|)^{\alpha}\min(|\xi_i|,|\widetilde{\xi}_{i}|,|\widetilde{\xi}_{i+1}|) \lesssim |\xi_3|^{\alpha}|\xi_4|.
    \end{align*}
    Therefore we have $|\xi_3|^{\alpha}|\xi_4| \gtrsim |\xi_1|^{\alpha}\min\big(|\xi_1+\xi_2|,|\xi-\xi_1|\big)$, which contradicts (II). Therefore, at least one of (I), (II), (III), or (IV) must be false. In other words, at least one of the following statements holds true:
    \begin{enumerate}[label=(\roman*)]
        \item $|\Omega_{n}(\xi_1,\dots,\xi_{n})|\gtrsim |\xi_{1}|^{\alpha}\min\big(|\xi_1+\xi_2|,|\xi-\xi_1|\big)$.
        \item $|\xi_{3}|^{\alpha}|\xi_4| \gtrsim |\xi_1|^{\alpha}\min\big(|\xi_1+\xi_2|,|\xi-\xi_1|\big)$.
        \item $|\xi_3| \gtrsim |\xi|$. 
        \item $\xi = \xi_1$.
    \end{enumerate}
    If $|\xi_1|\gg|\xi_2|$, then $\min\big(|\xi_1+\xi_2|,|\xi-\xi_1|\big)=|\xi-\xi_1|$, which proves the trichotomy (a), (b) and (c). To show the dichotomy (d) and (e) in the $|\xi_1|\approx|\xi_2|$ case, we demonstrate that (ii) and (iv) can be absorbed into (iii). It is easy to see that (iv) with $|\xi_1|\approx|\xi_2|$ implies (iii), since $\xi=\xi_1$, implies $|\xi_2| \approx |\xi_3|$. Now we assume $|\xi_1|\approx|\xi_2|$ and concentrate on (ii). Observe that if $|\xi_1+\xi_2| \ll |\xi_3|$, then we have $|\xi_3|\gtrsim |\xi|$, which implies (iii). Hence we may assume $|\xi_1+\xi_2| \gtrsim |\xi_3|$ and replace the inequality in (ii) by
    \begin{align} \label{eq: 34 min}
        |\xi_{3}|^{\alpha}|\xi_4| \gtrsim |\xi_1|^{\alpha}\min\big(|\xi_3|,|\xi-\xi_1|\big).
    \end{align}
    If $\min(|\xi_3|,|\xi-\xi_1|)=|\xi_3|$, then the inequality (\ref{eq: 34 min}) implies (iii). Hence we may only consider the case $\min(|\xi_3|,|\xi-\xi_1|)=|\xi-\xi_1|$. This forces $|\xi_2|\approx|\xi_3|$ to hold by (\ref{eq: 34 min}). Therefore, we have $|\xi_1|\approx|\xi_3|$, which we can absorb into case (iii). 
\end{proof}

\section{Normal form reduction} \label{section: Normal form reduction}

\subsection{The gauged equation}

We perform a gauge transform to eliminate some bad resonant frequency interactions in the nonlinearity. For $P(x)=\sum_{k=1}^{d}c_{k}x^{k}$, let $u$ be a solution to (\ref{eq: guage dgbo}). Recall that the gauge transform $\mathcal{G}$ is defined as
\begin{align*}
    \mathcal{G} u(x,t)=u\left(x-\int_{0}^{t}\int_{\T}P'(u)(x',t')dx' dt' , t\right).
\end{align*}
Then $\t u:=\mathcal{G} u$ satisfies
\begin{equation}\label{eq: guage dgbo}
\begin{cases}
\partial_{t} \t u + \partial_{x}D_{x}^{\alpha} \t u + \mathbf{P} \left(P'(\t u) \right) \partial_{x}\t u =0, \\
\t u(x,0) = g,
\end{cases}
\end{equation}
where $\mathbf{P}$ denotes the projection onto mean-zero functions
\begin{align*}
    \mathbf{P} \left(v \right):=v-\int_{\T} v(x) dx.
\end{align*}
In the sequel, we will work with the guaged equation (\ref{eq: guage dgbo}) instead of the original equation (\ref{eq: dgbo}). For simplicity, we write $u$ for $\t u$ in (\ref{eq: guage dgbo}). Note that, since
\begin{multline*}
    \mathcal{F}\left[\int_{\T}k u^{k-1}(x')dx' \partial_{x}u\right](\xi)
    =ik\xi\hat{u}(\xi)\sum_{\xi_1+\cdots+\xi_{k-1}=0}\hat{u}(\xi_1)\cdots\hat{u}(\xi_{k-1}) \\
    =ik\xi\sum_{\xi_1+\cdots+\xi_{k}=\xi}\mathbbm{1}_{\{\xi_=\xi_k\}}\hat{u}(\xi_1)\cdots\hat{u}(\xi_{k})
    =i\xi\sum_{\xi_1+\cdots+\xi_{k}=\xi}\sum_{j=1}^{k}\mathbbm{1}_{\{\xi_=\xi_j\}}\hat{u}(\xi_1)\cdots\hat{u}(\xi_{k}),
\end{multline*}
the Fourier side of the nonlinear term of (\ref{eq: guage dgbo}) can be written as
\begin{align*}
    \mathcal{F}_{x}\left[\mathbf{P} \left(P'(u) \right) \partial_{x} u\right](\xi)=i\sum_{k=2}^{d}c_k\left[\sum_{\xi_{1}+\cdots+\xi_{k}=\xi} \xi \phi_{k}(\xi_1,\dots,\xi_{k}) \hat{u}(\xi_{1}) \cdots \hat{u}(\xi_{k})\right],
\end{align*}
where
\begin{align} \label{eq: phi k definition}
        \phi_{k}(\xi_1,\dots,\xi_{k}):=\mathbbm{1}_{\Z^{k}_{\ast}}-\sum_{j=1}^{k}\mathbbm{1}_{\{(\xi_1,\dots,\xi_k)\in\Z^{k}_{\ast} \,:\, \xi_j=\xi_1+\cdots+\xi_k \}}.
\end{align}
Observe that for $k \geq 2$, we have
\begin{align} \label{eq: 0th resonance vanishes}
    \mathbbm{1}_{\left\{|\xi_1^*|\gg|\xi_2^*|\right\}}\phi_{k}(\xi_1,\dots,\xi_k)=\mathbbm{1}_{\left\{|\xi_1^*|\gg|\xi_2^*|\right\} \cap \left\{\xi_2^*+\cdots+\xi_k^*\neq 0\right\}}(\xi_1,\dots,\xi_k).
\end{align}

\subsection{Decomposition of the frequency domain} \label{subsection: Decomposition of the frequency domain}

For each $n \geq 2$ we divide the frequency domain $\Z_{\ast}^{n}$ as
\begin{align}\label{eq: partition}
    \Z_{\ast}^{n}=\mathcal{R}^{1}_{n} \sqcup \mathcal{R}^{2}_{n} \sqcup \mathcal{N}_{n} \sqcup  \mathcal{D}_{n}.
\end{align}
First, let $$\mathcal{R}^{1}_{2}:=\varnothing, \ \mathcal{R}^{2}_{2}:=\varnothing, \ \mathcal{N}_{2}:=\Z_{\ast}^{2}, \ \mathcal{D}_{2}:=\varnothing.$$ 
Next, we use Lemma \ref{lem: phase function asymptotics 2} to partition $\Z_{\ast}^{3}$ into the sets
\begin{align*}
    &\mathcal{R}^{1}_{3} = \left\{(\xi_1,\xi_2,\xi_{3}) \in \Z_{\ast}^{3}:\xi_{i}=\xi_1+\xi_2+\xi_{3} \text{ for some } 1 \leq i \leq 3\right\}, \\
    &\mathcal{R}^{2}_{3} = \left\{(\xi_1,\xi_2,\xi_{3}) \in \Z_{\ast}^{3}:|\xi_{1}|\approx|\xi_{2}|\approx|\xi_{3}|\right\} \setminus \mathcal{R}^{1}_{3}, \\
    &\mathcal{N}_{3} = \Z_{\ast}^{3} \setminus (\mathcal{R}^{1}_{3} \cup \mathcal{R}^{2}_{3}), \\
    &\mathcal{D}_{3}=\varnothing.
\end{align*}
Recall that by Lemma \ref{lem: phase function asymptotics 2}, we have $|\Omega_3(\xi_1,\xi_2,\xi_3)|\gtrsim |\xi_1^*|^{\alpha}\rho_3(\xi_1,\xi_2,\xi_3)$ for $(\xi_1,\xi_2,\xi_3) \in \mathcal{N}_{3}$.

For $n \geq 4$, let $$\mathcal{X}_{n}:=\{(\xi_1,\dots,\xi_{n}) \in \Z_{\ast}^{n}:|\xi_1^*| \gg |\xi_2^*|\}.$$ We use Lemma \ref{lem: higher-order resonances} to partition $\Z_{\ast}^{n}$ into
\begin{align*}
    &\mathcal{R}^{1}_{n} \subseteq \left\{(\xi_1,\dots,\xi_{n}) \in \mathcal{X}_{n}:\xi_{1}^{*}=\xi_1+\dots+\xi_{n} \textnormal{ and } |\xi_{3}^*|^{\alpha}|\xi_4^*| \ll |\xi_1^*|^{\alpha}\right\},\\
    &\mathcal{R}^{2}_{n} \subseteq \left\{(\xi_1,\dots,\xi_{n}) \in \mathcal{X}_{n}:|\xi_{3}^*|^{\alpha}|\xi_4^*| \gtrsim |\xi_1^*|^{\alpha} \lb\xi_2^*+\cdots+\xi_n^*\rb\right\}, \\
    &\mathcal{N}_{n} \subseteq \left\{(\xi_1,\dots,\xi_{n}) \in \mathcal{X}_{n}: |\Omega_{n}(\xi_1,\dots,\xi_{n})| \gtrsim |\xi_{1}^{*}|^{\alpha}\lb\xi_2^*+\cdots+\xi_n^*\rb\right\}, \\
    &\mathcal{D}_{n} \subseteq \Z_{\ast}^{n} \setminus \mathcal{X}_{n}=\left\{(\xi_1,\dots,\xi_{n}) \in \Z_{\ast}^{n}:|\xi_1^*| \approx |\xi_2^*|\right\}.
\end{align*}

\subsection{Differentiation by parts}
Fix $d=\textnormal{deg}(P) \geq 2$. Let $n \geq 0$ and $(k_0,\dots,k_n) \in [d]^{n+1}$. Write 
\begin{align} \label{eq: nu n}
\begin{split}
    &\nu_0:=k_0, \\
    &\nu_n:=k_0+\sum_{i=1}^{n}(k_{i}-1), \quad n \geq 1.
\end{split}
\end{align}
Let $\mu_{k_0}:=\phi_{k_0}$, where $\phi_{k}$ is as defined in (\ref{eq: phi k definition}). For $1\leq j_1 \leq \nu_0$, let
\begin{align*}
    \mu^{j_1}_{k_0,k_1}(\xi_1,\dots,\xi_{k_{0}+k_{1}-1})
    :=i(\xi_{j_{1}}+\cdots+\xi_{j_{1}+k_{1}-1})\phi_{k_{1}}(\xi_{j_{1}},\dots,\xi_{{j_{1}}+k_{1}-1})
    \mathbf{X}^{k_{1}}_{j_{1}}\left(\frac{\mathbbm{1}_{\mathcal{N}_{k_0}}\mu_{k_0}}{\Omega_{k_{0}}}\right).
\end{align*}
For $n \geq 2$, inductively define
\begin{multline} \label{eq: mu j1 jn}
    \mu^{j_1,\dots,j_n}_{k_0,\dots,k_n}(\xi_1,\dots,\xi_{\nu_n})
    :=(-1)^{n-1}i 
    (\xi_{j_{n}}+\cdots+\xi_{j_{n}+k_{n}-1})\\
    \times \phi_{k_{n}}(\xi_{j_{n}},\dots,\xi_{{j_{n}}+k_{n}-1})
    \mathbf{X}^{k_{n}}_{j_{n}}\left(\frac{\mathbbm{1}_{\mathcal{N}_{\nu_{n-1}}}\mu^{j_1,\dots,j_{n-1}}_{k_0,\dots,k_{n-1}}}{\Omega_{\nu_{n-1}}}\right),
\end{multline}
where $1 \leq j_m \leq \nu_{m-1}$ for each $1\leq m \leq n$. Define
\begin{align*}
    \mu_{k_0,\dots,k_n}(\xi_1,\dots,\xi_{\nu_n})
    :=\sum_{j_1=1}^{\nu_0}\cdots\sum_{j_n=1}^{\nu_{n-1}}\mu^{j_1,\dots,j_n}_{k_0,\dots,k_n}(\xi_1,\dots,\xi_{\nu_n}).
\end{align*}

\begin{lem} \label{lem: normal form}
Let $u$ be a smooth solution to the equation (\ref{eq: guage dgbo}) with $P(x)=\sum_{k=1}^{d}c_k x^{k}$. For $(k_0,\dots,k_{n}) \in [d]^{n+1}$, let $c_{k_0,\dots,k_n}:=\prod_{j=0}^{n}c_{k_j}$. Then for any $N \geq 1$, we have
\begin{align*}
    u(t)-S(t)g 
    =&\sum_{n=0}^{N-1}\sum_{(k_0,\dots,k_{n}) \in [d]^{n+1}}c_{k_0,\dots,k_n} \left[B_{k_0,\dots,k_{n}}(u)(t) 
    -S(t)B_{k_0,\dots,k_{n}}(g)\right] \\
    &+\sum_{n=1}^{N}\sum_{(k_0,\dots,k_{n}) \in [d]^{n+1}}c_{k_0,\dots,k_n}\int_{0}^{t} S(t-t')\left[R^{1}_{k_0,\dots,k_{n}}(u)(t')\right] dt'\\
    &+\sum_{n=0}^{N}\sum_{(k_0,\dots,k_{n}) \in [d]^{n+1}}c_{k_0,\dots,k_n}\int_{0}^{t} S(t-t')\left[R^{2}_{k_0,\dots,k_{n}}(u)(t')+D_{k_0,\dots,k_{n}}(u)(t')\right] dt'\\
    &+\sum_{(k_0,\dots,k_{N}) \in [d]^{N+1}}c_{k_0,\dots,k_{N}}\int_{0}^{t}S(t-t')N_{k_0,\dots,k_{N}}(u)(t')dt,
\end{align*}
where for $n \geq 0$ and $(k_0,\dots,k_n) \in k \in [d]^{n+1}$, we define for $\xi \in \Z_{\ast}$,
\begin{align*}
    &\mathcal{F}\left[B_{k_0,\dots,k_n}(u)\right](\xi)
    := \sum_{\xi_1+\cdots+\xi_{\nu_{n}}=\xi}^{\ast}\frac{i\xi \mathbbm{1}_{\mathcal{N}_{\nu_{n}}}\mu_{k_0,\dots,k_n}(\xi_1,\dots,\xi_{\nu_{n}})}{\Omega_{\nu_n}(\xi_1,\dots,\xi_{\nu_{n}})}\hat{u}(\xi_1)\cdots \hat{u}(\xi_{\nu_{n}}), \\
    &\mathcal{F}\left[R^{1}_{k_0,\dots,k_n}(u)\right](\xi):= \sum_{\xi_1+\cdots+\xi_{\nu_{n}}=\xi}^{\ast}\xi\mathbbm{1}_{\mathcal{R}^{1}_{\nu_{n}}}\mu_{k_0,\dots,k_n}(\xi_1,\dots,\xi_{\nu_{n}})\hat{u}(\xi_1)\cdots \hat{u}(\xi_{\nu_{n}}), \\
    &\mathcal{F}\left[R^{2}_{k_0,\dots,k_n}(u) \right](\xi)
    := \sum_{\xi_1+\cdots+\xi_{\nu_{n}}=\xi}^{\ast}\xi\mathbbm{1}_{\mathcal{R}^{2}_{\nu_{n}}}\mu_{k_0,\dots,k_n}(\xi_1,\dots,\xi_{\nu_{n}})\hat{u}(\xi_1)\cdots \hat{u}(\xi_{\nu_{n}}),\\
    &\mathcal{F}\left[D_{k_0,\dots,k_n}(u)\right](\xi)
    := \sum_{\xi_1+\cdots+\xi_{\nu_{n}}=\xi}^{\ast}\xi \mathbbm{1}_{\mathcal{D}_{\nu_{n}}} \mu_{k_0,\dots,k_n}(\xi_1,\dots,\xi_{\nu_{n}})\hat{u}(\xi_1)\cdots \hat{u}(\xi_{\nu_{n}}), \\
    &\mathcal{F}\left[N_{k_0,\dots,k_n}(u)\right](\xi):= \sum_{\xi_1+\cdots+\xi_{\nu_{n}}=\xi}^{\ast}\xi \mathbbm{1}_{\mathcal{N}_{\nu_{n}}} \mu_{k_0,\dots,k_n}(\xi_1,\dots,\xi_{\nu_{n}})\hat{u}(\xi_1)\cdots \hat{u}(\xi_{\nu_{n}}).
\end{align*}
Zero Fourier modes are set to be zero.
\end{lem}
\begin{proof}
    For simplicity we only consider the case $P(x)=x^d$. In this case, we can write the equation (\ref{eq: guage dgbo}) on the Fourier side  
\begin{equation}
\begin{dcases}
    \partial_{t}\hat{u}_{\xi} =i \omega(\xi)\hat{u}_{\xi} +\frac{\xi}{i}\sum_{\xi_1+\cdots+\xi_d=\xi}\phi_{d}(\xi_1,\dots,\xi_d)\hat{u}_{\xi_1}\cdots\hat{u}_{\xi_d}, \\
    \hat{u}_{\xi}(0)=\hat{g}(\xi).
\end{dcases}
\end{equation}
    Let $v(t):=S(-t)u$. Then we have
    \begin{align}\label{eq: interaction variable equation}
        \partial_{t}\hat{v}_{\xi} 
        =\frac{\xi}{i}\sum_{\xi_1+\cdots+\xi_d=\xi}\phi_{d}(\xi_1,\dots,\xi_d) e^{-i\Omega_{d}(\xi_1,\dots,\xi_d)t}\hat{v}_{\xi_1}\cdots\hat{v}_{\xi_d}.
    \end{align}
    Using (\ref{eq: partition}) we can write the right-hand side of (\ref{eq: interaction variable equation}) as
    \begin{align} \label{eq: R+R+D+N}
        \mathcal{F}\left[S(-t)\left[R^1_d(u)+R^2_d(u)+D_d(u)+N_d(u)\right]\right].
    \end{align}
    By (\ref{eq: 0th resonance vanishes}), we have $R^1_d(u)=0$. Differentiating by part the $\mathcal{F}\left[S(-t)N_d(u,\dots,u)\right]$ portion of (\ref{eq: R+R+D+N}), we have
\begin{align*}
    \mathcal{F}\left[S(-t)N_d(u,\dots,u)\right]
    &=\frac{\xi}{i}\sum_{\xi_1+\cdots+\xi_d=\xi}\partial_{t} \left( \frac{i\mathbbm{1}_{\mathcal{N}_{d}}\phi_{d}(\xi_1,\dots,\xi_d)}{\Omega_{d}(\xi_1,\dots,\xi_d)}e^{-i\Omega_{d}(\xi_1,\dots,\xi_d)t}\right)\hat{v}_{\xi_1}\cdots\hat{v}_{\xi_d} \\
    &=\begin{multlined}[t][0.7\linewidth]\frac{\xi}{i}\sum_{\xi_1+\cdots+\xi_d=\xi}\partial_{t} \left(\frac{i\mathbbm{1}_{\mathcal{N}_{d}}\phi_{d}(\xi_1,\dots,\xi_d)}{\Omega_{d}(\xi_1,\dots,\xi_d)}e^{-i\Omega_{d}(\xi_1,\dots,\xi_d)t}\hat{v}_{\xi_1}\cdots\hat{v}_{\xi_d} \right) \\ -\frac{\xi}{i}\sum_{\xi_1+\cdots+\xi_d=\xi}\frac{i\mathbbm{1}_{\mathcal{N}_{d}}\phi_{d}(\xi_1,\dots,\xi_d)}{\Omega_{d}(\xi_1,\dots,\xi_d)}e^{-i\Omega_{d}(\xi_1,\dots,\xi_d)t}\partial_{t} \left( \hat{v}_{\xi_1}\cdots\hat{v}_{\xi_d} \right) 
    \end{multlined}\\
    &=\begin{multlined}[t][0.7\linewidth]\partial_{t} \left( \sum_{\xi_1+\cdots+\xi_d=\xi} \frac{\xi\mathbbm{1}_{\mathcal{N}_{d}}\phi_{d}(\xi_1,\dots,\xi_d)} {\Omega_{d}(\xi_1,\dots,\xi_d)}e^{-i\Omega_{d}(\xi_1,\dots,\xi_d)t}\hat{v}_{\xi_1}\cdots\hat{v}_{\xi_d} \right) \\ -\sum_{\xi_1+\cdots+\xi_d=\xi}\sum_{j=1}^{d}\frac{\xi\mathbbm{1}_{\mathcal{N}_{d}}\phi_{d}(\xi_1,\dots,\xi_d)} {\Omega_{d}(\xi_1,\dots,\xi_d)}e^{-i\Omega_{d}(\xi_1,\dots,\xi_d)t} \hat{v}_{\xi_1}\cdots\partial_{t}\hat{v}_{\xi_j}\cdots\hat{v}_{\xi_d}.
    \end{multlined}
\end{align*}
Substituting (\ref{eq: interaction variable equation}) into $\partial_t\hat{v}_{\xi_j}$ in the last line, using (\ref{eq: partition}), and then transforming back to the $u$ equation, we get the $N=1$ case. The cases $N \geq 2$ can be obtained by repeatedly applying the above process.
\end{proof}

Next, we perform differentiate by parts on the term $D_{k_0,\dots,k_n}(u)$ for each $(k_0,\dots,k_n)\in[d]^{n+1}$. For $n \geq 4$, define $\mathcal{D}^{1}_{n}$ and $\mathcal{D}^{2}_{n}$ by the sets satisfying $D_n=\mathcal{D}^{1}_{n} \sqcup \mathcal{D}^{2}_{n}$ and
\begin{align*}
    &\mathcal{D}^{1}_{n} \subseteq \left\{(\xi_1,\dots,\xi_{n}) \in \mathcal{D}_{n}:|\xi_3^*| \gtrsim |\xi_1+\cdots+\xi_n|\right\}, \\
    &\mathcal{D}^{2}_{n} \subseteq \left\{(\xi_1,\dots,\xi_{n}) \in \mathcal{D}_{n}:|\Omega_{n}(\xi_1,\dots,\xi_{n})|\gtrsim |\xi_1^*|^{\alpha}\rho_n(\xi_1,\dots,\xi_n)\right\}.
\end{align*}
This can be justified by Lemma \ref{lem: higher-order resonances}. Let $\nu_{n}$ be as in (\ref{eq: nu n}). For $(l_1,\dots,l_m)\in [d]^m$, let
\begin{align} \label{eq: nu nm}
    \begin{split}
        &\nu_{n,0}:=\nu_{n}, \\
        &\nu_{n,m}:=k_0+\sum_{i=1}^{n}(k_{i}-1)+\sum_{i=1}^{m}(l_{i}-1), \quad m\geq 1. 
    \end{split}
\end{align}
Let $\mathfrak{m}_{k_0,\dots,k_n}:=\mu_{k_0,\dots,k_n}$ and define
\begin{multline*}
    \mathfrak{m}_{k_0,\dots,k_n}^{l_{1}}(\xi_1,\dots,\xi_{\nu_{n,1}})\\
    :=i\sum_{1 \leq j \leq \nu_{n}}^{\ast}\Bigg[(\xi_j+\cdots+\xi_{j+l_1-1})\phi_{l_1}(\xi_j,\dots,\xi_{j+l_1-1})
    \mathbf{X}^{l_1}_{j}\left(\frac{\mathbbm{1}_{\mathcal{D}^{2}_{\nu_{n}}}\mathfrak{m}_{k_0,\dots,k_n}}{\Omega_{\nu_{n}}}\right)(\xi_1,\dots,\xi_{\nu_{n,1}})\Bigg].
\end{multline*}
Inductively define for each $m \geq 2$,
\begin{multline} \label{eq: mathfrak m definition}
    \mathfrak{m}_{k_0,\dots,k_n}^{l_{1},\dots,l_m}(\xi_1,\dots,\xi_{\nu_{n,m}})
    :=(-1)^{m-1}i\sum_{1 \leq j \leq \nu_{n,m-1}}^{\ast}\Bigg[(\xi_j+\cdots+\xi_{j+l_m-1})\\
    \times \phi_{l_m}(\xi_j,\dots,\xi_{j+l_m-1}) 
    \mathbf{X}^{l_m}_{j}\left(\frac{\mathbbm{1}_{\mathcal{D}^{2}_{\nu_{n,m-1}}}\mathfrak{m}_{k_0,\dots,k_n}^{l_{1},\dots,l_{m-1}}}{\Omega_{\nu_{n,m-1}}}\right)(\xi_1,\dots,\xi_{\nu_{n,m}})\Bigg].
\end{multline}
In the rest of this paper, we use the convention that
\begin{align*}
    \mathfrak{m}_{k_0,\dots,k_n}^{l_{1},\dots,l_0}:=\mathfrak{m}_{k_0,\dots,k_n}.
\end{align*}

\begin{lem}\label{lem: normal form 2}
Let $u$ be a smooth solution to the equation (\ref{eq: guage dgbo}) with $P(x)=\sum_{k=1}^{d}c_k x^{k}$. For $(l_{1},\dots,l_m) \in [d]^m$, let $c_{l_1,\dots,l_m}:=\prod_{j=1}^{m}c_{l_j}$. Then for any $n \geq 0$, $(k_0,\dots,k_n) \in [d]^{n+1}$, and $M \geq 1$, we have
\begin{align*}
    \int_{0}^{t} S(t-t')\left[D_{k_0,\dots,k_n}(u)(t')\right] dt
    &=\sum_{m=0}^{M-1}\sum_{(l_1,\dots,l_m)\in[d]^{m}}c_{l_1,\dots,l_m}\left[\mathfrak{B}_{k_0,\dots,k_n}^{l_{1},\dots,l_m}(u)-S(t)\mathfrak{B}_{k_0,\dots,k_n}^{l_{1},\dots,l_m}(g)\right]\\
    &\quad+\sum_{m=0}^{M}\sum_{(l_1,\dots,l_m)\in[d]^{m}}c_{l_1,\dots,l_m}\int_{0}^{t} S(t-t')[\mathfrak{R}_{k_0,\dots,k_n}^{l_{1},\dots,l_m}(u)(t')]dt'\\
    &\quad+\sum_{(l_1,\dots,l_M)\in[d]^{M}}c_{l_1,\dots,l_M}\int_{0}^{t} S(t-t')[\mathfrak{N}_{k_0,\dots,k_n}^{l_{1},\dots,l_{M}}(u)(t')]dt',
\end{align*}
where for $m \geq 1$ and $\xi \in \Z_{\ast}$ we define
\begin{align*}
    &
    \mathcal{F}\left[\mathfrak{B}_{k_0,\dots,k_n}^{l_{1},\dots,l_m}(u)\right](\xi)
    := \sum_{\xi_1+\cdots+\xi_{\nu_{n,m}}=\xi}^{\ast}\frac{i\xi \mathbbm{1}_{\mathcal{D}^{2}_{\nu_{n,m}}}\mathfrak{m}_{k_0,\dots,k_n}^{l_{1},\dots,l_m}(\xi_1,\dots,\xi_{\nu_{n,m}})}{\Omega_{\nu_{n,m}}(\xi_1,\dots,\xi_{\nu_{n,m}})}\hat{u}(\xi_1)\cdots \hat{u}(\xi_{\nu_{n,m}}),
    \\
    &
    \mathcal{F}\left[\mathfrak{R}_{k_0,\dots,k_n}^{l_{1},\dots,l_m}(u)\right](\xi)
    := \sum_{\xi_1+\cdots+\xi_{\nu_{n,m}}=\xi}^{\ast}\xi \mathbbm{1}_{\mathcal{D}^{1}_{\nu_{n,m}}} \mathfrak{m}_{k_0,\dots,k_n}^{l_{1},\dots,l_m}(\xi_1,\dots,\xi_{\nu_{n,m}})
    \hat{u}(\xi_1)\cdots \hat{u}(\xi_{\nu_{n,m}}), \\
    &
    \mathcal{F}\left[\mathfrak{N}_{k_0,\dots,k_n}^{l_{1},\dots,l_m}(u)\right](\xi)
    := \sum_{\xi_1+\cdots+\xi_{\nu_{n,m}}=\xi}^{\ast}\xi \mathbbm{1}_{\mathcal{D}^{2}_{\nu_{n,m}}} \mathfrak{m}_{k_0,\dots,k_n}^{l_{1},\dots,l_m}(\xi_1,\dots,\xi_{\nu_{n,m}})
    \hat{u}(\xi_1)\cdots \hat{u}(\xi_{\nu_{n,m}}),
\end{align*}
and zero Fourier modes are set to be zero.
\end{lem}
Proof of Lemma \ref{lem: normal form 2} is similar to that of Lemma \ref{lem: normal form}, and is therefore omitted.

\section{Preliminary lemmas} \label{section: Preliminary lemmas}

In this section, we present some preliminary lemmas we need for the proofs of the results in the next section.
\begin{lem}\label{lem: max}
Let $(x_m)_{1\leq m \leq n}$ be a sequence of real numbers. Then we have
    \begin{align*}
        \max\left(|x_1|,|x_1+x_2|,\dots,|x_1+x_2+\cdots+x_n| \right) \approx \max\left(|x_1|,|x_2|,\dots,|x_n| \right).
    \end{align*}
\end{lem}
\begin{proof}
    The $\lesssim$ direction is trivial. To show the $\gtrsim$ direction. let $M:=\max\left(|x_1|,|x_2|,\dots,|x_n| \right)$ and $L:= \max\left(|x_1|,|x_1+x_2|,\dots,|x_1+x_2+\cdots+x_n| \right)$. Suppose that $M=|x_m|$. If $|x_m| \gg |x_1+\cdots+x_{m-1}|$, then $M=|x_m|\approx|x_1+\cdots+x_m| \lesssim L$. If $|x_m| \approx |x_1+\cdots+x_{m-1}|$, then $M=|x_m| \approx |x_1+\cdots+x_{m-1}| \lesssim L$.
\end{proof}
The following two lemmas capture the cancellation property of the $R^{1}_{k_0,\dots,k_n}$ terms:
\begin{lem}\label{lem: zero sum}
Let $(a_n)_{1\leq n \leq N}$ be a finite sequence of real numbers. If $a_1+\cdots+a_N=0$, then 
    \begin{align*}
     \sum_{\sigma \in S_{N}}\frac{\mathbbm{1}_{\{a_{\sigma(i)} \neq 0 \textnormal{ and } a_{\sigma(1)}+\cdots+a_{\sigma(i)}\neq 0 \textnormal{ for all } 1 \leq i \leq N-1\}}}{a_{\sigma(1)}a_{\sigma(2)}\cdots a_{\sigma(N-1)}} =0.
    \end{align*}
\end{lem}
\begin{proof}
    Since
    \begin{align*}
        &\sum_{\sigma \in S_{N}}\frac{\mathbbm{1}_{\{a_{\sigma(i)} \neq 0 \textnormal{ and } a_{\sigma(1)}+\cdots+a_{\sigma(i)}\neq 0 \textnormal{ for all } 1 \leq i \leq N-1\}}}{a_{\sigma(1)}a_{\sigma(2)}\cdots a_{\sigma(N-1)}} \\
        &=\frac{\mathbbm{1}_{\{a_i \neq 0 \textnormal{ for all } 1 \leq i \leq N\}}}{a_1\cdots a_N}\sum_{\sigma \in S_{N}}\mathbbm{1}_{\{ a_{\sigma(1)}+\cdots+a_{\sigma(i)}\neq 0 \textnormal{ for all } 1 \leq i \leq N-1\}}a_{\sigma(N)},
    \end{align*}
    it suffices to show that
    \begin{align*}
    \sum_{\sigma \in S_N} \mathbbm{1}_{\{ a_{\sigma(1)}+\cdots+a_{\sigma(i)}= 0 \textnormal{ for some } 1 \leq i \leq N-1\}}a_{\sigma(N)}=0. 
    \end{align*}
    Note that the left-hand side is equal to
    \begin{align*}
       \sum_{j=1}^{N}a_j\sum_{\substack{\sigma \in S_N \\ \sigma(N)=j}} \mathbbm{1}_{\{ a_{\sigma(1)}+\cdots+a_{\sigma(i)}= 0 \textnormal{ for some } 1 \leq i \leq N-1\}}.
    \end{align*}
    Let 
    \begin{align*}
        \mathcal{Z}:=\left\{Z \subseteq \{1,\dots,N\}: \sum_{n \in Z}a_n=0\right\}.
    \end{align*}
    For each $1 \leq j \leq N$, let $\mathcal{Z}_{j}:=\left\{Z \in \mathcal{Z}: j \notin Z\right\}$. We say that $Z \in \mathcal{Z}_{j}$ is \textit{maximal in} $\mathcal{Z}_{j}$ if $Z$ is not a proper subset of another element of $\mathcal{Z}_{j}$. Let $\mathcal{M}_{j}:=\{Z \in \mathcal{Z}_{j}: Z \textnormal{ is maximal in }\mathcal{Z}_{j}\}$. Then
    \begin{align*}
        \sum_{\substack{\sigma \in S_N \\ \sigma(N)=j}} \mathbbm{1}_{\{ a_{\sigma(1)}+\cdots+a_{\sigma(i)}=0 \textnormal{ for some } 1 \leq i \leq N-1\}}=\sum_{Z \in \mathcal{M}_{j}} \#Z! (N-\#Z-1)!.
    \end{align*}
    Indeed, the left-hand side counts the number of sequences $(b_{n})_{1\leq n \leq N-1}$ that rearrange the sequence $$a_1,\dots,a_{j-1},a_{j+1},\dots,a_{N}$$ such that there exists $1\leq i \leq N-1$ with $b_1+\cdots+b_i=0$. In the right-hand side, we generate such $(b_{n})_{1\leq n \leq N-1}$ in the following way: pick $Z \in \mathcal{M}_{j}$, and then permute the set $Z$ to get $b_1,\dots,b_{\#Z}$. Next, permute the set $(\{1,\dots,N\}\setminus\{j\})\setminus Z$ to get $b_{\#Z+1},\dots,b_{N-1}$. 
    
    Now, let $\mathcal{M}:=\cup_{j}\mathcal{M}_{j}$. We claim that if $Z \in \mathcal{M}$, then we have $Z \in \mathcal{M}_{j}$ if and only if $j \notin Z$. The ``only if" part is trivial. To show the ``if" direction, let $j \notin Z$, and suppose for a contradiction that $Z$ is not maximal in $\mathcal{Z}_{j}$. Then there exists $\t Z \in \mathcal{Z}_{j}$ with $Z \subsetneq \t Z$. Take $k$ such that $Z \in \mathcal{M}_{k}$. By the maximality of $Z$ in $\mathcal{Z}_{k}$, we have $k \in \t Z\setminus Z$. Using this fact with $a_1+\cdots+a_N=0$, we can see that $\t{Z}^{c} \cup Z =(\t Z\setminus Z)^{c}\in \mathcal{Z}_{k}$. This implies, by maximality of $Z$ in $\mathcal{Z}_{k}$, that $\t{Z}^{c} \subseteq Z$. On the other hand, since $Z \subseteq \t Z$, it is trivial that $\t Z^c \subseteq Z^c$. Therefore we have $\t Z^c \subseteq Z \cap Z^c = \varnothing$, hence $\t Z=\{1,\dots,N\}$. This contradicts the fact that $j \notin \t Z$.

    Next, observe that
    \begin{align*}
     \sum_{j=1}^{N}a_j\sum_{Z \in \mathcal{M}_{j}} \# Z!(N-\# Z-1)!
     &=\sum_{j=1}^{N}\sum_{Z \in \mathcal{M}}a_j\mathbbm{1}_{\{Z \in \mathcal{M}_{j}\}} \# Z!(N-\# Z-1)!\\
     &=\sum_{Z \in \mathcal{M}}\# Z!(N-\# Z-1)!\sum_{j:Z \in \mathcal{M}_{j}}a_j.
    \end{align*}
    However, the claim above implies
    \begin{align*}
        \sum_{j:Z \in \mathcal{M}_{j}}a_j=\sum_{j \notin Z}a_j=-\sum_{j \in Z}a_j=0.
    \end{align*}
    This completes the proof.
\end{proof}
Before proceeding to the following lemma, it is recommended that a reader revisits subsection \ref{subsection: Basic notations} to remind some basic notations.
\begin{lem} \label{lem: cancellation}
For $N\geq 1$, consider the multiset $\mathcal{K}:=\{k_1^{n_1},k_2^{n_2},\dots,k_{m}^{n_m}\}$ with $\sum_{j=1}^{m}n_j=N$. For $M:=\sum_{j=1}^{m}n_jk_j$, let $(\xi_i)_{1\leq i \leq M}$ be a sequence of real numbers satisfying $\sum_{i=1}^{M}\xi_i=0$. For each $\pi \in \textnormal{Perm}(\mathcal{K})$, let
\begin{align*}
    &x_{\pi,1}:=\sum_{j=1}^{\pi(1)}\xi_{j}, \\
    &x_{\pi,i}:=\sum_{j=\pi(1)+\cdots+\pi(i-1)+1}^{\pi(1)+\cdots+\pi(i)}\xi_{j}, \quad 2\leq i \leq N.
\end{align*}
Then we have
\begin{align*}
    \sum_{\sigma \in S_M}\sigma \cdot\left[\sum_{\pi \in \textnormal{Perm}(\mathcal{K})}\frac{\mathbbm{1}_{\{x_{\pi,i}\neq 0 \textnormal{ and } x_{\pi,1}+x_{\pi,2}+\cdots+x_{\pi,i} \neq 0 \textnormal{ for all } 1 \leq i \leq N-1\}}}{x_{\pi,1}(x_{\pi,1}+x_{\pi,2})\cdots(x_{\pi,1}+x_{\pi,2}+\cdots+x_{\pi,N-1})}\right] 
    =0,
\end{align*}
where $\sigma \in S_M$ permutes the indices of $(\xi_i)_{1\leq i \leq M}$.
\end{lem}
 
\begin{proof}
Define $e \in \textnormal{Perm}(\mathcal{K})$ by
\begin{align*}
    e(n_1+\cdots+n_{p-1}+q)=k_p
    \textnormal{ for each } 1\leq p \leq m \textnormal{ and } 1\leq q \leq n_p,
\end{align*}
and let $u_{i}:=x_{e,i}$ for $1 \leq i \leq N$. Define
\begin{multline*}
    G:=\big\{\theta \in S_N:
    \theta(n_1+\cdots+n_{p-1}+q)<\theta(n_1+\cdots+n_{p-1}+r) \\
    \textnormal{ for each } 1\leq p \leq m \textnormal{ and } 1\leq q<r \leq n_{p}\big\}.
\end{multline*}
By rearranging the indices, we have
\begin{multline*}
    \sum_{\sigma \in S_M}\sigma \cdot\left[\sum_{\pi \in \textnormal{Perm}(\mathcal{K})}\frac{\mathbbm{1}_{\{x_{\pi,i}\neq 0 \textnormal{ and } x_{\pi,1}+x_{\pi,2}+\cdots+x_{\pi,i} \neq 0 \textnormal{ for all } 1 \leq i \leq N-1\}}}{x_{\pi,1}(x_{\pi,1}+x_{\pi,2})\cdots(x_{\pi,1}+x_{\pi,2}+\cdots+x_{\pi,N-1})}\right] \\
    =\sum_{\sigma \in S_M}\sigma \cdot\left[\sum_{\theta \in G}\frac{\mathbbm{1}_{\{u_{\theta(i)}\neq 0 \textnormal{ and } u_{\theta(1)}+u_{\theta(2)}+\cdots+u_{\theta(i)} \neq 0 \textnormal{ for all } 1 \leq i \leq N-1\}}}{u_{\theta(1)}(u_{\theta(1)}+u_{\theta(2)})\cdots(u_{\theta(1)}+u_{\theta(2)}+\cdots+u_{\theta({N-1})})}\right].
\end{multline*}
For $\theta \in G$, write $\mathbbm{1}_{\theta}:=\mathbbm{1}_{\{u_{\theta(i)}\neq 0 \textnormal{ and } u_{\theta(1)}+u_{\theta(2)}+\cdots+u_{\theta(i)} \neq 0 \textnormal{ for all } 1 \leq i \leq N-1\}}$. Then
\begin{align*}
    &\sum_{\theta\in G}\frac{\mathbbm{1}_{\theta}}{u_{\theta(1)}(u_{\theta(1)}+u_{\theta(2)})\cdots(u_{\theta(1)}+u_{\theta(2)}+\cdots+u_{\theta({N-1})})} \\
    &=\sum_{\theta\in G}\mathbbm{1}_{\theta}\int_{0}^{\infty}e^{-u_{\theta(1)} x_1}dx_1\int_{0}^{\infty}e^{-(u_{\theta(1)}+u_{\theta(2)}) x_2}dx_2 \int_{0}^{\infty}\cdots e^{-(u_{\theta(1)}+\cdots+u_{\theta({N-1})}) x_{N-1}}dx_{N-1}\\
    &=\sum_{\theta\in G}\mathbbm{1}_{\theta}\int_{[0,\infty)^{m-1}}e^{-u_{\theta(1)} (x_1+\cdots+x_{N-1})}e^{-u_{\theta(2)} (x_2+\cdots+x_{N-1})} \cdots e^{-u_{\theta({N-1})} x_{N-1}} dx_1 dx_2 \dots dx_{N-1}\\
    &=\int_{0<y_{N-1}<\dots<y_1<\infty} \sum_{\theta\in G}\mathbbm{1}_{\theta}e^{-u_{\theta(1)} y_1-u_{\theta(2)} y_2- \cdots-u_{\theta({N-1)}} y_{N-1}} dy_1 dy_2 \dots dy_{N-1},
\end{align*}
where in the last equality we changed the variables by setting $y_j=x_j+\cdots+x_{N-1}$ for each $1 \leq j\leq N-1$. 

For each $1 \leq p \leq m$, $1 \leq q \leq n_p$, and $1\leq r\leq k_p$, write $n_{p,q,r}:=\sum_{i=1}^{p-1}n_{i}k_{i}+(q-1)k_p+r$. Define
\begin{multline*}
    H:=\big\{\sigma \in S_M: \textnormal{for each } 1\leq p \leq m \textnormal{ and } 1 \leq q \leq n_p, \\
    \textnormal{ there exists } 1 \leq q' \leq n_p \textnormal{ such that } \sigma(n_{p,q,r})=n_{p,q',r} \textnormal{ for all } 1\leq r \leq k_p\big\}.
\end{multline*}
Let
\begin{align*}
    f(y_1,\dots,y_{N-1}):= \sum_{\sigma \in S_M}\sigma \cdot \left[ \sum_{\theta \in G}\mathbbm{1}_{\theta}e^{-\sum_{i=1}^{N-1}u_{\theta(i)} y_i}\right].
\end{align*}
    Then we have
    \begin{align*}
        f(y_1,\dots,y_{N-1})
        &=\frac{1}{|H|}\sum_{\sigma \in S_M}\sigma \cdot \left[\sum_{\tau \in H} \tau \cdot \sum_{\theta \in G}\mathbbm{1}_{\theta}e^{-\sum_{i=1}^{N-1}u_{\theta(i)} y_i}\right] \\
        &=\frac{1}{|H|}\sum_{\sigma \in S_M}\sigma \cdot \left[\sum_{\theta \in S_N}\mathbbm{1}_{\theta}e^{-\sum_{i=1}^{N-1}u_{\theta(i)} y_i}\right].
    \end{align*}
    Thus the function $f$ is symmetric. Using the symmetry of $f$, we have
    \begin{align*}
        &\int_{0<y_{N-1}<\dots<y_1<\infty}f(y_1,\dots,y_{N-1})dy_1 dy_2 \dots dy_{N-1} \\
        &\quad=\frac{1}{(N-1)!}\int_{[0,\infty)^{N-1}}f(y_1,\dots,y_{N-1})dy_1 dy_2 \dots dy_{N-1} \\
        &\quad=\frac{1}{|H|(N-1)!} \sum_{\sigma \in S_M}\sigma \cdot \left[\sum_{\theta \in S_N}\frac{\mathbbm{1}_{\theta}}{u_{\theta(1)}u_{\theta(2)}\cdots u_{\theta({N-1})}}\right].
    \end{align*}
However, the last line vanishes by Lemma \ref{lem: zero sum}, and the desired result follows.
\end{proof}

\section{Pointwise estimates} \label{section: Pointwise estimates}

In this section, we establish some pointwise estimates of the symbols associated with the Fourier multiplier operators defined in Section \ref{section: Normal form reduction}. We start with the following fundamental estimates:
\begin{lem} \label{lem: mu pointwise bound} 
Fix $1<\alpha<2$. Let $\nu_n$ and $\nu_{n,m}$ be as in (\ref{eq: nu n}) and (\ref{eq: nu nm}). \begin{enumerate}[label=(\alph*)]
    \item For $n \geq 0$, we have
        \begin{align*}
            \left| \mu_{k_0,\dots,k_n}(\xi_1,\dots,\xi_{\nu_n})\right| \lesssim |\xi_1+\cdots+\xi_{\nu_n}|^{-n(\alpha-1)}.
        \end{align*}
    \item For $m,n \geq 0$, we have
        \begin{align*}
            \left| \mathfrak{m}_{k_0,\dots,k_n}^{l_1,\dots,l_m}(\xi_1,\dots,\xi_{\nu_{n,m}})\right| \lesssim |\xi_1+\cdots+\xi_{\nu_{n,m}}|^{-(m+n)(\alpha-1)}.
        \end{align*}
    \item If $k_0=\cdots=k_n=2$ and $l_1=\cdots=l_m=2$, then 
        \begin{align*}
            \left| \mathfrak{m}_{k_0,\dots,k_n}^{l_1,\dots,l_m}(\xi_1,\dots,\xi_{n+m+2})\right| \lesssim \sum_{j=1}^{n+m+1}\frac{\prod_{i=3}^{n+m+2}\max\left(|\xi_i^*|,|\xi|\right)^{-(\alpha-1)}}{\mathbf{X}^{j}_{2}\left(\rho_{n+m+1}\right)(\xi_1,\dots,\xi_{n+m+2})}.
        \end{align*}
\end{enumerate}
\end{lem}

\begin{proof}
Let us first consider (a). We look at the expression (\ref{eq: mu j1 jn}), and for simplicity assume that $j_n=1$. Write $\t \xi_{k_n}$ for $\xi_1+\cdots+\xi_{k_n}$. Recall that
\begin{align*}
    \left|\frac{\mathbbm{1}_{\mathcal{N}_{\nu_{n-1}}}}{\Omega_{\nu_{n-1}}(\t \xi_{k_n},\xi_{k_n+1},\dots,\xi_{\nu_n})}\right| 
    \lesssim \big|\max\big( \t \xi_{k_n},\xi_{k_n+1},\dots,\xi_{\nu_n}\big)\big|^{-\alpha}.
\end{align*}
Hence
\begin{align*}
    \left|\mu^{j_1,\dots,j_n}_{k_0,\dots,k_n}(\xi_1,\dots,\xi_{\nu_n})\right|
    &\lesssim \left|\frac{\t \xi_{k_n}\mathbbm{1}_{\mathcal{N}_{\nu_{n-1}}}\mu^{j_1,\dots,j_{n-1}}_{k_0,\dots,k_{n-1}}(\t \xi_{k_n},\xi_{k_n+1},\dots,\xi_{\nu_n})}{\Omega_{\nu_{n-1}}(\t \xi_{k_n},\xi_{k_n+1},\dots,\xi_{\nu_n})}\right| \\
    &\lesssim |\t \xi_{k_n}|\big|\max\big( \t \xi_{k_n},\xi_{k_n+1},\dots,\xi_{\nu_n}\big)\big|^{-\alpha}\left|\mu^{j_1,\dots,j_{n-1}}_{k_0,\dots,k_{n-1}}(\t \xi_{k_n},\xi_{k_n+1},\dots,\xi_{\nu_n})\right| \\
    &\lesssim \big(|\t \xi_{k_n}|+|\xi_{k_n+1}|+\cdots+|\xi_{\nu_n}|\big)^{-(\alpha-1)}\left|\mu^{j_1,\dots,j_{n-1}}_{k_0,\dots,k_{n-1}}(\t \xi_{k_n},\xi_{k_n+1},\dots,\xi_{\nu_n})\right| \\
    &\lesssim |\xi|^{-(\alpha-1)}\left|\mu^{j_1,\dots,j_{n-1}}_{k_0,\dots,k_{n-1}}(\t \xi_{k_n},\xi_{k_n+1},\dots,\xi_{\nu_n})\right|.
\end{align*}
Now the desired inequality follows by induction. The inequality in (b) can be similarly deduced. Finally, (c) follows from a similar induction with the additional observation that
\begin{align*}
    \frac{\mathbbm{1}_{\mathcal{D}_{n+m+1}^2}}{|\Omega_{n+m+1}(\xi_1,\dots,\xi_{n+m+1})|} \lesssim \frac{1}{|\xi_1^*|^{\alpha}\rho_{n+m+1}(\xi_1,\dots,\xi_{n+m+1})}.
\end{align*}
\end{proof}

If $\textnormal{deg}(P) \leq 3$, then we can upgrade the estimates in Lemma \ref{lem: mu pointwise bound} on the set $\mathcal{R}^2_{\nu_n}$, as the following lemma shows.

\begin{lem} \label{lem: R2 pointwise bound} Fix $1<\alpha<2$.
 \begin{enumerate}[label=(\alph*)]
    \item We have
         \begin{align*}
            \big| \mathbbm{1}_{\mathcal{R}^2_{3}}\mu_{2,2}(\xi_1,\xi_2,\xi_3)\big| \lesssim |\xi_1^*|^{-\alpha},
        \end{align*}
    \item Let $n\geq 2$. Write $\xi:=\xi_1+\cdots+\xi_{n+2}$. Then we have
        \begin{align*}
            \big| \mathbbm{1}_{\mathcal{R}^{2}_{n+2}}\mu_{\underbrace{\scriptstyle2,\dots,2}_{n+1}}(\xi_1,\dots,\xi_{n+2})\big| 
            \lesssim
            \frac{|\xi_1^*|^{-n(\alpha-1)-1}|\xi_3^*|}{\min\left(\lb\xi-\xi_1^*\rb,\lb\xi-\xi_1^*-\xi_2^*\rb,\lb\xi-\xi_1^*-\xi_3^*\rb,\dots,\lb\xi-\xi_1^*-\xi_{n+2}^*\rb\right)}.
            \end{align*}
    \item For $n\geq 1$, let $(k_0,\dots,k_n) \in \{2,3\}^{n+1}$. If $\nu_n=4$ or $\nu_n \geq 5$ with $|\xi_4^*| \gg |\xi_5^*|$, then we have
        \begin{align*}
            \big| \mathbbm{1}_{\mathcal{R}^2_{\nu_n}}\mu_{k_0,\dots,k_n}(\xi_1,\dots,\xi_{\nu_n})\big| \lesssim |\xi_1^*|^{-n(\alpha-1)-1}|\xi_3^*|.
        \end{align*}
\end{enumerate}
\end{lem}
\begin{proof}[Proof of (a)]
A direct computation gives
\begin{align*}
   \mu_{2,2}(\xi_1,\xi_2,\xi_3)=\frac{i (\xi_1+\xi_2)}{4\Omega_{2}(\xi_1+\xi_2,\xi_3)}+\frac{i (\xi_2+\xi_3)}{4\Omega_{2}(\xi_1,\xi_2+\xi_3)}.
\end{align*}
This with Lemma \ref{lem: phase function asymptotics 1} proves the desired inequality. 
\end{proof}

\begin{proof}[Proof of (b)] 
Assume without loss of generality that $(\xi_1,\dots,\xi_{n+2}) \in \mathcal{R}_{n+2}^2$, $\xi_1^*=\xi_1$ and $|\xi_1| \gg |\xi_2^*|$. 

\noindent\textbf{Case 1: $(j_1,\dots,j_{n}) \neq (1,\dots,1)$.}


In this case, by induction we have
\begin{align*}
\left| \mu^{j_1,\dots,j_{n}}_{2,\dots,2}\right| 
&\lesssim \frac{|\xi_1|^{-n(\alpha-1)-1}|\xi^*_2|}{\rho_{n+1}(\xi_1,\dots,\xi_{j_n}+\xi_{j_{n}+1},\dots,\xi_{n+2})} \\
&\lesssim |\xi_1^*|^{-n(\alpha-1)-1}|\xi_3^*|\bigg(\lb\xi-\xi_1^*\rb^{-1}
+\sum_{j=2}^{n+2}\lb\xi-\xi_1^*-\xi_j^*\rb^{-1}\bigg).
\end{align*}
In the second inequality, we used that $|\xi_2^*| \approx |\xi_3^*|$.

\noindent\textbf{Case 2: $(j_1,\dots,j_{n}) = (1,\dots,1)$.}

We now consider $\mu^{1,\dots,1}_{2,\dots,2}$. For $0\leq j \leq n-1$, let $\t \xi_{n-j+1}:=\xi_1+\cdots+\xi_{n-j+1}$ and
\begin{align*}
    &\phi_{j}:=\phi_{2}(\t \xi_{n-j+1},\xi_{n-j+2}), \\
    &\mathbbm{1}_{j}:=\mathbbm{1}_{\mathcal{N}_{j+2}}(\t \xi_{n-j+1},\xi_{n-j+2},\dots,\xi_{n+2}), \\
    &\Omega_{j}:=\Omega_{j+2}(\t \xi_{n-j+1},\xi_{n-j+2},\dots,\xi_{n+2}).
\end{align*} 
Then
\begin{align} \label{eq: mu 1111}
    \mu^{1,\dots,1}_{2,\dots,2}(\xi_1,\dots,\xi_{n+2})
    =\phi_{2}(\xi_1,\xi_{2})
    \prod_{j=0}^{n-1}\frac{\t \xi_{n-j+1}\phi_{j}\mathbbm{1}_{j}}{\Omega_{j}}.
\end{align}
Also,
\begin{align*}
    \frac{\mathbbm{1}_{j}}{|\Omega_j|} \lesssim |\xi_1^*|^{-\alpha}\rho_{j+2}(\t \xi_{n-j+1},\xi_{n-j+2},\dots,\xi_{n+2})^{-1} = |\xi_1^*|^{-\alpha}\lb \xi_{n-j+2}+\cdots+\xi_{n+2} \rb^{-1}.
\end{align*}
Therefore we have
\begin{align*}
    (\ref{eq: mu 1111})
    &\lesssim |\xi_1|^{-n(\alpha-1)}\prod_{j=0}^{n-1}\frac{1}{\rho_{j+2}(\t \xi_{n-j+1},\xi_{n-j+2},\dots,\xi_{n+2})} \\
    &\lesssim \frac{|\xi_1|^{-n(\alpha-1)}}{\max_{0 \leq j \leq n-1}\lb\xi_{n-j+2}+\cdots+\xi_{n+2}\rb} \\
    &\lesssim |\xi_1^*|^{-n(\alpha-1)}\lb \xi_3^*\rb^{-1} \\
    &\lesssim |\xi_1^*|^{-n(\alpha-1)-1}|\xi_3^*|\lb \xi-\xi_1^*\rb^{-1},
\end{align*}
where we used Lemma \ref{lem: max} in the third inequality, and that $(\xi_1,\dots,\xi_{n+2}) \in \mathcal{R}_{n+2}^2$ in the last inequality.
\end{proof}

\begin{proof}[Proof of (c)] 
For simplicity we only consider the case $n \geq 1$ and $(k_0,\dots,k_n)=(3,\dots,3)$. Assume without loss of generality that $\xi_1^*=\xi_1$ and $|\xi_1| \gg |\xi_2^*|$. If $(j_1,\dots,j_{n}) \neq (1,\dots,1)$, then by induction we have
\begin{align}\label{eq: R2 pointwise 1}
    \left| \mu^{j_1,\dots,j_{n}}_{3,\dots,3}\right| \lesssim |\xi_1|^{-n(\alpha-1)-1}|\xi^*_2|.
\end{align}
We now consider $\mu^{1,\dots,1}_{3,\dots,3}$. Proceeding as in the proof of (b), we have
\begin{align} \label{eq: 333}
    \big|\mathbbm{1}_{\mathcal{R}^2_{2n+3}}\mu^{1,\dots,1}_{3,\dots,3}(\xi_1,\dots,\xi_{2n+3})\big| \lesssim \frac{|\xi_1|^{-n(\alpha-1)}\mathbbm{1}_{\mathcal{R}^2_{2n+3}}}{\prod_{j=1}^{n}\lb\xi_{2j+2}+\cdots+\xi_{2n+3}\rb}.
\end{align}
Assume that $|\xi_4^*|\gg|\xi_5^*|$. Also, since we are on $\mathcal{R}^2_{2n+3}$, we may assume that $|\xi_1^*|\gg|\xi_2^*|$ and $|\xi_3^*|^{\alpha}|\xi_4^*|\gtrsim \lb\xi-\xi_1^*\rb|\xi_1^*|^{\alpha}$, where $\xi:=\xi_1+\cdots+\xi_{\nu_n}$. Notice that we have $|\xi_2^*+\xi_3^*| \gtrsim |\xi_4^*|$. Indeed, if $|\xi_2^*+\xi_3^*| \ll |\xi_4^*|$, then we have $\lb\xi-\xi_1^*\rb \approx |\xi_4^*|$. This implies $|\xi_1^*|^{\alpha}|\xi_4^*|\approx|\xi_1^*|^{\alpha}\lb\xi-\xi_1^*\rb \lesssim |\xi_3^*|^{\alpha}|\xi_4^*|$, hence $|\xi_1^*|\approx|\xi_3^*|$, a contradiction. 

Now we can see that
\begin{align*}
    \max_{1 \leq j \leq n}|\xi_{2j+2}+\xi_{2j+3}| \gtrsim |\xi_4^*|.
\end{align*}
Using this, (\ref{eq: 333}), and Lemma \ref{lem: max}, we have
\begin{align*}
    \big| \mathbbm{1}_{\mathcal{R}^2_{2n+3}}\mu_{3,\dots,3}(\xi_1,\dots,\xi_{2n+3})\big| \lesssim \mathbbm{1}_{\mathcal{R}^2_{2n+3}}\left(|\xi_1^*|^{-1}|\xi_2^*|+|\xi_4^*|^{-1}\right)|\xi_1^*|^{-n(\alpha-1)}\lesssim |\xi_1^*|^{-n(\alpha-1)-1}|\xi_2^*|
\end{align*}
as desired.
\end{proof}

Next we estimate the symbol associated with the near-resonant term $R^{1}_{2,2}(u)$. For this we need some cancellation properties induced by symmetrization.
\begin{lem} \label{lem: R22 pointwise estimate}
We have 
   \begin{align}\label{eq: R22 estimate}
    \hat{R^{1}_{2,2}(u)}(\xi)=\sum_{\xi_1+\xi_2+\xi_3=\xi}m(\xi_1,\xi_2,\xi_3)\hat{u}_{\xi_1}\hat{u}_{\xi_2}\hat{u}_{\xi_3}
\end{align}
for some $m:\Z_{\ast}^{3} \to \C$ such that $\supp(m) \subseteq \mathcal{R}^{1}_{3}$ and $|m(\xi_1,\xi_2,\xi_3)|\lesssim |\xi_1^*|^{1-\alpha}$.
\end{lem}
\begin{proof}
Let $\mathcal{A}:=\{(\eta ,\xi ) \in \Z^2: \eta >0, \, \xi  \neq 0, \, |\xi | \neq \eta  \}$ and
\begin{align*}
    g(\eta ,\xi ):=\mathbbm{1}_{\mathcal{A}}(\eta ,\xi )\xi \left[\frac{\xi -\eta }{\Omega_{2}(\eta ,-\xi )} + \frac{\xi +\eta }{\Omega_{2}(-\eta ,-\xi )}\right].
\end{align*}
Direct computation gives
\begin{align*}
    \hat{R^{1}_{2,2}(u)}(\xi)=\frac{i \xi^2}{4(2^{\alpha}-1) \omega(\xi)}|\hat{u}_{\xi} |^2 \hat{u}_{\xi}+\frac{i}{2}\sum_{\eta >0} g(\eta ,\xi )\hat{u}_{\xi}\hat{u}_{\eta}\hat{u}_{-\eta}.
\end{align*}
It suffices to show that $|g(\eta,\xi)|\lesssim \max(|\xi|,|\eta|)^{1-\alpha}$. Note that
\begin{align*}
    g(\eta ,\xi )
    =\mathbbm{1}_{\mathcal{A}}(\eta ,\xi ) \xi\frac{2(|\eta |^{\alpha+2}-|\xi |^{\alpha+2})-(\eta ^2-\xi ^2)(|\eta +\xi |^{\alpha}+|\eta -\xi |^{\alpha})}{\Omega_{2}(\eta ,-\xi )\Omega_{2}(-\eta ,-\xi )}.
\end{align*}

\noindent\textbf{Case 1: $\eta  \gg |\xi |$.} 

Let $t=\xi/\eta$. Then
\begin{align*}
    |g(\eta ,\xi )| & \approx |\eta |^{-2\alpha}|\xi |^{-1} |2(|\eta |^{\alpha+2}-|\xi |^{\alpha+2})-(\eta ^2-\xi ^2)(|\eta +\xi |^{\alpha}+|\eta -\xi |^{\alpha})| \\
    &\lesssim |\eta |^{2-\alpha}|\xi |^{-1}|2(1-|t|^{\alpha+2})-(1-t^2)(|1 +t |^{\alpha}+|1-t|^{\alpha})| \\
    &= |\eta |^{2-\alpha}|\xi |^{-1}|2(1-|t|^{\alpha+2})-(1-t^2)(2+O(|t|))| \\
    &\lesssim |\xi||\eta|^{-\alpha}.
\end{align*}

\noindent\textbf{Case 2: $|\xi | \gg \eta  $.} 

Let $t=\eta /\xi$. Then
\begin{align*}
    |g(\eta ,\xi )| & \approx |\eta |^{-2}|\xi |^{1-2\alpha} |2(|\eta |^{\alpha+2}-|\xi |^{\alpha+2})-(\eta ^2-\xi ^2)(|\eta +\xi |^{\alpha}+|\eta -\xi |^{\alpha})| \\
    &=|\eta |^{-2}|\xi |^{3-\alpha}\left| 2(|t|^{\alpha+2}-1)-(t^2-1)(|t+1|^{\alpha}+|t-1|^{\alpha})\right| \\
    & = |\eta |^{-2}|\xi |^{3-\alpha}\left|2(|t|^{\alpha+2}-1)-(t^2-1)(2+O(|t|))\right| \\
    & \lesssim |\xi |^{1-\alpha}.
\end{align*}

\noindent\textbf{Case 3: $|\xi | \approx \eta $.} 

Without loss of generality assume that $\xi \eta >0$. By Lemma \ref{lem: phase function asymptotics 1}, we have
\begin{align*}
    |\Omega_{2}(\eta,-\xi)\Omega_{2}(-\eta,-\xi)| \gtrsim |\xi|^{2\alpha}|\xi^2-\eta^2| \approx |\xi|^{2\alpha+1}|\xi-\eta|.
\end{align*}
Using this and the mean value theorem, we have
\begin{align*}
   \left| \frac{|\eta |^{\alpha+2}-|\xi |^{\alpha+2}}{\Omega_{2}(\eta ,-\xi )\Omega_{2}(-\eta ,-\xi )} \right| \lesssim |\xi|^{-\alpha},
\end{align*}
and
\begin{align*}
    \frac{|\eta ^2-\xi ^2|(|\eta +\xi |^{\alpha}+|\eta -\xi |^{\alpha})}{|\Omega_{2}(\eta ,-\xi )\Omega_{2}(-\eta ,-\xi )|} \lesssim |\xi|^{-\alpha}.
\end{align*}
Therefore $|g(\eta ,\xi )| \lesssim |\xi |^{1-\alpha}$.
\end{proof}

Lemma \ref{lem: Rk multiplier} below provides estimates of the symbols associated with the higher-order near-resonant terms. As in the proof of the previous lemma, we crucially make use of the cancellation property induced by the symmetrization of indices. We stress that the summation over $\textnormal{Perm}(\mathcal{K})$ in the left-hand side of (\ref{eq: m bound}) below is essential, since the individual summands $R^{1}_{\theta(1),\dots,\theta(N)}(u)$ do not enjoy the required cancellation property.
\begin{lem} \label{lem: Rk multiplier}
    Let $N \geq 2$. Consider the multiset $\mathcal{K}:=\{k_1^{n_1},\dots,k_m^{n_m}\}$ with $\sum_{j=1}^{m}n_{j}=N$. Let $M:=1+\sum_{j=1}^{m}n_{j}(k_{j}-1)$. Then we have
    \begin{align}\label{eq: m bound}
    \mathcal{F}\Bigg[\sum_{\theta \in \textnormal{Perm}(\mathcal{K})}R^{1}_{\theta(1),\dots,\theta(N)}(u)\Bigg](\xi)=\sum_{\xi_1+\cdots+\xi_M=\xi} m(\xi_1,\dots,\xi_M)\hat{u}(\xi_1)\cdots\hat{u}(\xi_M) 
\end{align}
for some $m:\Z_{\ast}^{M} \to \C$ such that $\supp(m)\subseteq\mathcal{R}^{1}_{M}$ and
\begin{align*}
    |m(\xi_1,\dots,\xi_M)|\lesssim
    \begin{cases}
       |\xi_{1}^{*}|^{-(N-1)(\alpha-1)}|\xi_2^*|, & \text{if }\mathcal{K}=\{2^{n_1},3^{n_2}\}, \\
       |\xi_{1}^{*}|^{-(N-1)(\alpha-1)}|\xi_2^*||\xi_4^*|^{\frac{1}{\alpha}}, & \text{otherwise}.
   \end{cases}
\end{align*}
\end{lem}

\begin{proof}
    Assume without loss of generality that $|\xi_1^*|=|\xi_1| \gg |\xi_2^*|$.  It is easy to see by induction that if $(j_1,\dots,j_{N-1})\neq(1,\dots,1)$, then
\begin{align*}
    \left|\mu_{\theta(1),\dots,\theta(N)}^{j_1,\dots,j_{N-1}}(\xi_1,\dots,\xi_M)\right| \lesssim |\xi_{1}^{*}|^{-(N-1)(\alpha-1)-1}|\xi_2^*|.
\end{align*}
    Below we concentrate on the case  $(j_1,\dots,j_{N-1})=(1,\dots,1)$.
    
    Write $\mu_{\theta}:=\mathbbm{1}_{\mathcal{R}^{1}_{M}}\mu^{1,\dots,1}_{\theta(1),\dots,\theta(N)}$. Let $S:=\{\sigma \in S_M:\sigma(1)=1\}$. We claim that
\begin{align}\label{eq: mu theta symmetrized}
    \left|\sum_{\sigma \in S}\sum_{\theta \in \textnormal{Perm}(\mathcal{K})}\mu_{\theta}(\xi_1,\xi_{\sigma(2)}\dots,\xi_{\sigma(M)})\right| 
    \lesssim \begin{cases}
       |\xi_{1}^{*}|^{-(N-1)(\alpha-1)-1}|\xi_2^*|, &\text{if }\mathcal{K}=\{2^{n_1},3^{n_2}\}, \\
       |\xi_{1}^{*}|^{-(N-1)(\alpha-1)-1}|\xi_2^*||\xi_4^*|^{\frac{1}{\alpha}}, &\text{otherwise}.
   \end{cases}
\end{align}
For $\theta \in \textnormal{Perm}(\mathcal{K})$ and $1\leq n \leq N-1$, let $\iota_{\theta,n}:=M-\sum_{i=1}^{n}(\theta(i)-1)$. That is,
\begin{align*}
    \iota_{\theta,1}&=\theta(2)+\cdots+\theta(N), \\
    \iota_{\theta,2}&=\theta(3)+\cdots+\theta(N), \\
    &\ \, \vdots \\
    \iota_{\theta,N-1}&=\theta(N).
\end{align*}
Write $\t \xi_{n}:=\xi_1+\cdots+\xi_n$ and let
\begin{align*}
    &\phi_{\theta,n}:=\phi_{\theta(n)}(\t \xi_{\iota_{\theta,n}},\xi_{\iota_{\theta,n}+1},\dots,\xi_{\iota_{\theta,n-1}}), \\
    &\mathbbm{1}_{\theta,n}:=\mathbbm{1}_{\mathcal{N}_{\theta(1)+\sum_{i=2}^{n}(\theta(i)-1)}}(\t \xi_{\iota_{\theta,n}},\xi_{\iota_{\theta,n}+1},\dots,\xi_{M}), \\
    &\Omega_{\theta,n}:=\Omega_{\theta(1)+\sum_{i=2}^{n}(\theta(i)-1)}(\t \xi_{\iota_{\theta,n}},\xi_{\iota_{\theta,n}+1},\dots,\xi_{M}).
\end{align*} 
Then
\begin{align*}
    \mu_{\theta}(\xi_1,\dots,\xi_M)
    =\mathbbm{1}_{\mathcal{R}^{1}_{M}}(\xi_1,\dots,\xi_M)\phi_{\theta(N)}(\xi_1,\dots,\xi_{\theta(N)})
    \prod_{n=1}^{N-1}\frac{\t \xi_{\iota_{\theta,n}}\phi_{\theta,n}\mathbbm{1}_{\theta,n}}{\Omega_{\theta,n}}.
\end{align*}

We claim that for
\begin{multline*}
    \mathbbm{1}_{\theta}:=\mathbbm{1}_{\left\{\sum_{i=2}^{\theta(N)}\xi_i\neq 0, \,\sum_{i=\theta(N)+1}^{\theta(N)+\theta(N-1)}\xi_i\neq 0,\dots,\sum_{i=\theta(N)+\cdots+\theta(2)+1}^{M}\xi_i\neq 0 \right\}}\\
    \times\mathbbm{1}_{\left\{\sum_{i=2}^{\theta(N)}\xi_i\neq 0, \,\sum_{i=2}^{\theta(N)+\theta(N-1)}\xi_i\neq 0,\dots,\sum_{i=2}^{\theta(N)+\cdots+\theta(2)}\xi_i\neq 0 \right\}},
\end{multline*}
we have
\begin{align} \label{eq: mu theta}
    \mu_{\theta}(\xi_1,\xi_{\sigma(2)}\dots,\xi_{\sigma(M)})=\mathbbm{1}_{\mathcal{R}^{1}_{M}}(\xi_1,\dots,\xi_{M})\times \sigma \cdot \left[\mathbbm{1}_{\theta}\prod_{n=1}^{N-1}\frac{\t \xi_{\iota_{\theta,n}}}{\Omega_{\theta,n}}\right]
\end{align}
for all $\sigma \in S$. To see this, first note that the function $\mathbbm{1}_{\mathcal{R}^{1}_{M}}$ is already symmetric. By (\ref{eq: 0th resonance vanishes}), we have
\begin{align} \label{eq: 11}
    \phi_{\theta,n}=\mathbbm{1}_{\left\{\sum_{i=\iota_{\theta,n}+1}^{\iota_{\theta,n-1}}\xi_i \neq 0 \right\}}
\end{align}
for $1\leq n \leq N$. Taking $\sigma \cdot \prod_{n=1}^{N-1}$ to (\ref{eq: 11}), we get the first factor of $\sigma \cdot\mathbbm{1}_{\theta}$. To see how the second factor of $\sigma \cdot\mathbbm{1}_{\theta}$ is obtained, we examine the $\sigma \cdot \mathbbm{1}_{\theta,n}$ portion of the left-hand side of (\ref{eq: mu theta}). For $\theta \in \textnormal{Perm}(\mathcal{K})$ and $1 \leq n \leq N-1$, define the set $\mathcal{B}_{\theta,n} \subseteq \Z^{M}_{\ast}$ by
\begin{align*}
    \mathcal{B}_{\theta,n}:=\big\{ (\xi_1,\dots,\xi_{M}) \in \Z^{M}_{\ast} :  
    \left|\Omega_{\theta,n}\right| \gtrsim |\xi_{\iota_{\theta,n}+1}+\cdots+\xi_M||\xi_1|^{\alpha}\},
\end{align*}
and let $\mathcal{B}_{\theta}:=\bigcap_{n=1}^{N-1}\mathcal{B}_{\theta,n}$. Observe that we have
\begin{align} \label{eq: assumption B}
    (\xi_1,\xi_{\sigma(2)},\dots,\xi_{\sigma(M)})\in \mathcal{B}_{\theta} \text{ for all } \sigma \in S.
\end{align}
Indeed, if not, there exists at least one $(n,\sigma)$ such that
\begin{align} \label{eq: original condition}
    \xi_{\sigma(\iota_{\theta,n}+1)}+\cdots+\xi_{\sigma(M)} \neq 0 \textnormal{ and } \sigma \cdot \left|\Omega_{\theta,n}\right| \ll |\xi_{\sigma(\iota_{\theta,n}+1)}+\cdots+\xi_{\sigma(M)}||\xi_1|^{\alpha}.
\end{align}
By Lemma \ref{lem: higher-order resonances}, this implies
\begin{align*}
    |\xi_3^*|^{\alpha}|\xi_4^*| \gtrsim |\xi_1|^{\alpha},
\end{align*}
however we have already filtered out this possibility in the definition of $\mathcal{R}^{1}_{M}$. Now, by (\ref{eq: assumption B}), we see that $\sigma \cdot \mathbbm{1}_{\theta,n}$ is non-zero if $\xi_{\sigma(\iota_{\theta,n}+1)}+\cdots+\xi_{\sigma(M)}\neq 0$. Conversely, we show that
\begin{align} \label{eq: zero}
    \xi_{\sigma(\iota_{\theta,n}+1)}+\cdots+\xi_{\sigma(M)}=0
\end{align}
implies $\sigma \cdot \mathbbm{1}_{\theta,n}=0$. To see this, let $$(\eta_1,\eta_2,\dots,\eta_{\theta(1)+\sum_{i=2}^{n}(\theta(i)-1)}):=(\sigma \cdot \t \xi_{\iota_{\theta,n}},\xi_{\sigma(\iota_{\theta,n}+1)},\dots,\xi_{\sigma(M)}).$$ Observe that (\ref{eq: zero}) implies
\begin{align} \label{eq: Omega-2}
    |\Omega_{\theta,n}|=\big|\Omega_{\sum_{i=1}^{n}(\theta(i)-1)-1}(\eta_{3}^*,\dots,\eta_{\theta(1)+\sum_{i=2}^{n}(\theta(i)-1)}^*)\big|.
\end{align}
By the mean value theorem and the definition of $\mathcal{R}^{1}_{M}$, we have
\begin{align*}
    (\ref{eq: Omega-2}) \leq \big| \eta_{4}^*+\dots+\eta_{\theta(1)+\sum_{i=2}^{n}(\theta(i)-1)}^*\big||\eta_3^*|^{\alpha}+|\eta_4^*|^{\alpha+1} \lesssim |\xi_3^*|^{\alpha}|\xi_4^*| \ll |\xi_1|^{\alpha} \approx |\eta_1|^{\alpha},
\end{align*}
which implies $\sigma \cdot \mathbbm{1}_{\theta,n}=0$.\footnote{For this implication, one needs some care in selecting the implicit multiplicative constants in the definition of the sets $\mathcal{R}^{1}_n$, $\mathcal{R}^{2}_n$, and $\mathcal{N}_n$ in subsection \ref{subsection: Decomposition of the frequency domain}. This is however straightforward by interpreting ``$|\xi_{3}^*|^{\alpha}|\xi_4^*| \ll |\xi_1^*|^{\alpha}$" and ``$|\xi_{3}^*|^{\alpha}|\xi_4^*| \gtrsim  \lb\xi_2^*+\cdots+\xi_n^*\rb|\xi_1^*|^{\alpha}$" in that definition as ``$|\xi_{3}^*|^{\alpha}|\xi_4^*| < C|\xi_1^*|^{\alpha}$"  and ``$|\xi_{3}^*|^{\alpha}|\xi_4^*| \geq C \lb\xi_2^*+\cdots+\xi_n^*\rb|\xi_1^*|^{\alpha}$" for a sufficiently small $0<C\ll 1$, without touching the multiplicative constant in the definition of $\mathcal{N}_n$.} Hence, $\sigma \cdot \mathbbm{1}_{\theta,n} \neq 0$ if and only if $\xi_{\sigma(\iota_{\theta,n}+1)}+\cdots+\xi_{\sigma(M)}\neq 0$. Invoking $\xi_2+\cdots+\xi_M=0$, we have
\begin{align*}
   \sigma \cdot\prod_{n=1}^{N-1}\mathbbm{1}_{\theta,n}
    =\sigma \cdot\mathbbm{1}_{\left\{\sum_{i=2}^{\theta(N)}\xi_i\neq 0, \,\sum_{i=2}^{\theta(N)+\theta(N-1)}\xi_i\neq 0,\dots,\sum_{i=2}^{\theta(N)+\cdots+\theta(2)}\xi_i\neq 0 \right\}},
\end{align*}
obtaining the second factor of $\sigma \cdot \mathbbm{1}_{\theta}$. This proves (\ref{eq: mu theta}).

Next we examine some asymptotic properties of $\Omega_{\theta,n}$. Define $Q^{1}_{\theta,n},Q^{2}_{\theta,n}$ by
\begin{align*}
     &Q^{1}_{\theta,n}:=-(\alpha+1)(\xi_2+\cdots+\xi_{\iota_{\theta,n}})|\xi_1|^{\alpha}, \\
     &Q^{2}_{\theta,n}:=\Omega_{\theta,n}-Q^{1}_{\theta,n}.
\end{align*}
Let $t=(\xi_2+\cdots+\xi_{\iota_{\theta,n}})\xi_{1}^{-1}$. Then $|t| \ll 1$, and by Taylor's theorem,
\begin{align*} 
    \Omega_{\theta,n}
    &=\omega(\xi_{1})-\omega(\xi_1+\cdots+\xi_{\iota_{\theta,n}})-\sum_{j=\iota_{\theta,n}+1}^{M}\omega(\xi_i) \\
    &=\omega(\xi_{1})(1-\omega(1+t))-\sum_{j=\iota_{\theta,n}+1}^{M}\omega(\xi_i) \\
    &=\omega(\xi_{1})(-(\alpha+1)t+O(|t|^2))-\sum_{j=\iota_{\theta,n}+1}^{M}\omega(\xi_i) \\
    &=
    -(\alpha+1)(\xi_2+\cdots+\xi_{\iota_{\theta,n}})|\xi_1|^{\alpha}
    -\left[\sum_{j=\iota_{\theta,n}+1}^{M}\omega(\xi_i)+O\left(|\xi_1|^{\alpha-1}|\xi_2+\cdots+\xi_{\iota_{\theta,n}}|^{2}\right) \right].
\end{align*}
Therefore, we have
\begin{align} \label{eq: Q2 asymptotic}
     Q^{2}_{\theta,n}
     =-\sum_{j=\iota_{\theta,n}+1}^{M}\omega(\xi_i)+O\left(|\xi_1|^{\alpha-1}|\xi_2+\cdots+\xi_{\iota_{\theta,n}}|^{2}\right) .
\end{align}
Write $R_{\theta}$ for
\begin{align*}
    R_{\theta}:=\prod_{n=1}^{N-1}\Omega_{\theta,n}-\prod_{n=1}^{N-1}Q^{1}_{\theta,n}.
\end{align*}
Then we have
\begin{align*}
    \mu_{\theta}(\xi_1,\dots,\xi_M)
    &=\mathbbm{1}_{\theta}\prod_{n=1}^{N-1}\frac{\t \xi_{\iota_{\theta,n}}}{\Omega_{\theta,n}} \\
    &=\mathbbm{1}_{\theta}\frac{\prod_{n=1}^{N-1}\t \xi_{\iota_{\theta,n}}}{R_{\theta}+\prod_{n=1}^{N-1}Q^{1}_{\theta,n}} \\
    &=\mathbbm{1}_{\theta}\prod_{n=1}^{N-1}\frac{\t \xi_{\iota_{\theta,n}}}{Q^{1}_{\theta,n}}-\frac{R_{\theta}\mu_{\theta}(\xi_1,\dots,\xi_M)}{\prod_{n=1}^{N-1}Q^{1}_{\theta,n}} \\
    &=:A_{\theta}+B_{\theta},
\end{align*}
where in the third equality, we used the identity
\begin{align*}
    \frac{c}{a+b}=\frac{c}{b}-\frac{a}{b}\frac{c}{a+b}.
\end{align*}
We first look at $A_{\theta}$:
\begin{align*}
    A_{\theta}=\mathbbm{1}_{\theta}\frac{\xi_1^{N-1}}{\prod_{n=1}^{N-1}Q^{1}_{\theta,n}}+\mathbbm{1}_{\theta}\frac{\left(\prod_{n=1}^{N-1}\t \xi_{\iota_{\theta,n}}\right)-\xi_1^{N-1}}{\prod_{n=1}^{N-1}Q^{1}_{\theta,n}}=:A_{\theta}^{1}+A_{\theta}^{2}
\end{align*}
By Lemma \ref{lem: cancellation}, we have
\begin{align*}
    \sum_{\sigma \in S}\sigma \cdot \Bigg[\sum_{\theta \in \textnormal{Perm}(\mathcal{K})}A_{\theta}^{1} \Bigg]=0.
\end{align*}
Moreover, since
\begin{align*}
    \left|\frac{\xi_1}{Q^{1}_{\theta,n}}\right| \lesssim |\xi_1|^{-(\alpha-1)}, \quad
    \left|\frac{\t \xi_{\iota_{\theta,n}}-\xi_1}{Q^{1}_{\theta,n}}\right| \lesssim |\xi_1|^{-\alpha},
\end{align*}
we have $\left|A_{\theta}^{2}\right| \lesssim |\xi_1|^{-(N-1)(\alpha-1)-1}$.  Therefore,
\begin{align*}
    \sum_{\sigma \in S}\sigma \cdot \Bigg[\sum_{\theta \in \textnormal{Perm}(\mathcal{K})}A_{\theta} \Bigg] \lesssim |\xi_1|^{-(N-1)(\alpha-1)-1}.
\end{align*}

Next we examine $B_{\theta}$. Since
\begin{align*}
    R_{\theta}=\sum_{(i_1,\dots,i_{N-1})\in\{1,2\}^{N-1}\setminus\{(1,\dots,1)\}}\prod_{n=1}^{N-1}Q_{\theta,n}^{i_n},
\end{align*}
we have
\begin{align*}
    \left|\frac{R_{\theta}}{\prod_{n=1}^{N-1}Q^{1}_{\theta,n}}\right| \lesssim \max_{1\leq n \leq N-1}\left|\frac{Q^2_{\theta,n}}{Q^1_{\theta,n}}\right|.
\end{align*}
By (\ref{eq: Q2 asymptotic}),
\begin{align*}
    \left|\frac{Q^2_{\theta,n}}{Q^1_{\theta,n}}\right|=\left|\frac{\omega(\xi_{\iota_{\theta,n}+1})+\cdots+\omega(\xi_M)}{\xi_{\iota_{\theta,n}+1}+\cdots+\xi_M}\right||\xi_1|^{-\alpha}+O\left(|\xi_1|^{-1}|\xi_2+\cdots+\xi_{\iota_{\theta,n}}|\right).
\end{align*}
Also by (\ref{eq: assumption B}) we have for all $\sigma \in S$,
\begin{align*}
    \left|\mu_{\theta}(\xi_1,\xi_{\sigma(2)},\dots,\xi_{\sigma(M)})\right|\lesssim \frac{|\xi_1|^{-(N-1)(\alpha-1)}\sigma \cdot \mathbbm{1}_{\theta,n}}{\prod_{n=1}^{N-1}\big|\xi_{\sigma(2)}+\cdots+\xi_{\sigma(\iota_{\theta,n})}\big|}.
\end{align*}
Hence
\begin{multline*}
    \sigma \cdot \left|B_{\theta}\right| \lesssim \frac{|\xi_1|^{-(N-1)(\alpha-1)-\alpha}\sigma \cdot \mathbbm{1}_{\theta,n}}{\prod_{n=1}^{N-1}\big|\xi_{\sigma(2)}+\cdots+\xi_{\sigma(\iota_{\theta,n})}\big|}\max_{1\leq n \leq N-1}\left| \frac{\omega(\xi_{\sigma(\iota_{\theta,n}+1)})+\cdots+\omega(\xi_{\sigma(M)})}{\xi_{\sigma(\iota_{\theta,n}+1)}+\cdots+\xi_{\sigma(M)}}\right|\\
    +O\left(|\xi_1|^{-(N-1)(\alpha-1)-1}\right).
\end{multline*}
Below we estimate
\begin{align}\label{eq: omega/xi}
    \frac{\sigma \cdot \mathbbm{1}_{\theta,n}}{\prod_{k=1}^{N-1}\big|\xi_{\sigma(2)}+\cdots+\xi_{\sigma(\iota_{\theta,k})}\big|}\max_{1\leq n \leq N-1}\left|\frac{\omega(\xi_{\sigma(\iota_{\theta,n}+1)})+\cdots+\omega(\xi_{\sigma(M)})}{\xi_{\sigma(\iota_{\theta,n}+1)}+\cdots+\xi_{\sigma(M)}}\right|.
\end{align}

\noindent \textbf{Case 1: $\mathcal{K}=\{2^{n_1},3^{n_2}\}$.}

We only consider the case $\mathcal{K}=\{3^N\}$, since the remaining cases can be handled in a similar way. By Lemma \ref{lem: max}, we have\begin{align*}
    \frac{\sigma \cdot \mathbbm{1}_{\theta,n}}{\prod_{k=1}^{N-1}\big|\xi_{\sigma(2)}+\cdots+\xi_{\sigma(\iota_{\theta,k})}\big|} \lesssim \frac{\mathbbm{1}_{\{\xi_{\sigma(2)}+\xi_{\sigma(3)}\neq 0\}}}{\max_{1\leq n \leq N-1}\left|\xi_{\sigma(2n+2)}+\xi_{\sigma(2n+3)}\right|}.
\end{align*}
Moreover, we have
\begin{align*}
    \max_{1\leq n \leq N-1}\left|\frac{\omega(\xi_{\sigma(\iota_{\theta,n}+1)})+\cdots+\omega(\xi_{\sigma(M)})}{\xi_{\sigma(\iota_{\theta,n}+1)}+\cdots+\xi_{\sigma(M)}}\right|
    & \leq \sum_{n=1}^{N-1}\left|\omega(\xi_{\sigma(2n+2)})+\omega(\xi_{\sigma(2n+3)})\right| \\
    & \lesssim |\xi_2^*|^{\alpha}\sum_{n=1}^{N-1}\left|\xi_{\sigma(2n+2)}+\xi_{\sigma(2n+3)}\right|,
\end{align*}
where we used the mean value theorem in the second inequality.
Therefore, (\ref{eq: omega/xi}) is dominated by $|\xi_2^*|^{\alpha}$.

\noindent \textbf{Case 2: $\mathcal{K} \neq \{2^{n_1},3^{n_2}\}$.}

Below we show that $(\ref{eq: omega/xi}) \lesssim |\xi_1^*|^{\alpha-1}|\xi_2^*||\xi_4^*|^{\frac{1}{\alpha}}$. Note that since $\xi_2^*+\cdots+\xi_M^*=0$, we have $|\xi_2^*|\approx|\xi_3^*|$. Also, recall that we are assuming $|\xi_3^*|^{\alpha}|\xi_4^*| \ll |\xi_1^*|^{\alpha}$. Hence,
\begin{align}\label{eq: xi ineq}
    |\xi_2^*|^{\alpha}=|\xi_1|^{\alpha-1}|\xi_1|^{-(\alpha-1)}|\xi_2^*|^{\alpha}\lesssim |\xi_1|^{\alpha-1}(|\xi_2^*|^{\alpha}|\xi_4^*|)^{-\frac{\alpha-1}{\alpha}}|\xi_2^*|^{\alpha}=|\xi_1|^{\alpha-1}|\xi_2^*||\xi_4^*|^{-1+\frac{1}{\alpha}}.
\end{align}
We consider the following subcases:

\noindent \underline{Subcase 2a: $|\xi_{\sigma(\iota_{\theta,n}+1)}+\cdots+\xi_{\sigma(M)}| \approx |\xi_2^*|$.} In this case, we have $(\ref{eq: omega/xi}) \lesssim |\xi_2^*|^{\alpha-1}$, which is more than enough.

\noindent\underline{Subcase 2b: $|\xi_{\sigma(\iota_{\theta,n}+1)}+\cdots+\xi_{\sigma(M)}| \ll |\xi_2^*|$ and $|\xi_3^*| \approx |\xi_4^*|$.} In this case, we have $(\ref{eq: omega/xi}) \lesssim |\xi_2^*|^{\alpha+1} \approx |\xi_2^*|^{\alpha}|\xi_4^*|$. This with (\ref{eq: xi ineq}) implies $(\ref{eq: omega/xi}) \lesssim |\xi_1^*|^{\alpha-1}|\xi_2^*||\xi_4^*|^{\frac{1}{\alpha}}$.

\noindent\underline{Subcase 2c: $|\xi_{\sigma(\iota_{\theta,n}+1)}+\cdots+\xi_{\sigma(M)}| \ll |\xi_2^*|$ and $|\xi_3^*| \gg |\xi_4^*|$.} In this case, we have either 
\begin{align} \label{eq: first}
    \max\big(|\xi_{\sigma(\iota_{\theta,n}+1)}|,\dots,|\xi_{\sigma(M)}|\big) \lesssim |\xi_4^*|
\end{align}
or
\begin{align} \label{eq: second}
   \{\xi_2^*,\xi_3^*\} \subset \{\xi_{\sigma(\iota_{\theta,n}+1)},\dots,\xi_{\sigma(M)}\}.
\end{align}
Under condition (\ref{eq: first}), we have $(\ref{eq: omega/xi}) \lesssim |\xi_4^*|^{\alpha+1}$, which implies the desired bound by (\ref{eq: xi ineq}). Now assume (\ref{eq: second}). Recall that since $\xi_2^*+\cdots+\xi^*_{M}=0$, we have $|\xi_2^*+\xi_3^*| \lesssim |\xi_4^*|$. Hence the mean value theorem implies
\begin{align*}
    |\omega(\xi_2^*)+\omega(\xi_3^*)| \lesssim |\xi_2^*+\xi_3^*||\xi_2^*|^{\alpha} \lesssim |\xi_4^*||\xi_2^*|^{\alpha}.
\end{align*}
Therefore, by (\ref{eq: xi ineq}) and (\ref{eq: second}) we have
\begin{align*}
   (\ref{eq: omega/xi}) \lesssim |\omega(\xi_2^*)+\omega(\xi_3^*)|+|\xi_4^*|^{\alpha+1} \lesssim |\xi_4^*||\xi_2^*|^{\alpha}+|\xi_4^*|^{\alpha+1} \lesssim |\xi_1^*|^{\alpha-1}|\xi_2^*||\xi_4^*|^{\frac{1}{\alpha}}.
\end{align*}
This proves (\ref{eq: mu theta symmetrized}). 
\end{proof}

\section{Smoothing estimates}\label{section: Smoothing estimates}

This section is devoted to the proof of Theorem \ref{thm: smoothing}. Estimates in this section are done without using any auxiliary function spaces.

\subsection{Multilinear estimates}

In this subsection, we present estimates of the terms that can be examined with standard multilinear analysis. Based on the pointwise estimates in Section \ref{section: Pointwise estimates}, the proofs below are mostly repetitive applications of basic inequalities like Young's convolution inequality
\begin{align*}
    \|f \ast g\|_{L^{r}} \lesssim \|f\|_{L^p}\|g\|_{L^q} \textnormal{ for } 1\leq p,q,r \leq \infty \textnormal{ and } \frac{1}{p}+\frac{1}{q}=1+\frac{1}{r},
\end{align*}
or Bernstein's inequality
\begin{align*}
    \|P_N f\|_{L^q} \lesssim N^{\frac{1}{p}-\frac{1}{q}} \|P_N f\|_{L^p} \textnormal{ for } 1\leq p \leq q \leq  \infty.
\end{align*}
The estimates for the terms $R^{2}_{k_0,\dots,k_n}(u)$ and $\mathfrak{R}_{k_0,\dots,k_n}^{l_1,\dots,l_m}(u)$ require different techniques, hence we postpone these estimates to Subsection \ref{subsection: Energy estimates}.

\subsubsection{Boundary terms}

\begin{lem} \label{lem: boundary term sobolev}
     Let $m,n \geq 0$, $(k_0,\dots,k_n) \in [d]^{n+1}$, and $(l_1,\dots,l_m) \in [d]^{m}$. 
    \begin{enumerate}[label=(\alph*)]
    \item For $s > \frac32-\alpha$, $a < n\big(s+\alpha-\frac32\big)$ and $a\leq n(\alpha-1)$, we have
    \begin{align*}
        \|B_{k_0,\dots,k_n}(u)\|_{H^{s+a}} \lesssim \|u\|_{H^{s}}^{\nu_n}.
    \end{align*}
    \item For $s > \frac12$ and $a\leq (n+m+1)(\alpha-1)$, we have
    \begin{align*}
        \|\mathfrak{B}_{k_0,\dots,k_n}^{l_1,\dots,l_m}(u)\|_{H^{s+a}} \lesssim \|u\|_{H^{s}}^{\nu_{n,m}}.
    \end{align*}
    \end{enumerate}
\end{lem}

\begin{proof}
Consider (a). By Lemma \ref{lem: mu pointwise bound} (a), we have
\begin{align*}
    \left |\frac{\xi \mathbbm{1}_{\mathcal{N}_{\nu_{n}}}\mu_{k_0,\dots,k_n}(\xi_1,\dots,\xi_{\nu_{n}})}{\Omega_{\nu_{n}}(\xi_1,\dots,\xi_{\nu_{n}})}\right | \lesssim |\xi_1^*|^{-n(\alpha-1)}.
\end{align*}
Hence by Young's convolution inequality,
    \begin{align*}
        \left \|B_{k_0,\dots,k_n}(u) \right \|_{H^{s+a}}
        &\lesssim \Bigg \| \sum_{\xi_1+\cdots+\xi_{\nu_{n}}=\xi}|\xi_1^*|^{s+a+n(1-\alpha)}|\hat{u}(\xi_1)\cdots \hat{u}(\xi_{\nu_{n}})| \Bigg \|_{\ell^{2}_{\xi}} \\
        &\lesssim \Bigg \|\sum_{\xi_1+\cdots+\xi_{\nu_{n}}=\xi}|\xi_1^*|^s|\xi_2^* \cdots \xi_{\nu_{n}}^*|^{\frac{1}{n}(a+n(1-\alpha))}|\hat{u}(\xi_1)\cdots \hat{u}(\xi_{\nu_{n}})|\Bigg \|_{\ell^{1}_{\xi}} \\
         &\lesssim \left \| u\right \|_{H^s}\left \|u\right \|_{H^{\frac{1}{n}(a+n(1-\alpha))+\frac12+}}^{\nu_{n}-1} \\
        &\lesssim \|u\|_{H^{s}}^{\nu_{n}}.
    \end{align*}

Next we consider (b). By Lemma \ref{lem: mu pointwise bound} (b), we have
\begin{align*}
    \left |\frac{\xi \mathbbm{1}_{\mathcal{D}^{2}_{\nu_{n,m}}}\mathfrak{m}_{k_0,\dots,k_n}^{l_1\dots,l_{m}}(\xi_1,\dots,\xi_{\nu_n})}{\Omega_{\nu_n}(\xi_1,\dots,\xi_{\nu_n})}\right | \lesssim |\xi|^{-(n+m+1)(\alpha-1)}.
\end{align*}
Since $a-(n+m+1)(\alpha-1)\leq 0$, by Young's convolution inequality
    \begin{align*}
        \left \|\mathfrak{B}_{k_0,\dots,k_n}^{l_1,\dots,l_m}(u)\right \|_{H^{s+a}} 
        &\lesssim \left \| \sum_{\xi_1+\cdots+\xi_{\nu_{n,m}}=\xi}|\xi|^{s+a-(n+m+1)(\alpha-1)}|\hat{u}(\xi_1)\cdots \hat{u}(\xi_{\nu_{n,m}})| \right \|_{\ell^{2}_{\xi}} \\
        &\lesssim \left \| u\right \|_{H^s}\left \|\hat{u}\right \|_{\ell^{1}_{\xi}}^{\nu_{n,m}-1} \\
         &\lesssim \left \| u\right \|_{H^s}\left \| u\right \|_{H^{\frac12+}}^{\nu_{n,m}-1}.
    \end{align*}
\end{proof}

\subsubsection{Near-resonant terms}

\begin{lem} \label{lem: Rn estimate}
    Let $N \geq 2$. Consider the multiset $\mathcal{K}:=\{k_1^{n_1},\dots,k_m^{n_m}\}$ with $\sum_{j=1}^{m}n_{j}=N$. Let $M:=1+\sum_{j=1}^{m}n_{j}(k_{j}-1)$ and suppose that $M \geq 4$. If $\mathcal{K}=\{2^{n_1},3^{n_2}\}$, let $a \leq (N-1)(\alpha-1)$ and
    \begin{align*} 
    \begin{cases}
      s>\frac12-\frac{(N-1)(\alpha-1)}{M-1}, \\
      a<\min\big((M-1)\big(s-\frac12),2s-1\big)+(N-1)(\alpha-1).
    \end{cases}
    \end{align*}
    Otherwise, let $a \leq (N-1)(\alpha-1)$ and
    \begin{align*} 
    \begin{cases}
      s>\max\big(\frac{1}{2},\frac12+\frac{1}{3\alpha}-\frac{(N-1)(\alpha-1)}{3}\big), 
      \\a<\min\big(3s-\frac{3}{2}-\frac{1}{\alpha},2s-1\big)+(N-1)(\alpha-1),
    \end{cases}
    \end{align*}
    Then we have
    \begin{align*}
        \Bigg\| \sum_{\theta \in \textnormal{Perm}(\mathcal{K})}R^{1}_{\theta(1),\dots,\theta(N)}(u)\Bigg\|_{H^{s+a}} \lesssim \|u\|_{H^s}^{M}.
    \end{align*}
\end{lem}
\begin{proof}
First consider the case $\mathcal{K}=\{2^{n_1},3^{n_2}\}$. Note that by the assumption made on $s$ and $a$ in the statement of this lemma, we have $s\geq 0$, $a-(N-1)(\alpha-1)\leq 0$, $a+1-(N-1)(\alpha-1)-2s \leq 0$, and $s>\frac{1+a-(N-1)(\alpha-1)-2s}{M-3}+\frac12$. Also, since $\xi_2^*+\cdots+\xi_{M}^*=0$, we have $|\xi_2^*| \approx |\xi_3^*|$. Assume for simplicity that $|\xi_1| \geq |\xi_2| \geq \dots \geq |\xi_{M}|$. Using  (\ref{eq: m bound}) and Young's convolution inequality, we have
\begin{align*}
    &\Bigg\| \sum_{\theta \in \textnormal{Perm}(\mathcal{K})}R^{1}_{\theta(1),\dots,\theta(N)}(u)\Bigg\|_{H^{s+a}} \\
    &\quad\lesssim \bigg \| \sum_{\xi_1+\cdots+\xi_{M}=\xi} \mathbbm{1}_{\{\xi=\xi_1\}}|\xi_1|^{s+a-(N-1)(\alpha-1)}|\xi_2|\left|\hat{u}(\xi_{1})\cdots\hat{u}(\xi_{M})\right|\bigg \|_{\ell^2_{\xi}} \\
    &\quad\lesssim \|u\|_{H^{s}}  \bigg\| \sum_{\xi_2+\cdots+\xi_{M}=\xi} |\xi_2 \xi_3|^{s}|\hat{u}(\xi_{2})\hat{u}(\xi_{3})|\prod_{l=4}^{M}|\xi_l|^{\frac{1+a-(N-1)(\alpha-1)-2s}{M-3}}|\hat{u}(\xi_{l})| \bigg\|_{\ell^{\infty}_{\xi}}\\ 
    &\quad\lesssim  \|u\|_{H^{s}}^3\|u\|_{H^{\frac{1+a-(N-1)(\alpha-1)-2s}{M-3}+\frac12+}}^{M-3}\\
    &\quad\lesssim \|u\|_{H^{s}}^{M}.
\end{align*}
If $\mathcal{K}\neq\{2^{n_1},3^{n_2}\}$, assuming $|\xi_1| \geq |\xi_2| \geq \dots \geq |\xi_{M}|$, 
\begin{align*}
    &\Bigg\| \sum_{\theta \in \textnormal{Perm}(\mathcal{K})}R^{1}_{\theta(1),\dots,\theta(N)}(u)\Bigg\|_{H^{s+a}} \\
    &\quad\lesssim \bigg \| |\xi|^{s} |\hat{u}(\xi)| \sum_{\xi_2+\cdots+\xi_{M}=0} |\xi_2|^{1+a-(N-1)(\alpha-1)}|\xi_4|^{\frac{1}{\alpha}}|\hat{u}(\xi_{2})\cdots\hat{u}(\xi_{M})|\bigg \|_{\ell^2_{\xi}} \\
    &\quad\lesssim \|u\|_{H^{s}}  \bigg\| \sum_{\xi_2+\cdots+\xi_{M}=\xi} |\xi_2 \xi_3|^{s}|\hat{u}(\xi_{2})\hat{u}(\xi_{3})|\xi_4|^{1+\frac{1}{\alpha}+a-(N-1)(\alpha-1)-2s}|\hat{u}(\xi_{4})\cdots\hat{u}(\xi_{M})| \bigg\|_{\ell^{\infty}_{\xi}}\\ 
    &\quad\lesssim  \|u\|_{H^{s}}^3 \|u\|_{H^{1+\frac{1}{\alpha}+a-(N-1)(\alpha-1)-2s+\frac12+}}\|u\|_{H^{\frac12+}}^{M-4} \\
    &\quad\lesssim \|u\|_{H^{s}}^{M},
\end{align*}
as desired.
\end{proof}

\subsubsection{Non-resonant terms}

\begin{lem}
Let $(k_0\dots,k_n) \in [d]^{n+1}$ and $(l_1,\dots,l_m) \in [d]^{m}$.
\begin{enumerate}[label=(\alph*)]
    \item  Let $n \geq 2$ and $\frac{n+1}{n} < \alpha < 2$. For $s >\frac{3}{2}-\frac{n\alpha}{n+1}$ and $a <\min\big(n(\alpha-1)-1,(n+1)(s-\frac32)+n\alpha\big)$, we have
\begin{align*}
    \|N_{k_0,\dots,k_n}(u) \|_{H^{s+a}} \lesssim \|u\|_{H^{s}}^{\nu_n}.
\end{align*}
    \item Let $m+n \geq 1$ and $\frac{m+n+1}{m+n} < \alpha < 2$. For $s>\frac12$ and $a \leq (m+n)(\alpha-1)-1$, we have
\begin{align*}
    \|\mathfrak{N}_{k_0,\dots,k_n}^{l_1,\dots,l_m}(u) \|_{H^{s+a}} \lesssim \|u\|_{H^{s}}^{\nu_{n,m}}.
\end{align*}
\end{enumerate}
        
\end{lem}
\begin{proof}
We only consider (a), since (b) can be similarly proved. Using Lemma \ref{lem: mu pointwise bound} (a), we have
    \begin{align*}
        \|N_{k_0,\dots,k_n}(u)\|_{H^{s+a}}
        &\lesssim \bigg\| |\xi|^{s+a+1-n(\alpha-1)}\sum_{\xi_1+\cdots+\xi_{\nu_n}=\xi}|\hat{u}(\xi_1)\cdots\hat{u}(\xi_{\nu_n})|\bigg\|_{\ell^2_{\xi}} \\
        &\lesssim \bigg\| \sum_{\xi_1+\cdots+\xi_{\nu_n}=\xi}|\xi_1^*|^{s}|\xi_2^*\cdots\xi_{\nu_n}^*|^{\frac{n-n\alpha+a+1}{n+1}}|\hat{u}(\xi_1)\cdots\hat{u}(\xi_{\nu_n})|\bigg\|_{\ell^2_{\xi}} \\
        &\lesssim \|u\|_{H^{s}}\|u\|_{H^{\frac{n-n\alpha+a+1}{n+1}+\frac12+}}^{\nu_n-1} \\
        &\lesssim \|u\|_{H^{s}}^{\nu_n}.
    \end{align*}
The last line is acceptable as long as $s>\frac12+\frac{n-n\alpha+a+1}{n+1}$.
\end{proof}

\subsection{Estimates for the \texorpdfstring{$\textnormal{deg}(P)=2$}{deg(P)=2} case}

For the $\textnormal{deg}(P)=2$ case, finer analysis significantly improves some of the estimates in the previous subsection. 

\subsubsection{Boundary terms}
\begin{lem} \label{lem: B2 deg 2 estimate}
 For $s>\frac12-\alpha$ and $a <\min(\alpha,s+\alpha-\frac{1}{2})$ we have 
    \begin{align*}
        \| B_{2}(u) \|_{H^{s+a}} \lesssim \|u\|_{H^{s}}^2.
    \end{align*} 
\end{lem}
\begin{proof}
By symmetry we may only consider the cases $|\xi_2| \leq |\xi_1| \leq 2 |\xi_2|$ and $|\xi_1| \geq 2 |\xi_2|$. 

\noindent\textbf{Case 1: $|\xi_2| \leq |\xi_1| \leq 2 |\xi_2|$.}
    
In this case we have $|\xi| \lesssim |\xi_1|\approx |\xi_2|$. Also by Lemma \ref{lem: phase function asymptotics 1} $|\Omega_{2}(\xi_1, \xi_2)| \approx |\xi_1|^{\alpha}|\xi_2|\approx |\xi_2|^{\alpha+1}$. Thus for $-s+a-\alpha+\frac12<0$ we have
    \begin{align*}
        \left \| |\xi|^{s+a+1} \sum_{|\xi_2| \leq |\xi_1| \leq 2 |\xi_2|} \frac{|\hat{u}_{\xi_1}\hat{u}_{\xi_2}|}{|\Omega_{2}(\xi_1, \xi_2)|}\right \|_{\ell^{2}_\xi } 
        & \lesssim \left \| \sum_{|\xi_2| \leq |\xi_1| \leq 2 |\xi_2|} |\xi|^{-s+a-\alpha+\frac12+}|\xi_1|^s |\hat{u}_{\xi_1}|\frac{|\xi_2|^s |\hat{u}_{\xi_2}|}{|\xi_2|^{\frac12+}}\right \|_{\ell^{2}_\xi } \\
        & \leq \left \| \sum_{|\xi_2| \leq |\xi_1| \leq 2 |\xi_2|} |\xi_1|^s |\hat{u}_{\xi_1}||\xi_2|^{s-\frac12-}|\hat{u}_{\xi_2}|\right \|_{\ell^{2}_\xi }  \\
        & \lesssim \| |\xi|^s |\hat{u}_{\xi} | \|_{\ell^2_\xi } \left \| |\xi|^{s-\frac12-} |\hat{u}_{\xi} | \right \|_{\ell^1_\xi } \\
        & \leq \|u\|_{H^{s}}^2.
    \end{align*}
    
\noindent\textbf{Case 2: $|\xi_1| \geq 2 |\xi_2|$.} 
    
In this case we have $|\xi_2| \lesssim |\xi| \approx |\xi_1|$ and $|\Omega_{2}(\xi_1, \xi_2)| \approx |\xi_1|^{\alpha}|\xi_2|$. Thus for $a \leq \alpha$ and $s>a-\alpha+\frac12$ we have
    \begin{align*}
        \left \| |\xi|^{s+a} \sum_{|\xi_1| \geq 2 |\xi_2|} \frac{|\hat{u}_{\xi_1}\hat{u}_{\xi_2}|}{|\Omega_{2}(\xi_1, \xi_2)|}\right \|_{\ell^{2}_\xi } 
        \lesssim \left \| \sum_{|\xi_1| \geq 2 |\xi_2|} |\xi_2|^{-s+a-\alpha-\frac12+}|\xi_1|^s |\hat{u}_{\xi_1}|\xi_2|^{s-\frac12-} |\hat{u}_{\xi_2}|\right \|_{\ell^{2}_\xi}, 
    \end{align*}
and the right-hand side can be bounded by $\|u\|_{H^{s}}^2$ as in the first case.
\end{proof}

\begin{lem}
Let $m,n \geq 0$, $k_0=\cdots=k_n=2$ and $l_1=\cdots=l_m=2$. For $s>\frac32-\alpha$, $a \leq \alpha-1$, and $a<2s+2\alpha-2$, we have
    \begin{align*}
        \|\mathfrak{B}_{k_0,\dots,k_n}^{l_1,\dots,l_m}(u)\|_{H^{s+a}} \lesssim \|u\|_{H^s}^{n+m+2}.
    \end{align*}
\end{lem}

\begin{proof}
We need to show that the quantity
\begin{align} \label{eq: BB}
    \left\|\sum_{\xi_1+\cdots+\xi_{n+m+2}=\xi}^{\ast} \frac{|\xi|^{1+s+a}\mathbbm{1}_{\mathcal{D}^2_{n+m+2}}\big|\mathfrak{m}_{k_0,\dots,k_n}^{l_0,\dots,l_m}(\xi_1,\dots,\xi_{n+m+2})\big|}{|\xi_1 \cdots \xi_{n+m+2}|^{s}|\Omega_{n+m+2}(\xi_1,\dots,\xi_{n+m+2})|} |\hat{u}(\xi_1)\cdots\hat{u}(\xi_{n+m+2})|\right\|_{\ell^{2}_{\xi}}
\end{align}
is less than $\|u\|_{H^s}^{n+m+2}$. By duality, (\ref{eq: BB}) is equal to
\begin{multline} \label{eq: BB dual}
    \sup_{\|v\|_{L^2}=1}\sum_{\xi_1,\dots,\xi_{n+m+2}}^{\ast}\frac{|\xi|^{s+a+1}\mathbbm{1}_{\mathcal{D}^2_{n+m+2}}\big|\mathfrak{m}_{k_0,\dots,k_n}^{l_0,\dots,l_m}\big|}{|\xi_1 \cdots \xi_{n+m+2}|^{s}|\Omega_{n+m+2}(\xi_1,\dots,\xi_{n+m+2})|} \\
    \times |\hat{u}(\xi_1)\cdots\hat{u}(\xi_{n+m+2})\hat{v}(\xi_1+\cdots+\xi_{n+m+2})|.
\end{multline}
For each $1 \leq i \leq n+m+2$, let
\begin{align*}
    \zeta_i:=\max\left(|\xi_i^*|,|\xi|\right).
\end{align*}
By Lemma \ref{lem: mu pointwise bound} (c), for $2-2\alpha-2s+a<0$ we have
\begin{align*}
    \frac{|\xi|^{s+a+1}\mathbbm{1}_{\mathcal{D}^2_{n+m+2}}\big|\mathfrak{m}_{k_0,\dots,k_n}^{l_0,\dots,l_m}\big|}{|\xi_1 \cdots \xi_{n+m+2}|^{s}|\Omega_{n+m+2}(\xi_1,\dots,\xi_{n+m+2})|}
    &\lesssim \frac{|\xi|^{1-\alpha+s+a}|\zeta_3\cdots\zeta_{n+m+2}|^{-(\alpha-1)}}{|\xi_1^* \cdots \xi_{n+m+2}^*|^{s}\rho_{n+m+2}(\xi_1,\dots,\xi_{n+m+2})} \\
    &\lesssim \frac{|\xi|^{2-2\alpha-2s+a-}}{|\xi_4^* \cdots \xi_{n+m+2}^*|^{s-\alpha+1}|\rho_{n+m+2}(\xi_1,\dots,\xi_{n+m+2})|^{1+}} \\
    &\lesssim \frac{1}{|\xi_4^* \cdots \xi_{n+m+2}^*|^{s-\alpha+1}|\rho_{n+m+2}(\xi_1,\dots,\xi_{n+m+2})|^{1+}}.
\end{align*}
If $\rho_{n+m+2}=\lb \xi_1^*+\xi_2^* \rb$, then by Cauchy-Schwarz, the sum in (\ref{eq: BB dual}) is dominated by
\begin{multline*}
    \left(\sum_{\xi_1,\dots,\xi_{n+m+2}}^{\ast}\frac{|\hat{u}(\xi_1^*)\hat{u}(\xi_3^*)\cdots\hat{u}(\xi_{n+m+2}^*)|^2}{\lb \xi_1^*+\xi_2^* \rb^{1+}} \right)^{\frac12} \\
    \times \left(\sum_{\xi_1,\dots,\xi_{n+m+2}}^{\ast}\frac{|\hat{v}(\xi_1^*+\cdots+\xi_{n+m+2}^*)\hat{u}(\xi_2^*)|^{2}}{\lb \xi_1^*+\xi_2^* \rb^{1+}|\xi_4^*\cdots\xi_{n+m+2}^*|^{s+\alpha-1}} \right)^{\frac12}=:A\times B.
\end{multline*}
We sum $A$ in the order $\xi_2^*,\xi_1^*,\xi_3^*,\xi_4^*,\dots,\xi_{n+m+2}^*$, and sum $B$ in the order $\xi_3^*,\xi_1^*,\xi_2^*,\xi_4^*,\dots,\xi_{n+m+2}^*$.

If $\rho_{n+m+2}=\lb \xi_2^*+\cdots+\xi_{n+m+2}^*\rb$, then (\ref{eq: BB dual}) is less than
\begin{multline*}
    \left(\sum_{\xi_1,\dots,\xi_{n+m+2}}^{\ast}\frac{|\hat{u}(\xi_1^*)\hat{u}(\xi_3^*)\cdots\hat{u}(\xi_{n+m+2}^*)|^2}{\lb \xi_2^*+\cdots+\xi_{n+m+2}^* \rb^{1+}} \right)^{\frac12} \\
    \times \left(\sum_{\xi_1,\dots,\xi_{n+m+2}}^{\ast}\frac{|\hat{v}(\xi_1^*+\cdots+\xi_{n+m+2}^*)\hat{u}(\xi_2^*)|^{2}}{\lb \xi_2^*+\cdots+\xi_{n+m+2}^* \rb^{1+}|\xi_4^*\cdots\xi_{n+m+2}^*|^{s+\alpha-1}} \right)^{\frac12}.
\end{multline*}
The second summation is taken in the order $\xi_1^*,\xi_3^*,\xi_2^*,\xi_4^*,\dots,\xi_{n+m+2}^*$.
\end{proof}

\subsubsection{Near-resonant terms}

\begin{lem} \label{lem: R22 Sobolev estimate}
    Let $s\geq\frac12-\frac{\alpha}{2}$, and $a \leq \min(2s+\alpha-1,\alpha-1)$. Then
    \begin{align*}
        \left\| R^{1}_{2,2}(u)\right\|_{H^{s+a}} \lesssim \|u\|_{H^s}^{3}.
    \end{align*}
\end{lem}
\begin{proof}
Using (\ref{eq: R22 estimate}), we have for $s\geq\frac12-\frac{\alpha}{2}$, and $a \leq 2s+\alpha-1$,
\begin{align*}
    \left\| R^{1}_{2,2}(u)\right\|_{H^{s+a}}
    &\lesssim \left\| |\xi|^{s+a}\sum_{\xi_1+\xi_2+\xi_3=\xi} \mathbbm{1}_{\{\xi=\xi_1\}}|\xi_1^*|^{1-\alpha}\left|\hat{u}(\xi_1)\hat{u}(\xi_2)\hat{u}(\xi_3)\right|\right\|_{\ell^2_{\xi}} \\
    &\lesssim \|u\|_{H^s}\sum_{\xi \in \Z}|\xi|^{a+1-\alpha}\left|\hat{u}(\xi)\right|^2 \\
    &\lesssim \|u\|_{H^s}^3
\end{align*}
\end{proof}

\begin{lem}
For $s>1-\frac{\alpha}{2}$ and $a<2s+\alpha-2$, we have
    \begin{align*}
    \|R^{2}_{2,2}(u)\|_{H^{s+a}}\lesssim \|u\|_{H^{s}}^{3}.
\end{align*}
\end{lem}
\begin{proof}
    This easily follows by Lemma \ref{lem: R2 pointwise bound} (a) and Bernstein's inequality.
\end{proof}

\begin{lem}
    For $n \geq 2$, let $k_0=\cdots=k_n=2$. Then for $s>\frac76-\frac{2}{3}\alpha$ and $a<\min\big(2\alpha-2,3s+2\alpha-\frac{7}{2}\big)$, we have
    \begin{align*}
        \|R_{k_0,\dots,k_n}^{2}(u)\|_{H^{s+a}} \lesssim \|u\|_{H^s}^{n+2}.
    \end{align*}
\end{lem}
\begin{proof}
By duality, we have
\begin{align} \label{eq: R2 dual}
    &\left\|\sum_{\xi_1+\cdots+\xi_{n+2}=\xi}^{\ast} \frac{|\xi_1^*|^{1+s+a}\mathbbm{1}_{\mathcal{R}^2_{n+2}}|\mu_{k_0,\dots,k_n}|}{|\xi_1 \cdots \xi_{n+2}|^{s}} |\hat{u}(\xi_1)\cdots\hat{u}(\xi_{n+2})|\right\|_{\ell^{2}_{\xi}} \nonumber \\
    &=\sup_{\|v\|_{L^2}=1}\sum_{\xi_1,\dots,\xi_{n+2}}\frac{|\xi_1^*|^{1+s+a}\mathbbm{1}_{\mathcal{R}^2_{n+2}}|\mu_{k_0,\dots,k_n}|}{|\xi_1 \cdots \xi_{n+2}|^{s}} |\hat{u}(\xi_1)\cdots\hat{u}(\xi_{n+2})\hat{v}(\xi_1+\cdots+\xi_{n+2})|.
\end{align}
We need to show that (\ref{eq: R2 dual}) is bounded by $\|u\|_{L^2}^{n+2}$. For $\xi:=\xi_1+\cdots+\xi_{n+2}$, let 
\begin{align*}
    R:=\min\left(\lb\xi-\xi_1^*\rb,\lb\xi-\xi_1^*-\xi_2^*\rb,\lb\xi-\xi_1^*-\xi_3^*\rb,\dots,\lb\xi-\xi_1^*-\xi_{n+2}^*\rb\right).
\end{align*}
Using Lemma \ref{lem: R2 pointwise bound} (b), we have
\begin{align*}
     \frac{|\xi_1^*|^{1+s+a}|\mathbbm{1}_{\mathcal{R}^2_{n+2}}\mu_{k_0,\dots,k_n}|}{|\xi_1 \cdots \xi_{n+2}|^{s}}
    &\lesssim \frac{|\xi_1^*|^{s+a-n(\alpha-1)}|\xi_3^*|}{|\xi_1 \cdots \xi_{n+2}|^{s}R} \\
    &\lesssim \frac{|\xi_3^*||\xi_1^*|^{a-2(\alpha-1)+}}{|\xi_2^*\xi_3^*\xi_4^*|^{s}|\xi_5^*\cdots \xi_{n+2}^*|^{s+\alpha-1}R^{1+}}.
\end{align*}

\noindent\textbf{Case 1: $R \neq \lb\xi-\xi_1^*-\xi_2^*\rb$ and $R \neq \lb\xi-\xi_1^*-\xi_3^*\rb$.}

For $s>\frac76-\frac{2}{3}\alpha$ and $a<3s+2\alpha-\frac{7}{2}$, we have
\begin{align*}
    \frac{|\xi_3^*||\xi_1^*|^{a-2(\alpha-1)+}}{|\xi_2^*\xi_3^*\xi_4^*|^{s}} \lesssim \frac{1}{|\xi_4^*|^{\frac12+}},
\end{align*}
hence
\begin{align*}
    \frac{|\xi_1^*|^{1+s+a}|\mathbbm{1}_{\mathcal{R}^2_{n+2}}\mu_{k_0,\dots,k_n}|}{|\xi_1 \cdots \xi_{n+2}|^{s}} \lesssim \frac{1}{|\xi_4^*|^{\frac12+}|\xi_5^*\cdots \xi_{n+2}^*|^{s+\alpha-1}R^{1+}}.
\end{align*}
By Cauchy-Schwarz, (\ref{eq: R2 dual}) is less than
\begin{multline*}
    \left(\sum_{\xi_1,\dots,\xi_{n+2}}\frac{|\hat{v}(\xi^*_1+\cdots+\xi^*_{n+2})\hat{u}(\xi^*_3)\cdots\hat{u}(\xi^*_{n+2})|^2}{R^{1+}}\right)^{\frac12} \\
    \times \left(\sum_{\xi_1,\dots,\xi_{n+2}}^{\ast}\frac{|\hat{u}(\xi^*_1)\hat{u}(\xi^*_2)|^{2}}{|\xi^*_4|^{1+}|\xi^*_5\cdots\xi^*_{n+2}|^{2s+2\alpha-2}R^{1+}} \right)^{\frac12}=:A\times B.
\end{multline*}
We may sum $A$ in the order $\xi_1^*, \xi_2^*, \xi_3^*,\dots,\xi_{n+2}^*$, and sum $B$ in the order $\xi_3^*, \xi_1^*, \xi_2^*,\xi_4^*,\xi_5^*\dots,\xi_{n+2}^*$.

\noindent\textbf{Case 2: $R = \lb\xi-\xi_1^*-\xi_2^*\rb$ or $R = \lb\xi-\xi_1^*-\xi_3^*\rb$.}
Since $|\xi_2^*|\approx|\xi_3^*|$, by possibly switching $|\xi_2^*|$ and $|\xi_3^*|$, we may only consider the case $R = \lb\xi-\xi_1^*-\xi_2^*\rb=\lb\xi_3^*+\xi_4^*+\cdots+\xi_{n+2}^*\rb$.

\noindent\underline{Subcase 2a: $R \gtrsim |\xi_2^*|$.}
In this case, for $s>\frac76-\frac{2}{3}\alpha$ and $a<3s+2\alpha-\frac{7}{2}$,
\begin{align*}
     \frac{|\xi_1^*|^{1+s+a}|\mathbbm{1}_{\mathcal{R}^2_{n+2}}\mu_{k_0,\dots,k_n}|}{|\xi_1 \cdots \xi_{n+2}|^{s}} 
     \lesssim \frac{|\xi_1^*|^{a-2(\alpha-1)}}{|\xi_2^*\xi_3^*\xi_4^*|^{s}|\xi_5^*\cdots \xi_{n+2}^*|^{s+\alpha-1}} \lesssim \frac{1}{|\xi_2^*\xi_3^*\xi_4^*|^{\frac12+}|\xi_5^*\cdots \xi_{n+2}^*|^{s+\alpha-1}}.
\end{align*}
By Cauchy-Schwarz, (\ref{eq: R2 dual}) is less than
\begin{align*}
    \left(\sum_{\xi_1,\dots,\xi_{n+2}}|\hat{v}(\xi^*_1+\cdots+\xi^*_{n+2})\hat{u}(\xi^*_2)\cdots\hat{u}(\xi^*_{n+2})|^2\right)^{\frac12}
    \times \left(\sum_{\xi_1,\dots,\xi_{n+2}}^{\ast}\frac{|\hat{u}(\xi^*_1)|^{2}}{|\xi_2^*\xi_3^*\xi_4^*|^{\frac12+}|\xi_5^*\cdots \xi_{n+2}^*|^{s+\alpha-1}} \right)^{\frac12}.
\end{align*}

\noindent\underline{Subcase 2b: $R \ll |\xi_2^*|$.}
In this case, we must have $|\xi_2^*|\approx|\xi_4^*|$. Hence, for $s>\frac76-\frac{2}{3}\alpha$ and $a<3s+2\alpha-\frac{7}{2}$, we have
\begin{align*}
    \frac{|\xi_3^*||\xi_1^*|^{a-2(\alpha-1)+}}{|\xi_2^*\xi_3^*\xi_4^*|^{s}} \lesssim \frac{1}{|\xi_2^*|^{\frac12+}},
\end{align*}
Therefore,
\begin{align*}
    \frac{|\xi_1^*|^{1+s+a}|\mathbbm{1}_{\mathcal{R}^2_{n+2}}\mu_{k_0,\dots,k_n}|}{|\xi_1 \cdots \xi_{n+2}|^{s}} \lesssim \frac{1}{|\xi_2^*|^{\frac12+}|\xi_5^*\cdots \xi_{n+2}^*|^{s+\alpha-1}\lb\xi_3^*+\xi_4^*+\cdots+\xi_{n+2}^*\rb^{1+}}.
\end{align*}
By Cauchy-Schwarz, the summation in (\ref{eq: R2 dual}) is less than
\begin{multline*}
    \left(\sum_{\xi_1,\dots,\xi_{n+2}}\frac{|\hat{v}(\xi^*_1+\cdots+\xi^*_{n+2})\hat{u}(\xi^*_2)\hat{u}(\xi^*_4)\cdots\hat{u}(\xi^*_{n+2})|^2}{\lb\xi_3^*+\xi_4^*+\cdots+\xi_{n+2}^*\rb^{1+}}\right)^{\frac12} \\
    \times \left(\sum_{\xi_1,\dots,\xi_{n+2}}^{\ast}\frac{|\hat{u}(\xi^*_1)\hat{u}(\xi^*_3)|^{2}}{|\xi^*_2|^{1+}|\xi^*_5\cdots\xi^*_{n+2}|^{2s+2\alpha-2}\lb\xi_3^*+\xi_4^*+\cdots+\xi_{n+2}^*\rb^{1+}} \right)^{\frac12}=:A\times B.
\end{multline*}
We may sum $A$ in the order  $\xi_1^*, \xi_3^*, \xi_2^*,\xi_4^*,\dots,\xi_{n+2}^*$, and sum $B$ in the order $\xi_1^*, \xi_2^*, \xi_4^*,\xi_3^*,\xi_5^*\dots,\xi_{n+2}^*$.
\end{proof}

\begin{lem} \label{lem: mathfrak R energy estimate deg 2}
Let $m,n \geq 0$, $k_0=\cdots=k_n=2$ and $l_1=\cdots=l_m=2$. Let $n+m \geq 2$. Suppose that $s>1-\frac{\alpha}{2}$ and $a<2s-2+\frac{\alpha}{2}$. Then we have
\begin{align*}
    \left \|\mathfrak{R}_{k_0,\dots,k_n}^{l_1,\dots,l_m}(u)\right\|_{H^{s+a}} \lesssim \|u\|_{H^{s}_{x}}^{n+m+2}
\end{align*}
\end{lem}

\begin{proof}
We need to dominate the quantity
\begin{align} \label{eq: mathfrak R ineq}
    \left\|\sum_{\xi_1+\cdots+\xi_{n+m+2}=\xi}^{\ast}\frac{|\xi|^{1+s+a}\mathbbm{1}_{\mathcal{D}^{1}_{n+m+2}} \big|\mathfrak{m}_{k_0,\dots,k_n}^{l_1,\dots,l_m}(\xi_1,\dots,\xi_{n+m+2})\big|}{|\xi_1\cdots\xi_{n+m+2}|^{s}}|\hat{u}(\xi_1)\cdots\hat{u}(\xi_{n+m+2})|\right\|_{\ell^{2}_{\xi}}
\end{align}
by $\|u\|_{L^2}^{n+m+2}$. Let
\begin{align*}
&\mathcal{A}:=\{(\xi_1,\dots,\xi_{n+m+2}) \in\Z_{\ast}^{n+m+2}: |\xi_1^*| \approx |\xi_2^*| \approx |\xi_3^*|\}, \\
&\mathcal{B}:=\{(\xi_1,\dots,\xi_{n+m+2}) \in\Z_{\ast}^{n+m+2}:  |\xi_1^*| \approx |\xi_2^*| \gg |\xi_3^*| \gtrsim|\xi|\}.
\end{align*}
Then by Lemma \ref{lem: higher-order resonances} and the definition of $\mathcal{D}^{1}_{n+m+2}$, we have
\begin{align*}
    \supp\left(\mathbbm{1}_{\mathcal{D}^{1}_{n+m+2}}\mathfrak{m}_{k_0,\dots,k_n}^{l_1,\dots,l_m}\right)\subseteq \mathcal{A} \cup \mathcal{B}.
\end{align*}
Let $R:=\min_{1\leq j \leq n+m+1}\mathbf{X}^{j}_{2}\left(\rho_{n+m+1}\right)(\xi_1,\dots,\xi_{n+m+2})$. By Lemma \ref{lem: mu pointwise bound} (c), we have 
\begin{align*}
    \left|\mathbbm{1}_{\mathcal{D}^{1}_{n+m+2}} \mathfrak{m}_{k_0,\dots,k_n}^{l_1,\dots,l_m}(\xi_1,\dots,\xi_{n+m+2})\right| \lesssim \frac{1}{R}\prod_{i=3}^{n+m+2}\zeta_i^{-(\alpha-1)},
\end{align*}
where we define $\zeta_i:=\max\left(|\xi_i^*|,|\xi|\right)$.

\noindent\underline{\textbf{CASE A: $(\xi_1,\dots,\xi_{n+m+2}) \in \mathcal{A}$.}}

Up to the permutation of $\xi_1^*,\xi_2^*,\xi_3^*$, one of the followings hold: 
\begin{align*}
    &\textnormal{\textbf{Case A1: }} R=\lb \xi_{1}^*+\xi_2^*+\xi_3^* \rb. \\
    &\textnormal{\textbf{Case A2: }} R=\lb \xi_{1}^*+\xi_2^*+\xi_k^* \rb \textnormal{ for some } 4 \leq k \leq n+m+2.\\
    &\textnormal{\textbf{Case A3: }} R=\lb \xi_{3}^*+\xi_{4}^*+\cdots+\xi_{n+m+2}^* \rb. \\
    &\textnormal{\textbf{Case A4: }} R=\lb \xi_{1}^*+\xi_{2}^*+\xi_{4}^*+\cdots+\xi_{n+m+2}^* \rb. \\
    &\textnormal{\textbf{Case A5: }} R \gtrsim |\xi_1^*|.
\end{align*}

\noindent\textbf{Case A1: $R=\lb \xi_{1}^*+\xi_2^*+\xi_3^* \rb$.}

\noindent\underline{Subcase A1a: $\lb \xi_{1}^*+\xi_2^*+\xi_3^* \rb \ll |\xi|$.}
In this case, we have $|\xi| \lesssim |\xi_4^*|$. Hence (\ref{eq: mathfrak R ineq}) is less than
\begin{align*}
    \left\|\sum_{\xi_1+\cdots+\xi_{n+m+2}=\xi}^{\ast}\frac{\mathbbm{1}_{\mathcal{A}}|\xi_4^*|^{1+s+a}|\zeta_{3}\cdots\zeta_{n+m+2}|^{-(\alpha-1)}}{|\xi_1\cdots\xi_{n+m+2}|^{s}\lb \xi_{1}^*+\xi_2^*+\xi_3^* \rb}|\hat{u}(\xi_1)\cdots\hat{u}(\xi_{n+m+2})|\right\|_{\ell^{2}_{\xi}},
\end{align*}
which by duality equals
\begin{align} \label{eq: 123 dual 1a}
    \sup_{\|v\|_{L^2}=1} \sum_{\xi_1,\dots,\xi_{n+m+2}}^{\ast}\frac{\mathbbm{1}_{\mathcal{A}}|\xi_4^*|^{1+s+a}|\zeta_{3}\cdots\zeta_{n+m+2}|^{-(\alpha-1)}}{|\xi_1\cdots\xi_{n+m+2}|^{s}\lb \xi_{1}^*+\xi_2^*+\xi_3^* \rb}\left|\hat{u}(\xi_1)\cdots\hat{u}(\xi_{n+m+2})\hat{v}(\xi_1+\cdots+\xi_{n+m+2})\right|.
\end{align}
Note that
\begin{align*}
    \frac{\mathbbm{1}_{\mathcal{A}}|\xi_4^*|^{1+s+a}|\zeta_{3}\cdots\zeta_{n+m+2}|^{-(\alpha-1)}}{|\xi_1\cdots\xi_{n+m+2}|^{s}\lb \xi_{1}^*+\xi_2^*+\xi_3^* \rb}
    &\lesssim \frac{|\xi_4^*|^{2-\alpha+a}}{|\xi_3^*|^{3s+\alpha-1-}|\xi_5^*\xi_6^*\cdots\xi_{n+m+2}^*|^{s+\alpha-1}\lb \xi_{1}^*+\xi_2^*+\xi_3^* \rb^{1+}} \\
    &\lesssim \frac{1}{|\xi_3^*|^{3s+2\alpha-3-a-}|\xi_5^*\xi_6^*\cdots\xi_{n+m+2}^*|^{s+\alpha-1}\lb \xi_{1}^*+\xi_2^*+\xi_3^* \rb^{1+}}.
\end{align*}
By Cauchy-Schwarz, the summation in (\ref{eq: 123 dual 1a}) is less than
\begin{multline*}
    \left(\sum_{\xi_1,\dots,\xi_{n+m+2}}^{\ast}\frac{|\hat{u}(\xi_1^*)\hat{u}(\xi_3^*)\hat{u}(\xi_4^*)\cdots\hat{u}(\xi_{n+m+2}^*)|^2}{\lb \xi_1^*+\xi_2^*+\xi_3^*\rb^{1+}} \right)^{\frac12} \\
    \times \left(\sum_{\xi_1,\dots,\xi_{n+m+2}}^{\ast}\frac{|\hat{v}(\xi_1^*+\cdots+\xi_{n+m+2}^*)\hat{u}(\xi_2^*)|^{2}}{|\xi_3^*|^{3s+2\alpha-3-a-}\lb \xi_1^*+\xi_2^*+\xi_3^*\rb^{1+}|\xi_5^*\cdots\xi_{n+m+2}^*|^{2s+2\alpha-2}} \right)^{\frac12}=:A\times B.
\end{multline*}
We sum $A$ in the order $\xi_2^*,\xi_1^*,\xi_3^*,\xi_4^*, \dots ,\xi_{n+m+2}^*$, and sum $B$ in the order $\xi_4^*,\xi_1^*,\xi_2^*,\xi_3^*,\xi_5^*, \dots ,\xi_{n+m+2}^*$. We need $s>\frac76-\frac{2}{3}\alpha$ and $a<3s+2\alpha-\frac{7}{2}$.

\noindent\underline{Subcase A1b: $\lb \xi_{1}^*+\xi_2^*+\xi_3^* \rb \gtrsim |\xi|$.}
In this case, (\ref{eq: mathfrak R ineq}) is less than 
\begin{align} \label{eq: 123 dual 1b}
    \sup_{\|v\|_{L^2}=1} \sum_{\xi_1,\dots,\xi_{n+m+2}}^{\ast}\frac{\mathbbm{1}_{\mathcal{A}}|\xi|^{s+a}|\zeta_{3}\cdots\zeta_{n+m+2}|^{-(\alpha-1)}}{|\xi_1\cdots\xi_{n+m+2}|^{s}}\left|\hat{u}(\xi_1)\cdots\hat{u}(\xi_{n+m+2})\hat{v}(\xi_1+\cdots+\xi_{n+m+2})\right|.
\end{align}
Since
\begin{align*}
    \frac{\mathbbm{1}_{\mathcal{A}}|\xi|^{s+a}|\zeta_{3}\cdots\zeta_{n+m+2}|^{-(\alpha-1)}}{|\xi_1\cdots\xi_{n+m+2}|^{s}}
    \lesssim \frac{1}{|\xi_2^*\xi_3^*|^{s+\frac{\alpha}{2}-\frac{1}{2}-\frac{a}{2}}|\xi_4^*\cdots\xi_{n+m+2}^*|^{s+\alpha-1}},
\end{align*}
By Cauchy-Schwarz, the summation in (\ref{eq: 123 dual 1a}) is less than
\begin{align*}
    \left(\sum_{\xi_1,\dots,\xi_{n+m+2}}^{\ast}|\hat{u}(\xi_1)\cdots\hat{u}(\xi_{n+m+2})|^2 \right)^{\frac12} 
    \times \left(\sum_{\xi_1,\dots,\xi_{n+m+2}}^{\ast}\frac{|\hat{v}(\xi_1+\cdots+\xi_{n+m+2})|}{|\xi_2^*\xi_3^*|^{2s+\alpha-1-a}|\xi_4^*\cdots\xi_{n+m+2}^*|^{s+\alpha-1}} \right)^{\frac12}.
\end{align*}
For summability, we need $s>1-\frac{\alpha}{2}$ and $a<2s+\alpha-2$.

\noindent\textbf{Case A2: $R=\lb \xi_{1}^*+\xi_2^*+\xi_k^* \rb$ for some $4 \leq k \leq n+m+2$.}

By duality,(\ref{eq: mathfrak R ineq}) is less than
\begin{align} \label{eq: 123 dual case2}
    \sup_{\|v\|_{L^2}=1} \sum_{\xi_1,\dots,\xi_{n+m+2}}^{\ast}\frac{\mathbbm{1}_{\mathcal{A}}|\xi|^{1+s+a}|\zeta_3\cdots\zeta_{n+m+2}|^{-(\alpha-1)}}{|\xi_1\cdots\xi_{n+m+2}|^{s}\lb \xi_{1}^*+\xi_2^*+\xi_k^* \rb}\left|\hat{u}(\xi_1)\cdots\hat{u}(\xi_{n+m+2})\hat{v}(\xi_1+\cdots+\xi_{n+m+2})\right|.
\end{align}
Also, note that for $s>1-\frac{\alpha}{2}$ and $a<2s+\alpha-2$,
\begin{align*}
    \frac{\mathbbm{1}_{\mathcal{A}}|\xi|^{1+s+a}|\zeta_3\cdots\zeta_{n+m+2}|^{-(\alpha-1)}}{|\xi_1\cdots\xi_{n+m+2}|^{s}\lb \xi_{1}^*+\xi_2^*+\xi_j^* \rb}
    &\lesssim \frac{|\xi_1^*|^{2-\alpha-2s+a+}}{|\xi_1^*|^{2s+\alpha-2-a-}|\xi_4^*\cdots\xi_{n+m+2}^*|^{s+\alpha-1}\lb \xi_{1}^*+\xi_2^*+\xi_k^* \rb^{1+}} \\
    &\lesssim \frac{1}{|\xi_4^*\cdots\xi_{n+m+2}^*|^{s+\alpha-1}\lb \xi_{1}^*+\xi_2^*+\xi_k^* \rb^{1+}}.
\end{align*}
By Cauchy-Schwarz, the summation in (\ref{eq: 123 dual case2}) is less than
\begin{multline*}
    \left(\sum_{\xi_1,\dots,\xi_{n+m+2}}^{\ast}\frac{|\hat{u}(\xi_1^*)\hat{u}(\xi_3^*)\hat{u}(\xi_4^*)\cdots\hat{u}(\xi_{n+m+2}^*)|^2}{\lb \xi_1^*+\xi_2^*+\xi_k^*\rb^{1+}} \right)^{\frac12} \\
    \times \left(\sum_{\xi_1,\dots,\xi_{n+m+2}}^{\ast}\frac{|\hat{v}(\xi_1^*+\cdots+\xi_{n+m+2}^*)\hat{u}(\xi_2^*)|^{2}}{\lb \xi_1^*+\xi_2^*+\xi_k^*\rb^{1+}|\xi_4^*\cdots\xi_{n+m+2}^*|^{2s+2\alpha-2}} \right)^{\frac12}=:A\times B.
\end{multline*}
We sum $A$ in the order $\xi_2^*,\xi_1^*,\xi_3^*,\xi_4^*, \dots ,\xi_{n+m+2}^*$, and sum $B$ in the order $\xi_3^*,\xi_1^*,\xi_2^*,\xi_4^*, \dots ,\xi_{n+m+2}^*$.

\noindent\textbf{Case A3: $R=\lb \xi_{3}^*+\xi_{4}^*+\cdots+\xi_{n+m+2}^* \rb$.}

If $|\xi_3^*|\gg|\xi_4^*|$, then $R \approx |\xi_1^*|$, hence we can absorb this into Case A5. Assume that $|\xi_3^*|\approx|\xi_4^*|$. Proceeding as in Subcase A1a, we obtain
\begin{multline*}
    (\ref{eq: mathfrak R ineq}) \lesssim \left(\sum_{\xi_1,\dots,\xi_{n+m+2}}^{\ast}\frac{|\hat{u}(\xi_1^*)\hat{u}(\xi_2^*)\hat{u}(\xi_4^*)\cdots\hat{u}(\xi_{n+m+2}^*)|^2}{\lb \xi_{3}^*+\xi_{4}^*+\cdots+\xi_{n+m+2}^* \rb^{1+}} \right)^{\frac12} \\
    \times \left(\sum_{\xi_1,\dots,\xi_{n+m+2}}^{\ast}\frac{|\hat{v}(\xi_1^*+\cdots+\xi_{n+m+2}^*)\hat{u}(\xi_3^*)|^{2}}{|\xi_2^*|^{3s+2\alpha-3-a-}\lb \xi_{3}^*+\xi_{4}^*+\cdots+\xi_{n+m+2}^*\rb^{1+}|\xi_5^*\cdots\xi_{n+m+2}^*|^{2s+2\alpha-2}} \right)^{\frac12}.
\end{multline*}
The second summation is taken in the order $\xi_1^*,\xi_2^*,\xi_4^*,\xi_3^*,\xi_5^*,\dots,\xi_{n+m+2}^*$.

\noindent\textbf{Case A4: $R=\lb \xi_{1}^*+\xi_{2}^*+\xi_{4}^*+\cdots+\xi_{n+m+2}^* \rb$.}

This case is similar to Case A2.

\noindent\textbf{Case A5: $R \gtrsim |\xi_1^*|$.}

This case is the simplest. We omit details.

\noindent\underline{\textbf{CASE B: $(\xi_1,\dots,\xi_{n+m+2}) \in \mathcal{B}$.}}

In this case, one of the followings hold: 
\begin{align*}
    &\textnormal{\textbf{Case B1: }} R=\lb \xi_{1}^*+\xi_2^*+\xi_3^* \rb. \\
    &\textnormal{\textbf{Case B2: }} R=\lb \xi_{1}^*+\xi_2^*\rb \textnormal{ or } \lb \xi-\xi_3^* \rb \textnormal{ or } \lb \xi_1^*+\xi_2^*+\xi_k^* \rb \textnormal{ for some } 4 \leq k \leq n+m+2. \\
    &\textnormal{\textbf{Case B3: }} R=\lb \xi_{3}^*+\xi_{4}^*\rb \textnormal{ or } \lb \xi_{3}^*+\xi_{4}^*+\cdots+\xi_{n+m+2}^* \rb. \\
    &\textnormal{\textbf{Case B4: }} R \gtrsim |\xi_1^*|.
\end{align*}
Also, (\ref{eq: mathfrak R ineq}) is less than
\begin{align} \label{eq: 1B}
    \sup_{\|v\|_{L^2}=1} \sum_{\xi_1,\dots,\xi_{n+m+2}}^{\ast}\frac{\mathbbm{1}_{\mathcal{B}}|\xi_4^*|^{1+s+a}|\zeta_{3}\cdots\zeta_{n+m+2}|^{-(\alpha-1)}}{|\xi_1\cdots\xi_{n+m+2}|^{s}\lb \xi_{1}^*+\xi_2^*+\xi_3^* \rb}\left|\hat{u}(\xi_1)\cdots\hat{u}(\xi_{n+m+2})\hat{v}(\xi_1+\cdots+\xi_{n+m+2})\right|.
\end{align}

\noindent\textbf{Case B1: $R=\lb \xi_{1}^*+\xi_2^*+\xi_3^* \rb$.}

\noindent\underline{Subcase B1a: $|\xi| \lesssim |\xi_4^*|$.}
In this subcase, we can proceed as in Subcase A1a.

\noindent\underline{Subcase B1b: $|\xi|\gg|\xi_4^*|$.}
In this subcase, we have $R \gtrsim |\xi|$. Hence for $s>1-\frac{\alpha}{2}$ and $a<2s+\alpha-2$,
\begin{align*}
    \frac{\mathbbm{1}_{\mathcal{B}}|\xi|^{1+s+a}|\zeta_{3}\cdots\zeta_{n+m+2}|^{-(\alpha-1)}}{|\xi_1\cdots\xi_{n+m+2}|^{s}R}
    \lesssim \frac{|\xi_1^*|^{-2s}|\xi|^{s+a-\alpha+1}}{|\xi_3^*|^{s}|\xi_4^* \cdots \xi_{n+m+2}^*|^{s+\alpha-1}}
    \lesssim \frac{1}{|\xi \xi_3^*|^{\frac12+}|\xi_4^* \cdots \xi_{n+m+2}^*|^{s+\alpha-1}}.
\end{align*}
Hence the sum in (\ref{eq: 1B}) is dominated by
\begin{align*}
    \left(\sum_{\xi_1,\dots,\xi_{n+m+2}}^{\ast}\frac{|\hat{u}(\xi_1^*)\hat{u}(\xi_3^*)\cdots\hat{u}(\xi_{n+m+2}^*)|^2}{\lb \xi \rb^{1+}} \right)^{\frac12}
    \times \left(\sum_{\xi_1,\dots,\xi_{n+m+2}}^{\ast}\frac{|\hat{v}(\xi_1^*+\cdots+\xi_{n+m+2}^*)\hat{u}(\xi_2^*)|^{2}}{|\xi_3^*|^{1+}|\xi_4^*\cdots\xi_{n+m+2}^*|^{2s+2\alpha-2}} \right)^{\frac12}.
\end{align*}
We take the first summation in the order $\xi_2^*,\xi_1^*,\xi_3^*,\xi_4^*,\dots,\xi_{n+m+2}^*$.

\noindent\textbf{Case B2: $R=\lb \xi_{1}^*+\xi_2^*\rb$ or $\lb \xi-\xi_3^* \rb$ or $\lb \xi_1^*+\xi_2^*+\xi_{k}^*\rb$ for $4 \leq k \leq n+m+2$.}

We have
\begin{align*}
    \frac{\mathbbm{1}_{\mathcal{B}}|\xi|^{1+s+a}|\zeta_{3}\cdots\zeta_{n+m+2}|^{-(\alpha-1)}}{|\xi_1\cdots\xi_{n+m+2}|^{s}R}
    \lesssim \frac{|\xi_1^*|^{-2s}|\xi_3^*|^{3-2\alpha}}{|\xi_4^*|^{s}|\xi_5^* \cdots \xi_{n+m+2}^*|^{s+\alpha-1}R}.
\end{align*}
Therefore, if $|\xi_1^*|^{-2s}|\xi_3^*|^{3-2\alpha}|\xi_4^*|^{-s} \lesssim |\xi_4^*|^{-\frac12-}$, then we can dominate the sum in (\ref{eq: 1B}) by
\begin{multline*}
    \left(\sum_{\xi_1,\dots,\xi_{n+m+2}}^{\ast}\frac{|\hat{u}(\xi_1^*)\hat{u}(\xi_3^*)\cdots\hat{u}(\xi_{n+m+2}^*)|^2}{R^{1+}} \right)^{\frac12} \\
    \times \left(\sum_{\xi_1,\dots,\xi_{n+m+2}}^{\ast}\frac{|\hat{v}(\xi_1^*+\cdots+\xi_{n+m+2}^*)\hat{u}(\xi_2^*)|^{2}}{R^{1+}|\xi_4^*|^{1+}|\xi_5^*\cdots\xi_{n+m+2}^*|^{2s+2\alpha-2}} \right)^{\frac12}=:A\times B.
\end{multline*}
We $A$ in the order $\xi_2^*,\xi_1^*,\xi_3^*,\xi_4^*,\dots,\xi_{n+m+2}^*$, and then sum $B$ in the order $\xi_3^*,\xi_1^*,\xi_2^*,\xi_4^*,\dots,\xi_{n+m+2}^*$.

\noindent\textbf{Case B3: $R=\lb \xi_{3}^*+\xi_4^*+\cdots+\xi_{n+m+2}^* \rb$ or $\lb \xi_3^*+\xi_4^*\rb$.}

\noindent\underline{Subcase B3a: $|\xi_3| \gg |\xi_4^*|$.} We have $R \gtrsim |\xi|$ in this case. We can proceed as in Subcase B1b.

\noindent\underline{Subcase B3b: $|\xi_3| \approx |\xi_4^*|$.}

If $s>\frac76-\frac{2\alpha}{3}$ and $a<3s+2\alpha-\frac72$, we have
\begin{align*}
    \frac{\mathbbm{1}_{\mathcal{B}}|\xi|^{1+s+a}|\zeta_{3}\cdots\zeta_{n+m+2}|^{-(\alpha-1)}}{|\xi_1\cdots\xi_{n+m+2}|^{s}R}
    &\lesssim \frac{|\xi|^{-s+a+3-2\alpha}}{|\xi_2^*|^{2s-}|\xi_5^*\xi_6^*\cdots\xi_{n+m+2}^*|^{s+\alpha-1}R^{1+}} \\
    &\lesssim \frac{1}{|\xi|^{\frac12+}|\xi_5^*\xi_6^*\cdots\xi_{n+m+2}^*|^{s+\alpha-1}R^{1+}}.
\end{align*}
Hence the summation in (\ref{eq: 1B}) is dominated by
\begin{multline*}
    \left(\sum_{\xi_1,\dots,\xi_{n+m+2}}^{\ast}\frac{|\hat{u}(\xi_1^*)\hat{u}(\xi_4^*)\cdots\hat{u}(\xi_{n+m+2}^*)|^2}{\lb \xi \rb^{1+}R^{1+}} \right)^{\frac12} \\
    \times \left(\sum_{\xi_1,\dots,\xi_{n+m+2}}^{\ast}\frac{|\hat{v}(\xi_1^*+\cdots+\xi_{n+m+2}^*)\hat{u}(\xi_2^*)\hat{u}(\xi_3^*)|^{2}}{R^{1+}|\xi_5^*\cdots\xi_{n+m+2}^*|^{2s+2\alpha-2}} \right)^{\frac12}=:A\times B.
\end{multline*}
We sum $A$ in the order $\xi_2^*,\xi_1^*,\xi_3^*,\xi_4^*, \dots ,\xi_{n+m+2}^*$, and sum $B$ in the order $\xi_1^*,\xi_4^*,\xi_2^*,\xi_3^*,\xi_5^* \dots ,\xi_{n+m+2}^*$.

\noindent\textbf{Case B4: $R \gtrsim |\xi_1^*|$.}

This is the simplest case. We omit details.
\end{proof}

\subsubsection{Non-resonant terms}

\begin{lem}
    Let $m,n \geq 0$, $k_0=\cdots=k_n=2$ and $l_1=\cdots=l_m=2$. For $s>\max\big(1-\frac{\alpha}{2},\frac{3}{2}-\frac{n+m}{2}(\alpha-1)\big)$ and $a<s+\frac{n+m}{2}(\alpha-1)-\frac32$, we have
    \begin{align*}
        \|\mathfrak{N}_{k_0,\dots,k_n}^{l_1,\dots,l_m}(u)\|_{H^{s+a}} \lesssim \|u\|_{H^s}^{n+m+2}.
    \end{align*}
\end{lem}

\begin{proof}
We need to show that
\begin{align} \label{eq: mfN dual}
    \left\|\sum_{\xi_1+\cdots+\xi_{n+m+2}=\xi}^{\ast} \frac{|\xi|^{1+s+a}\mathbbm{1}_{\mathcal{D}^2_{n+m+2}}\big|\mathfrak{m}_{k_0,\dots,k_n}^{l_1,\dots,l_m}\big|}{|\xi_1 \cdots \xi_{n+m+2}|^{s}} |\hat{u}(\xi_1)\cdots\hat{u}(\xi_{n+2})|\right\|_{\ell^{2}_{\xi}} \lesssim \|u\|_{L^2}^{n+m+2}.
\end{align}
By Lemma \ref{lem: mu pointwise bound} (c), for $\zeta_i:=\max\left(|\xi_i^*|,|\xi|\right)$ we have
\begin{align*}
    \frac{|\xi|^{1+s+a}\mathbbm{1}_{\mathcal{D}^2_{n+m+2}}\big|\mathfrak{m}_{k_0,\dots,k_n}^{l_1,\dots,l_m}\big|}{|\xi_1 \cdots \xi_{n+2}|^{s}}
    &\lesssim \frac{\lb \xi\rb^{\frac32-s+a+}|\zeta_3\cdots\zeta_{n+m+2}|^{-(\alpha-1)}}{\lb\xi\rb^{\frac12+}|\xi_3^* \cdots \xi_{n+m+2}^*|^{s}}  \\
    &\lesssim \frac{\lb \xi\rb^{\frac32-s+a+}|\zeta_3\cdots\zeta_{n+m+2}|^{\frac12-\frac{\alpha}{2}}}{\lb \xi\rb^{\frac12+}|\xi_3^* \cdots \xi_{n+m+2}^*|^{s+\frac{\alpha}{2}-\frac{1}{2}}} \\
    &\lesssim \frac{\lb \xi\rb^{\frac{3}{2}-s+a+(n+m)(\frac12-\frac{\alpha}{2})+}}{\lb \xi\rb^{\frac12+}|\xi_3^* \cdots \xi_{n+m+2}^*|^{s+\frac{\alpha}{2}-\frac12}}.
\end{align*}
Therefore, if $\frac{3}{2}-s+a+(n+m)(\frac12-\frac{\alpha}{2})<0$, the left-hand side of (\ref{eq: mfN dual}) is less than
\begin{multline*}
    \sup_{\|v\|_{L^2}=1}\left(\sum_{\xi_1,\dots,\xi_{n+m+2}}^{\ast}\frac{|\hat{u}(\xi_1^*)\hat{u}(\xi_3^*)\cdots\hat{u}(\xi_{n+m+2}^*)|^{2}}{\lb \xi \rb^{1+}} \right)^{\frac12} \\
    \times \left(\sum_{\xi_1,\dots,\xi_{n+m+2}}^{\ast}\frac{|\hat{v}(\xi_1^*+\cdots+\xi_{n+m+2}^*)\hat{u}(\xi_2^*)|^{2}}{|\xi_3^*\cdots\xi_{n+m+2}^*|^{2s+\alpha-1}} \right)^{\frac12}.
\end{multline*}
The second sum is finite if $s>1-\frac{\alpha}{2}$. This concludes the proof.
\end{proof}

\subsection{Energy estimates} \label{subsection: Energy estimates}

In this subsection, we estimate the near-resonant terms $R^{2}_{k_0,\dots,k_n}$ and $\mathfrak{R}_{k_0,\dots,k_n}^{l_1,\dots,l_m}$ in the $\textnormal{deg}(P)\geq 3$ case. To estimate these terms in low regularity spaces, we use the following bilinear Strichartz estimate by Molinet and Tanaka:
\begin{prop}{\cite[Proposition 3.1, 3.2]{MT2}}
Let $0<T<1$, $1 \leq \alpha \leq 2$ and  $N_1,N_2 \geq 1$.
For any real-valued $u_1,u_2\in C([-T,T],L^2(\T))$ and $f_1,f_2\in L^{\infty}([-T,T],L^2(\T))$ satisfying
\begin{align*}
    \partial_t u_j + \partial_{x}D^{\alpha}_{x} u_j +\partial_x f_j=0, \quad j=1,2
\end{align*}
on $\T \times [-T,T]$, we have
\begin{multline}\label{eq: bilinear strichartz 1}
    \| P_{N_1} u_1 P_{N_2} u_2\|_{L^{2}_{T}L^{2}_{x}}
    \lesssim \max(N_1,N_2)^{\frac{1}{2}-\frac{\alpha}{4}}\\
    \times(\|P_{ N_1} u_1\|_{L_T^{\infty} L^2_x}+\|P_{ N_1} f_1\|_{L_T^{\infty} L^2_x})(\|P_{N_2} u_2\|_{L_T^{\infty} L^2_x}+\|P_{N_2} f_2\|_{L_T^{\infty} L^2_x}).
\end{multline}
Moreover, if $\max(N_1,N_2) \gg \min(N_1,N_2)$, then we have
\begin{align}\label{eq: bilinear strichartz 2}
    \| P_{N_1} u_1 P_{N_2} u_2\|_{L_T^{2} L^2_x}
    \lesssim 
    (\|P_{ N_1} u_1\|_{L_T^{\infty} L^2_x}+\|P_{ N_1} f_1\|_{L_T^{\infty} L^2_x})(\|P_{N_2} u_2\|_{L_T^{\infty} L^2_x}+\|P_{N_2} f_2\|_{L_T^{\infty} L^2_x}).
\end{align}
\end{prop}

\begin{lem} \label{lem: R2 energy estimate}
Let $0<T<1$, $n \geq 0$, $d \geq 3$, and $(k_0,\dots,k_n)\in[d]^{n+1}$.
\begin{enumerate}[label=(\alph*)]
\item Suppose that $\nu_n \geq 4$,
\begin{align*}
    s>
    \begin{dcases}
        \ \max\left(\frac12,\frac{1}{2}+\frac{n+1}{4\alpha}-\frac{n\alpha}{4}\right), &\, d=3, \\
        \ \max\left(\frac12,\frac12+\frac{n+1}{3\alpha}-\frac{n\alpha}{3}\right), &\, d \geq 4,
    \end{dcases}
\end{align*}
and
\begin{align*}
    a<
    \begin{dcases}
        \ \min\left(\frac{3\alpha}{2(\alpha+1)}(2s-1),\frac{4\alpha}{\alpha+1}s-\frac{2\alpha+1}{\alpha+1}\right)+n(\alpha-1), &\, d=3, \\
        \ \frac{3\alpha}{\alpha+1}s-\frac{3\alpha+2}{2(\alpha+1)}+n(\alpha-1), &\, d \geq 4.
    \end{dcases}
\end{align*}
Let $u$ be a smooth solution to (\ref{eq: guage dgbo}). Then we have
\begin{align} \label{eq: R2 duhamel ineq}
    \left \|\int_{0}^{t} S(t-t') R^{2}_{k_0,\dots,k_n}(u)(t') dt' \right\|_{L^{\infty}_T H^{s+a}_{x}} \lesssim C(\|u\|_{L^{\infty}_T H^{s}_{x}})\|u\|_{L^{\infty}_T H^{s}_{x}}^{\nu_n},
\end{align}
where $C(\|u\|_{L^{\infty}_T H^{s}_{x}})$ is an increasing function of $\|u\|_{L^{\infty}_T H^{s}_{x}}$.
\item If $n=0$ and $k_0=3$, then the estimate (\ref{eq: R2 duhamel ineq}) holds for $s>1-\frac{\alpha}{4}$ and $a<2s+\frac{\alpha}{2}-2$.
\end{enumerate}
\end{lem}
\begin{proof}
First we consider the $d \geq 4$ case in (a). Let
\begin{align*}
    \mathcal{D}(u):=\int_{0}^{t} S(t-t') R^{2}_{k_0,\dots,k_n}(u)(t') dt'.
\end{align*}
Then $\mathcal{D}(u)$ is governed by $\partial_{t}\hat{\mathcal{D}(u)}(\xi)+i\xi|\xi|^{\alpha}\hat{\mathcal{D}(u)}(\xi)=\mathcal{F}[ R^{2}_{k_0,\dots,k_n}(u)](\xi)$. By the fundamental theorem of calculus we have
\begin{align*}
    \|P_N \mathcal{D}(u)(t)\|_{H^{s+a}}^2
    &=2\Re\left(\int_{0}^{t} \sum_{\xi \in \Z \setminus \{0\}} \mathbbm{1}_{\{|\xi|\approx N\}} |\xi|^{2(s+a)} \overline{\hat{\mathcal{D}(u)}(\xi,t')}\partial_{t'}\hat{\mathcal{D}(u)}(\xi,t') dt' \right ) \\
    &=2\Re\left(\int_{0}^{t} \sum_{\xi \in \Z\setminus \{0\}} \mathbbm{1}_{\{|\xi|\approx N\}} |\xi|^{2(s+a)} \overline{\hat{\mathcal{D}(u)}(\xi,t')}\mathcal{F}[ R^{2}_{k_0,\dots,k_n}(u)](\xi,t') dt' \right ) \\
    &=2\Re\left(I_{t}(P_N \mathcal{D}(u),u,\dots,u) \right),
\end{align*}
where
\begin{align*}
    I_{t}(u_0,\dots,u_{\nu_n})
    :=\int_{0}^{t} \sum_{\xi_0+\cdots+\xi_{\nu_n}=0} |\xi_{0}|^{2(s+a)} \xi_{0}\mathbbm{1}_{\mathcal{R}^{2}_{\nu_n}}\mu_{k_0,\dots,k_n}(\xi_1,\dots,\xi_{\nu_n})
    \hat{u}_{0}(\xi_0,t')\cdots \hat{u}_{\nu_n}(\xi_{\nu_n},t') dt'.
\end{align*}
Below we estimate 
\begin{align} \label{eq: R2 energy ineq}
    \sup_{t \in [-T,T]}\left|I_{t}(P_{N_0} \mathcal{D}(u),P_{N_1} u,\dots,P_{N_{\nu_n}}u) \right|.
\end{align}
Assume by symmetry that $N_0 \approx N_1 \gg N_2 \geq N_3 \geq \dots \geq N_{\nu_n}$. Then by Lemma \ref{lem: mu pointwise bound} (a) and Cauchy-Schwarz,
\begin{align*} 
    (\ref{eq: R2 energy ineq})
    \lesssim N_0^{2s+2a+1+n(1-\alpha)}\| P_{N_0} \mathcal{D}(u) P_{N_2} u\|_{L^{2}_{T}L^{2}_{x}}\| P_{N_1} u P_{N_3} u\|_{L^{2}_{T}L^{2}_{x}} \prod_{i=4}^{\nu_n}N_{i}^{\frac12}\|P_{N_i}u\|_{L^{\infty}_{T}L^2_{x}}.
\end{align*}
Using the bilinear Strichartz estimate (\ref{eq: bilinear strichartz 2}), we have
\begin{multline} \label{eq: PN0PN1}
    \| P_{N_0} \mathcal{D}(u) P_{N_2} u\|_{L^2_{T}L^{2}_{x}}
    \lesssim \left(\|P_{N_0} \mathcal{D}(u)\|_{L^{\infty}_{T}L^{2}_{x}}+\|P_{N_0}\partial_{x}^{-1}R^{2}_{k_0,\dots,k_n}(u)\|_{L^{\infty}_{T}L^{2}_{x}}\right)\\
    \times\left(\|P_{N_2} u\|_{L^{\infty}_{T}L^{2}_{x}}+\left\|P_{N_2}\partial_{x}^{-1}\left(\mathbf{P} \left(P'(u) \right) \partial_{x} u \right) \right\|_{L^{\infty}_{T}L^{2}_{x}}\right).
\end{multline}
Since for $w(x,t):=\mathcal{F}^{-1}_{\xi}|\mathcal{F}_{x}[u](\xi,t)|$, we have
\begin{align*}
    \|P_{N_0}\partial_{x}^{-1}R^{2}_{k_0,\dots,k_n}(u)\|_{L^2_{x}}
    \lesssim N_0^{-n(\alpha-1)-1}\|P_{N_0}(w^{\nu_n})\|_{L^2_x},
\end{align*}
and
\begin{align*}
    \left\|P_{N_2}\partial_{x}^{-1}\left( \mathbf{P} \left(P'(u) \right) \partial_{x} u\right) \right\|_{L^{2}_{x}}
    \lesssim \sum_{k=2}^{d}\|P_{N_2}(w^{k})\|_{L^2_x},
\end{align*}
the quantity (\ref{eq: PN0PN1}) is dominated by
\begin{align*}
    N_0^{-s-a}N_2^{-s}\sum_{k=2}^{d}\left( \|P_{N_0} \mathcal{D}(u)\|_{L^{\infty}_{T}H^{s+a}_{x}}+\|P_{N_0}(w^{\nu_n})\|_{L^{\infty}_{T}H^{s}_{x}}\right)
    \times \left(\|P_{N_2} u\|_{L^{\infty}_{T}H^{s}_{x}}+\|P_{N_2} (w^k)\|_{L^{\infty}_{T}H^{s}_{x}}\right).
\end{align*}
Similarly, we have
\begin{multline*}
    \| P_{N_1} u P_{N_3} u\|_{L^2_{T}L^{2}_{x}}
    \lesssim N_1^{-s}N_3^{-s}\sum_{j,k=2}^{d}\left(\|P_{N_1} u\|_{L^{\infty}_{T}H^{s}_{x}}+\|P_{N_1}(w^j)\|_{L^{\infty}_{T}H^{s}_{x}}\right)\\
    \times\left(\|P_{N_3} u\|_{L^{\infty}_{T}H^{s}_{x}}+\|P_{N_3} (w^k)\|_{L^{\infty}_{T}H^{s}_{x}}\right).
\end{multline*}
Therefore, for $s>\frac12$, by the above estimates and the algebra property of $H^{s}(\T)$, we have
\begin{multline} \label{eq: R1 ineq}
    \sum_{N_0,\dots,N_{\nu_n}}\sup_{t \in [-T,T]}\left|I_{t}(P_{N_0} \mathcal{D}(u),P_{N_1} u,\dots,P_{N_{\nu_n}}u) \right| \\
    \lesssim  \|\mathcal{D}(u)\|_{L^{\infty}_{T}H^{s+a}_{x}}\left(\|u\|_{L^{\infty}_{T}H^{s}_{x}}+\|u\|_{L^{\infty}_{T}H^{s}_{x}}^{d}\right)^3\|u\|_{L^{\infty}_{T}H^{s}_{x}}^{\nu_n-3}
    +\left(\|u\|_{L^{\infty}_{T}H^{s}_{x}}+\|u\|_{L^{\infty}_{T}H^{s}_{x}}^{d}\right)^3\|u\|_{L^{\infty}_{T}H^{s}_{x}}^{2\nu_n-3}
\end{multline}
provided the summability condition
\begin{align} \label{eq: summability}
    N_2^{-s}N_{3}^{-s}N_{4}^{\frac12-s}N_0^{a+1+n(1-\alpha)+}=O(1).
\end{align}
Using $N_3^{\alpha}N_4 \gtrsim N_1^{\alpha}$, for $s \geq \frac12$, we have
\begin{align*}
    N_2^{-s}N_{3}^{-s}N_{4}^{\frac12-s} \lesssim N_{3}^{(\alpha-2)s-\frac{\alpha}{2}}(N_{3}^{\alpha} N_{4})^{\frac12-s}\lesssim N_1^{-\frac{3\alpha}{\alpha+1}s+\frac{\alpha}{2(\alpha+1)}}.
\end{align*}
Hence, (\ref{eq: summability}) is satisfied if $s>\frac12+\frac{n+1}{3\alpha}-\frac{n\alpha}{3}$ and $a<\frac{3\alpha}{\alpha+1}s-\frac{\alpha}{2(\alpha+1)}+n(\alpha-1)-1$.

Notice that for any $0<\epsilon \ll 1$, the quantity (\ref{eq: R1 ineq}) is less than
\begin{align*}
    \epsilon \|\mathcal{D}(u)\|_{L^{\infty}_{T}H^{s+a}_{x}}^2+\epsilon^{-1}\left(\|u\|_{L^{\infty}_{T}H^{s}_{x}}+\|u\|_{L^{\infty}_{T}H^{s}_{x}}^{d}\right)^6\|u\|_{L^{\infty}_{T}H^{s}_{x}}^{2\nu_n-6}
    +\left(\|u\|_{L^{\infty}_{T}H^{s}_{x}}+\|u\|_{L^{\infty}_{T}H^{s}_{x}}^{d}\right)^3\|u\|_{L^{\infty}_{T}H^{s}_{x}}^{2\nu_n-3}.
\end{align*}
Canceling out $\epsilon \|\mathcal{D}(u)\|_{L^{\infty}_{T}H^{s+a}_{x}}^2$, we obtain the desired result for the $d \geq 4$ case. 

A similar proof applies to the $d=3$ case, however in this case we use Lemma \ref{lem: R2 pointwise bound} instead of Lemma \ref{lem: mu pointwise bound}. Then the summability condition for (\ref{eq: R1 ineq}) is
\begin{align*}
\begin{dcases}
    N_2^{-s+1}N_{3}^{-s}N_{4}^{\frac12-s}N_0^{a+n(1-\alpha)+}=O(1), &\, \textnormal{in the case } N_4 \gg N_5, \\
    N_2^{-s}N_{3}^{-s}N_{4}^{1-2s}N_0^{a+1+n(1-\alpha)+}=O(1), &\, \textnormal{in the case } N_4 \approx N_5.
\end{dcases}
\end{align*}
Using $N_3^{\alpha}N_4 \gtrsim N_1^{\alpha}$, we can verify this condition for $s$ and $a$ specified in the statement of this lemma.

For (b), we use the trivial inequality $|\mu_{3}(\xi_1,\xi_2,\xi_3)|\lesssim 1$. We apply the argument used in the proof of (a), however in this case we use the bilinear Strichartz estimate (\ref{eq: bilinear strichartz 1}) instead of (\ref{eq: bilinear strichartz 2}). We omit details.
\end{proof}

\begin{lem} \label{lem: mathfrak R energy estimate}
Let $0<T<1$, $m,n \geq 0$, $(k_0,\dots,k_n) \in [d]^{n+1}$ and $(l_1,\dots,l_m) \in [d]^{m}$. Let $\nu_{n,m}\geq 4$.\footnote{If $\nu_{n,m}\leq 3$, then $\mathfrak{R}_{k_0,\dots,k_n}^{l_1,\dots,l_m}(u)=0$.} Suppose that $s>\max\big(\frac12,1-\frac{\alpha}{4}-\frac{1}{2}(m+n)(\alpha-1)\big)$ and $a<2s-2+\frac{\alpha}{2}+(m+n)(\alpha-1)$. Let $u$ be a smooth solution to (\ref{eq: guage dgbo}). Then we have
\begin{align*}
    \left \|\int_{0}^{t} S(t-t') \mathfrak{R}_{k_0,\dots,k_n}^{l_1,\dots,l_m}(u)(t') dt' \right\|_{L^{\infty}_T H^{s+a}_{x}} \lesssim C(\|u\|_{L^{\infty}_T H^{s}_{x}})\|u\|_{L^{\infty}_T H^{s}_{x}}^{\nu_{n,m}},
\end{align*}
where $C(\|u\|_{L^{\infty}_T H^{s}_{x}})$ is an increasing function of $\|u\|_{L^{\infty}_T H^{s}_{x}}$.
\end{lem}
\begin{proof}
Observe that $\supp\big(\mathfrak{m}_{k_0,\dots,k_n}^{l_1,\dots,l_m}\big)\subseteq  \{  |\xi_1^*| \approx |\xi_2^*|\}$.
For $\xi=\xi_1+\cdots+\xi_{\nu_{n,m}}$, let
\begin{align*}
&\mathcal{A}:=\{(\xi_1,\dots,\xi_{\nu_{n,m}}) \in\Z_{\ast}^{\nu_{n,m}}: |\xi_1^*| \approx |\xi_2^*| \approx |\xi_3^*|\}, \\
&\mathcal{B}:=\{(\xi_1,\dots,\xi_{\nu_{n,m}}) \in\Z_{\ast}^{\nu_{n,m}}:  |\xi_1^*| \approx |\xi_2^*| \gg |\xi_3^*| \gtrsim|\xi|\}.
\end{align*}
Then by Lemma \ref{lem: higher-order resonances} and the definition of $\mathcal{D}^{1}_{\nu_{n,m}}$, we have
\begin{align*}
    \supp\left(\mathbbm{1}_{\mathcal{D}^{1}_{\nu_{n,m}}}\mathfrak{m}_{k_0,\dots,k_n}^{l_1,\dots,l_m}\right)\subseteq \mathcal{A} \cup \mathcal{B}.
\end{align*}

In the following, we assume by symmetry that $N_1 \geq N_2 \geq \dots \geq N_{\nu_{n,m}}$. Let
\begin{align*}
    \mathcal{D}(u):=\int_{0}^{t} S(t-t') \mathfrak{R}_{k_0,\dots,k_n}^{l_1,\dots,l_m}(u)(t') dt',
\end{align*}
and
\begin{multline*}
    I_{t}(u_0,\dots,u_{\nu_{n,m}})
    :=\int_{0}^{t} \sum_{\xi_0+\cdots+\xi_{\nu_{n,m}}=0} |\xi_{0}|^{2(s+a)} \xi_{0}\mathbbm{1}_{\mathcal{D}^{1}_{\nu_{n,m}}}\mathfrak{m}_{k_0,\dots,k_n}^{l_1,\dots,l_m}(\xi_1,\dots,\xi_{\nu_{n,m}}) \\
    \hat{u}_{0}(\xi_0,t')\cdots \hat{u}_{\nu_{n,m}}(\xi_{\nu_{n,m}},t') dt'.
\end{multline*}
By the H\"older and Bernstein's inequality we have
\begin{multline} \label{eq: G It}
    \sup_{t \in [-T,T]}\left|I_{t}(P_{N_0} \mathcal{D}(u),P_{N_1} u,\dots,P_{N_{\nu_{n,m}}}u) \right| \\
    \lesssim N_0^{2s+2a+1+(m+n)(1-\alpha)}\| P_{N_0} \mathcal{D}(u) P_{N_1} u\|_{L^2_{T}L^{2}_{x}}\| P_{N_2} u P_{N_3} u\|_{L^2_{T}L^{2}_{x}} \prod_{i=4}^{\nu_{n,m}}N_{i}^{\frac12}\|P_{N_i}u\|_{L^{\infty}_{T}L^2_{x}}.
\end{multline}

\noindent\textbf{Case 1: $(\xi_1,\dots,\xi_{\nu_{n,m}}) \in \mathcal{A}$.}

Let $w(x,t):=\mathcal{F}^{-1}_{\xi}|\mathcal{F}_{x}[u](\xi,t)|$. Using (\ref{eq: bilinear strichartz 1}) and $N_1 \approx N_2 \approx N_3$, we have
\begin{multline*}
    \| P_{N_0} \mathcal{D}(u) P_{N_1} u\|_{L^2_{T}L^{2}_{x}} \\
    \lesssim N_1^{\frac12-\frac{\alpha}{4}-s}\left(N_0^{-s-a}\|P_{N_0} \mathcal{D}(u)\|_{L^{\infty}_{T}H^{s+a}_{x}}+N_0^{-s-1+(m+n)(1-\alpha)}\|P_{N_0}(w^{\nu_{n,m}})\|_{L^{\infty}_{T}H^{s}_{x}}\right) \\
    \times\left(\|P_{N_1} u\|_{L^{\infty}_{T}H^{s}_{x}}+\sum_{k=2}^{d}\|P_{N_1} (w^k)\|_{L^{\infty}_{T}H^{s}_{x}}\right),
\end{multline*}
and
\begin{multline}
    \| P_{N_2} u P_{N_3} u\|_{L^2_{T}L^{2}_{x}}
    \lesssim N_1^{\frac12-\frac{\alpha}{4}-2s}\sum_{j,k=2}^{d}\left(\|P_{N_2} u\|_{L^{\infty}_{T}H^{s}_{x}}+\|P_{N_2}(w^j)\|_{L^{\infty}_{T}H^{s}_{x}}\right)\\
    \times\left(\|P_{N_3} u\|_{L^{\infty}_{T}H^{s}_{x}}+\|P_{N_3} (w^k)\|_{L^{\infty}_{T}H^{s}_{x}}\right).
\end{multline}
Therefore, $(\ref{eq: G It})$ is summable if $s>\max\big(\frac12,1-\frac{\alpha}{4}-\frac{1}{2}(m+n)(\alpha-1)\big)$ and $a<2s+(m+n)(\alpha-1)+\frac{\alpha}{2}-2$. For this $s$ and $a$, using the same argument as in the proof of Lemma \ref{lem: R2 energy estimate}, we can conclude that 
\begin{align*}
    \left \|\mathbb{P}_{\mathcal{A}}\mathcal{D}(u) \right\|_{L^{\infty}_T H^{s+a}_{x}} \lesssim C(\|u\|_{L^{\infty}_T H^{s}_{x}})\|u\|_{L^{\infty}_T H^{s}_{x}}^{\nu_{n,m}}.
\end{align*}

\noindent\textbf{Case 2: $(\xi_1,\dots,\xi_{\nu_{n,m}}) \in \mathcal{B}$.}

Recall that $N_1 \gg N_0$ and $N_2 \gg N_3$. By (\ref{eq: bilinear strichartz 2}), we have
\begin{multline*}
    \| P_{N_0} \mathcal{D}(u) P_{N_1} u\|_{L^2_{T}L^{2}_{x}} \\
    \lesssim \left(N_0^{-s-a}N_1^{-s}\|P_{N_0} \mathcal{D}(u)\|_{L^{\infty}_{T}H^{s+a}_{x}}+N_0^{-s-1+n(1-\alpha)}N_1^{-s} \|P_{N_0}(w^{\nu_{n,m}})\|_{L^{\infty}_{T}H^{s}_{x}}\right) \\
    \times\sum_{k=2}^{d}\left(\|P_{N_1} u\|_{L^{\infty}_{T}H^{s}_{x}}+\|P_{N_1} (w^k)\|_{L^{\infty}_{T}H^{s}_{x}}\right),
\end{multline*}
and
\begin{multline*}
    \| P_{N_2} u P_{N_3} u\|_{L^2_{T}L^{2}_{x}}
    \lesssim N_2^{-s}N_3^{-s}\sum_{j,k=2}^{d}\left(\|P_{N_2} u\|_{L^{\infty}_{T}H^{s}_{x}}+\|P_{N_2}(w^j)\|_{L^{\infty}_{T}H^{s}_{x}}\right)\\
    \times\left(\|P_{N_3} u\|_{L^{\infty}_{T}H^{s}_{x}}+\|P_{N_3} (w^k)\|_{L^{\infty}_{T}H^{s}_{x}}\right).
\end{multline*}
Since $(\xi_1,\dots,\xi_{\nu_{n,m}}) \in \mathcal{B}$, we have
\begin{align*}
    N_0^{2s+2a+1+(m+n)(1-\alpha)}N_0^{-s-a}N_1^{-s}N_2^{-s}N_3^{-s}
    \lesssim N_0^{-2s+a+1+(m+n)(1-\alpha)}.
\end{align*}
Therefore, $(\ref{eq: G It})$ is summable if $s>\frac12$ and $a<2s+(m+n)(\alpha-1)-1$. Hence for this $s$ and $a$, we have
\begin{align*}
    \left \|\mathbb{P}_{\mathcal{B}}\mathcal{D}(u) \right\|_{L^{\infty}_T H^{s+a}_{x}} \lesssim C(\|u\|_{L^{\infty}_T H^{s}_{x}})\|u\|_{L^{\infty}_T H^{s}_{x}}^{\nu_{n,m}}.
\end{align*}
\end{proof}

\subsection{Proof of smoothing}
We prove Theorem \ref{thm: smoothing} in this subsection.
\begin{proof}[Proof of Theorem \ref{thm: smoothing}]
Let $s>s(d,\alpha)$, $a<a(d,\alpha)$, and $a\leq \alpha-1$. For mean-zero, smooth initial data $g \in C^{\infty}(\T)$, let $u$ be the unique solution to (\ref{eq: dgbo}) emanating from $g$. By Lemma \ref{lem: normal form}, we have for $\t u :=\mathcal{G}u$,
\begin{align*}
    &\| \t u-S(t)g \|_{L_T^{\infty}H^{s+a}}
    \leq \sum_{n=0}^{N-1}\sum_{(k_0,\dots,k_{n}) \in [d]^{n+1}}\left(\| B_{k_0,\dots,k_n}(\t u)(t) \|_{L_T^{\infty}H^{s+a}} 
    + \| B_{k_0,\dots,k_n}(g) \|_{H^{s+a}}\right) \\
    &\quad+T\sum_{n=0}^{N}\bigg\| \sum_{(k_0,\dots,k_{n}) \in [d]^{n+1}}R^{1}_{k_0,\dots,k_{n}}(\t u) \bigg\|_{L_T^{\infty}H^{s+a}}
    + \sum_{n=0}^{N} \left\|\int_{0}^{t} S(t-t')R^{2}_{k_0,\dots,k_n}(\t u)(t') dt' \right\|_{L_T^{\infty}H^{s+a}}\\
    &\quad+ \sum_{n=0}^{N}\sum_{(k_0,\dots,k_{n}) \in [d]^{n+1}} \left\|\int_{0}^{t} S(t-t')D_{k_0,\dots,k_n}(\t u)(t') dt' \right\|_{L_T^{\infty}H^{s+a}}+T\|N_{k_0,\dots,k_{N}}(\t u)\|_{L_T^{\infty}H^{s+a}}.
\end{align*}
Also by Lemma \ref{lem: normal form 2}, for any $M \geq 1$, we have
\begin{align*}
    &\left\|\int_{0}^{t} S(t-t')D_{k_0,\dots,k_n}(\t u)(t') dt' \right\|_{L_T^{\infty}H^{s+a}} \\
    &\quad \lesssim \sum_{m=0}^{M-1}\sum_{(l_1,\dots,l_m)\in[d]^{m}}\left(\|\mathfrak{B}_{k_0,\dots,k_n}^{l_1,\dots,l_m}(\t u)\|_{L_T^{\infty}H^{s+a}}+\|\mathfrak{B}_{k_0,\dots,k_n}^{l_1,\dots,l_m}(g)\|_{H^{s+a}} \right)\\
    &\qquad +\sum_{m=0}^{M}\sum_{(l_1,\dots,l_m)\in[d]^{m}} \left\|\int_{0}^{t} S(t-t')[\mathfrak{R}_{k_0,\dots,k_n}^{l_1,\dots,l_m}(\t u)(t')]dt' \right\|_{L_T^{\infty}H^{s+a}} \\
    &\qquad +T\sum_{(l_1,\dots,l_m)\in[d]^{M}} \|\mathfrak{N}_{k_0,\dots,k_n}^{l_1,\dots,l_m}(\t u)\|_{L_T^{\infty}H^{s+a}}.
\end{align*}
Using the estimates in previous subsections and invoking the local theory in \cite{MT2}, by taking large enough $N,M$ we find the estimate 
\begin{align}\label{eq: smoothing}
    \left \|\mathcal{G}u-S(t)g \right \|_{L_T^{\infty}H^{s+a}} \lesssim C(\|\t u\|_{L_T^{\infty}H^{s}}) \lesssim C(\|g\|_{H^{s}}).
\end{align}
By Fatou's lemma, we can see that the same inequality holds true for mean-zero, non-smooth initial data $g \in H^{s}(\T)$. The continuity issue $u(t)-S(t)g \in C([-T,T],H^{s+a})$ can be handled as in \cite{ET1}. 
\end{proof}

\section{Low regularity well-posendess and smoothing} \label{section: Low regularity well-posendess and smoothing}

In this section, we prove Theorem \ref{thm: lwp}. Our argument is based on the Fourier restriction analysis applied to the normal form equation.

\subsection{Function spaces}
Fix $1<\alpha < 2$ and let 
\begin{align*}
    & D_1:=\left\{(\xi,\tau) \in \Z_{\ast}\times\R:  \lb \tau-\omega(\xi) \rb \gtrsim |\xi|^{\alpha}\right\}, \\
    & D_2:=\left(\Z_{\ast}\times\R\right)\setminus D_1,
\end{align*}
and write $\mathbb{P}_{D_1},\mathbb{P}_{D_2}$ for the Fourier projections onto $D_1,D_2$ respectively. For $0<\epsilon \ll 1$, define
\begin{align*}
    &\|u\|_{W^s}:=\|\lb \xi \rb^{s}\hat{u}\|_{\ell^{2}_{\xi}L^{1}_{\tau}}, \\
    &\|u\|_{Y^{s}}:=\|\mathbb{P}_{D_1}u\|_{X^{s-1+\frac{\alpha}{2}-\epsilon,\frac12}}+\|\mathbb{P}_{D_2}u\|_{X^{s,\frac12-\epsilon}}+\|u\|_{W^s}.
\end{align*}
For $T>0$, define the time-localized spaces by
\begin{align*}
    &\| u \|_{X^{s,b}_{T}}:=\inf\{\| v \|_{X^{s,b}}:v=u \text{ on } \T \times (-T,T) \}, \\
    &\| u \|_{Y^{s}_{T}}:=\inf\{\| v \|_{Y^{s}}:v=u \text{ on } \T \times (-T,T) \}.
\end{align*}
\begin{remark}
    The idea of decomposing phase space in defining $X^{s,b}$-based spaces first appeared in \cite{BT}. The small number $\epsilon>0$ in the above definition is to avoid the scaling argument. If the original equation has scaling symmetry, one may set $\epsilon=0$. Also, in the $\alpha \geq 2$, $P(x)=x^2$ case, we may use the simpler norm
    \begin{align*}
        \|u\|_{Y^{s}}=\|u\|_{X^{s,\frac12}}+\|u\|_{W^s},
    \end{align*}
    and then use the scaling argument.
\end{remark}
The following lemma summarizes some basic properties of the spaces introduced above:
\begin{lem} \label{lem: Ys space}
For any $s \in \R$ and $0<T<1$, the followings hold:
\begin{enumerate}[label=(\alph*)] 
    \item For any $u \in Y^{s}_{T}$, we have
    \begin{align*}
         \|u\|_{L^{\infty}_{T}H^{s}} \lesssim \|u\|_{Y^{s}_{T}}.
    \end{align*}
    \item For any $g \in H^{s}(\T)$, we have
    \begin{align*}
        \|S(t)g\|_{Y^{s}_{T}} \lesssim \|g\|_{H^{s}}.
    \end{align*}
    \item For any $F \in X^{s,-\frac12+}_{T}$, we have
    \begin{align*}
        \left\|\int_{0}^{t}S(t-t')Fdt'\right\|_{Y^{s}_{T}} \lesssim T^{0+}\left\|F \right\|_{X^{s,-\frac12+}_{T}}.
    \end{align*}
\end{enumerate}
\end{lem}
\begin{proof}
Recall the following well-known facts:
\begin{enumerate}[label=(\Alph*)]
        \item $W^s \hookrightarrow C_tH^{s}_{x}$.
        \item $\|\eta(t)S(t)g\|_{X^{s,\frac12}\cap W^s} \lesssim \|g\|_{H^{s}}$.
        \item $\left\|\eta(t)\int_{0}^{t}S(t-t')Fdt'\right\|_{X^{s,\frac12}\cap W^s} \lesssim \left\|F \right\|_{X^{s,-\frac12}}+\left\|\lb \xi \rb^{s}\lb \tau-\omega(\xi)\rb^{-1}\hat{F}\right\|_{\ell^{2}_{\xi}L^1_{\tau}}$.
\end{enumerate}
Here, $\eta(t)$ is a bump function supported on $[-2,2]$ which equals one on $[-1,1]$; see \cite{CKSTT}. (A) implies (a). The trivial inequality
\begin{align} \label{eq: YY}
    \|u\|_{Y^s} \lesssim \| u \|_{X^{s,\frac12}\cap W^s}
\end{align}
with (B) implies (b). By Cauchy-Schwarz, we can see that the right-hand side of (C) is less than $\|F\|_{X^{s,-\frac12+}}$. This with (\ref{eq: YY}) proves (c), where the factor $T^{0+}$ can be obtained by
\begin{align*}
    \|u\|_{X^{s,b}_T} \lesssim T^{b'-b}\|u\|_{X^{s,b'}_T}, \quad -\frac12<b<b'<\frac12.
\end{align*}
\end{proof}

\subsection{Bourgain space estimates for the nonlinear terms} \label{subsection: Bourgain space estimates for the nonlinear terms}

By Lemma \ref{lem: normal form}, the solution to the equation (\ref{eq: dgbo}) with $P(x)=x^2$ satisfies
\begin{align} \label{eq: Duhamel without truncation}
    u(t)=S(t)g +B_2(u)(t)-S(t)B_2(g) 
    +\int_{0}^{t} S(t-t')\left[R^{1}_{2,2}(u)(t')+R^{2}_{2,2}(u)(t')+N_{2,2}(u)(t')\right]dt.
\end{align}
To avoid the rescaling argument, we perform frequency truncation to (\ref{eq: Duhamel without truncation}). Let $N>0$. We use (\ref{eq: Duhamel without truncation}) only on the high frequency region $\{|\xi|\geq N\}$, and use the standard Duhamel formula (\ref{eq: Herr Duhamel}) on the low frequency region $\{|\xi|< N\}$. Then we can rewrite the original equation (\ref{eq: dgbo}) in the following form:
    \begin{multline} \label{eq: contraction}
        u(t)=S(t)g
        -\mathbb{P}_{\{|\xi|< N\}}\left[\int_{0}^{t} S(t-t')\left[\partial_{x}u^2(t')\right]dt'\right] \\
        +\mathbb{P}_{\{|\xi|\geq N\}}\bigg[B_2(u)(t)-S(t)B_2(g)
        +\int_{0}^{t} S(t-t')\left[R^{1}_{2,2}(u)(t')+R^{2}_{2,2}(u)(t')+N_{2,2}(u)(t')\right]dt'\bigg].
    \end{multline}
Similarly, for a general polynomial $P$, solution to (\ref{eq: guage dgbo}) satisfies
    \begin{align} \label{eq: gkdv normal form}
        u(t)=S(t)g
        -\mathbb{P}_{\{|\xi|< N\}}\left[\int_{0}^{t} S(t-t')\left[\partial_{x}P(u)(t')\right]dt'\right]
        +\mathbb{P}_{\{|\xi|\geq N\}}\bigg[\textnormal{Terms in Lemma \ref{lem: normal form} and \ref{lem: normal form 2}}\bigg].
    \end{align}
Below we estimate the $Y^s$ norms of the nonlinear terms appearing in the right-hand side of the above formulas. 

We use the following Strichartz estimates: 
\begin{align}\label{eq: L4 Strichartz}
    \|u\|_{L^4_{\T \times \R}} \lesssim \|u\|_{X^{0,\frac{\alpha+2}{4(\alpha+1)}}}
\end{align}
\begin{align} \label{eq: L6 Strichartz}
    \|u\|_{L^6_{\T \times \R}} \lesssim \|u\|_{X^{0+,\frac12+}},
\end{align}
see \cite[Lemma 3.1]{MT1}, \cite[Lemma 3.2]{S}. Also, recall the Sobolev embedding
\begin{align} \label{eq: L infinity}
    \|S(t)g\|_{L^{\infty}_{\T \times \R}} \lesssim \|g\|_{H^{\frac12+}}. 
\end{align}

\subsubsection{Boundary terms}
As explained in subsection \ref{subsection: Comments on the proof}, the main reason we introduced the space $Y^s$ is to estimate the ``boundary terms". Since the $X^{s,b}$ space is defined by $$\|u\|_{X^{s,b}}=\|S(-t)u\|_{H^{s}_{x}H^{b}_{t}},$$ this space has great strength in the estimates of the form $$\|S(t)F(u)\|_{X^{s,b}}.$$ On the other hand, this space generally does not provide good estimates for the terms that do not evolve like $S(t)$. This is the case of the boundary term $B_2(u)$ in the right-hand side of (\ref{eq: contraction}); there is no $S(t)$ attached in front of it. Technically, in the estimate 
\begin{align} \label{eq: B_2(u) Xsb}
    \|B_2(u)\|_{X^{s,b}} \lesssim \|u\|_{X^{s,b}}^2,
\end{align}
the modulation $b$ in the left-hand side acts like derivative loss, making it impossible to take the ideal $b=\frac12$. As in \cite{Mc2}, reducing the size of $b$ down to $1-\frac{1}{\alpha}$ can be a way to establish (\ref{eq: B_2(u) Xsb}). However, we go one step further: upon closer inspection, one can see that if $\alpha \geq \frac43$, actual loss in (\ref{eq: B_2(u) Xsb}) occurs only in the frequency region $D_1$, not on $D_2$. With this observation, instead of sacrificing the modulation $b$ in the entire frequency region, we bear the loss only in the region $D_1$.

\begin{lem} \label{lem: B Xsb} 
Let $\frac43\leq\alpha < 2$, $0<\epsilon\ll_{\alpha,s} 1$, and $s \geq \frac12-\frac{\alpha}{2}+\epsilon$. Then for $0<T<1$, we have
    \begin{align*}
    \|\mathbb{P}_{\{|\xi|\geq N\}}B_{2}(u)\|_{Y^{s}_{T}} \lesssim_{\epsilon} N^{-(\alpha-1)\epsilon}\|u\|_{Y^{s}_{T}}^2
\end{align*}
where the implicit constant does not depend on $T$.
\end{lem}

\begin{proof}
    We may drop the time localization, and work with $Y^{s}$ instead. Also, we always assume $|\xi|\geq N$, and suppress notations like $\mathbb{P}_{\{|\xi|\geq N\}}$ or $\mathbbm{1}_{\{|\xi|\geq N\}}$ below.
    
    We first estimate the $\|\mathbb{P}_{D_1}B_2(u)\|_{X^{s-1+\frac{\alpha}{2}-\epsilon,\frac12}}$ part of the $Y^s$ norm:
\begin{align*}
   &\| \mathbb{P}_{D_1} B_{2}(u) \|_{X^{s-1+\frac{\alpha}{2}-\epsilon,\frac12}} \\  
   &\quad \leq N^{-\epsilon}\Bigg \| \mathbbm{1}_{D_1}(\xi,\tau)\underset{\substack{\tau_{1}+\tau_{2}=\tau \\ \xi_{1}+\xi_{2}=\xi}}{\int\sum^{\ast}}\mathbbm{1}_{D_1}(\xi_1,\tau_1)\frac{\lb\xi\rb^{s+\frac{\alpha}{2}}\lb \tau-\omega(\xi)\rb^{\frac12}}{|\Omega_{2}(\xi_1,\xi_{2})|} 
   |\hat{u}(\xi_1,\tau_1)\hat{u}(\xi_{2},\tau_{2})|\Bigg \|_{L^2_{\xi,\tau}} \\
   &\qquad+N^{-\epsilon}\Bigg \| \mathbbm{1}_{D_1}(\xi,\tau)\underset{\substack{\tau_{1}+\tau_{2}=\tau \\ \xi_{1}+\xi_{2}=\xi}}{\int\sum^{\ast}}\mathbbm{1}_{D_2}(\xi_1,\tau_1)\frac{\lb\xi\rb^{s+\frac{\alpha}{2}}\lb \tau-\omega(\xi)\rb^{\frac12}}{|\Omega_{2}(\xi_1,\xi_{2})|} 
   |\hat{u}(\xi_1,\tau_1)\hat{u}(\xi_{2},\tau_{2})|\Bigg \|_{L^2_{\xi,\tau}} \\
   &\quad=:N^{-\epsilon}\times A+N^{-\epsilon}\times B.
\end{align*}
\noindent\textbf{Case 1: $\max_{i=1,2}\lb \tau_i-\omega(\xi_i)\rb \leq \frac{1}{4} \lb \tau-\omega(\xi)\rb$.} 

In this case, we have
\begin{align} \label{eq: case 1}
    \frac{2}{3}\lb \Omega_{2}(\xi_1,\xi_{2})\rb \leq \lb \tau-\omega(\xi)\rb \leq 2\lb \Omega_{2}(\xi_1,\xi_{2})\rb.
\end{align}
Assume by symmetry that $|\xi_1|\geq|\xi_2|$. If $|\xi_1|\gg|\xi_2|$, then by (\ref{eq: case 1}) and Lemma \ref{lem: phase function asymptotics 1}, we have
\begin{align*}
   \frac{\lb\xi\rb^{s+\frac{\alpha}{2}}\lb \tau-\omega(\xi)\rb^{\frac12}}{|\Omega_{2}(\xi_1,\xi_{2})|} \lesssim |\xi_1|^{s}|\xi_2|^{-\frac{1}{2}}.
\end{align*}
If $|\xi_1|\approx|\xi_2|$, then we have $|\Omega_2(\xi_1,\xi_2)| \gtrsim |\xi_1|^{\alpha}|\xi|$. Hence,
\begin{align*}
   \frac{\lb\xi\rb^{s+\frac{\alpha}{2}}\lb \tau-\omega(\xi)\rb^{\frac12}}{|\Omega_{2}(\xi_1,\xi_{2})|} \lesssim |\xi|^{s+\frac{\alpha}{2}-\frac{1}{2}}|\xi_1|^{-\frac{\alpha}{2}} \lesssim |\xi_1|^{s}|\xi_2|^{-\frac{1}{2}},
\end{align*}
where we used $s>\frac12-\frac{\alpha}{2}$ in the second inequality. Hence in any case, using
\begin{align*}
    \mathbbm{1}_{D_1}(\xi_1,\tau_1) \lesssim \mathbbm{1}_{D_1}(\xi_1,\tau_1)| \xi_1|^{-\frac{\alpha}{2}}\lb \tau_1-\omega(\xi_1)\rb^{\frac12},
\end{align*}
we have
\begin{align*}
   \mathbbm{1}_{D_1}(\xi_1,\tau_1)\frac{\lb\xi\rb^{s+\frac{\alpha}{2}}\lb \tau-\omega(\xi)\rb^{\frac12}}{|\Omega_{2}(\xi_1,\xi_{2})|} \lesssim \mathbbm{1}_{D_1}(\xi_1,\tau_1)|\xi_1|^{s-1+\frac{\alpha}{2}-\epsilon}\lb \tau_1-\omega(\xi_1)\rb^{\frac12}|\xi_2|^{\frac12-\alpha+\epsilon}.
\end{align*}
Applying Young's convolution inequality, we have
\begin{align*}
    A \lesssim \|\mathbb{P}_{D_1}u\|_{X^{s-1+\frac{\alpha}{2}-\epsilon,\frac12}}\|u\|_{W^s}.
\end{align*}
On the other hand, by the $L^4$-Strichartz estimate (\ref{eq: L4 Strichartz}), we have
\begin{align*}
    B 
    \lesssim \left\|\mathbb{P}_{D_2}u \right\|_{X^{s,\frac{\alpha+2}{4(\alpha+1)}}}\left\|u \right\|_{X^{-\frac12,\frac{\alpha+2}{4(\alpha+1)}}}.
\end{align*}
Hence, we have $A+B \lesssim \|u\|_{Y^s}^2$ for $s\geq \frac12-\frac{\alpha}{2}+\epsilon$.

\noindent\textbf{Case 2: $\max_{i=1,2}\lb \tau_i-\omega(\xi_i)\rb > \frac{1}{4} \lb \tau-\omega(\xi)\rb$.} 

Without loss of generality assume that $\lb \tau_1-\omega(\xi_1)\rb \geq \frac{1}{4} \lb \tau-\omega(\xi)\rb$. If $|\xi_1|\gg|\xi_2|$, we have
\begin{align*}
    \frac{\lb\xi\rb^{s+\frac{\alpha}{2}}\lb \tau-\omega(\xi)\rb^{\frac12}}{|\Omega_{2}(\xi_1,\xi_{2})|}
    &\lesssim |\xi_1|^{s-\frac{\alpha}{2}}\lb \tau_1-\omega(\xi_1)\rb^{\frac12} |\xi_2|^{-1} \\
    &\begin{multlined}
        \lesssim \mathbbm{1}_{D_1}(\xi_1,\tau_1)|\xi_1|^{s-1+\frac{\alpha}{2}-\epsilon}\lb \tau_1-\omega(\xi_1)\rb^{\frac12} |\xi_2|^{-\alpha+\epsilon}\\
        +\mathbbm{1}_{D_2}(\xi_1,\tau_1)|\xi_1|^{s}\lb \tau_1-\omega(\xi_1)\rb^{\frac12-\epsilon} |\xi_2|^{-1-\frac{\alpha}{2}+\alpha\epsilon}.
    \end{multlined}
\end{align*}
If $|\xi_2|\gg|\xi_1|$, then we have
\begin{align*}
    \frac{\lb\xi\rb^{s+\frac{\alpha}{2}}\lb \tau-\omega(\xi)\rb^{\frac12}}{|\Omega_{2}(\xi_1,\xi_{2})|}
    \lesssim |\xi_1|^{-1}\lb \tau_1-\omega(\xi_1)\rb^{\frac12}|\xi_2|^{s-\frac{\alpha}{2}}.
\end{align*}
If $|\xi_2|\approx|\xi_1|$, then we have
\begin{align*}
    \frac{\lb\xi\rb^{s+\frac{\alpha}{2}}\lb \tau-\omega(\xi)\rb^{\frac12}}{|\Omega_{2}(\xi_1,\xi_{2})|}
    &\lesssim |\xi|^{s+\frac{\alpha}{2}-1}|\xi_1|^{-\alpha}\lb \tau_1-\omega(\xi_1)\rb^{\frac12} \\
    &\lesssim 
    \begin{cases}
        |\xi_1|^{s-1+\frac{\alpha}{2}-\epsilon}\lb \tau_1-\omega(\xi_1)\rb^{\frac12}|\xi_2|^{-s+1-\frac{3\alpha}{2}+\epsilon}, &\, \textnormal{if } s<1-\frac{\alpha}{2}, \\
        |\xi_1|^{s-1+\frac{\alpha}{2}-\epsilon}\lb \tau_1-\omega(\xi_1)\rb^{\frac12}|\xi_2|^{-\alpha+\epsilon}, &\, \textnormal{if } s\geq1-\frac{\alpha}{2}.
    \end{cases}
\end{align*}
Therefore, by Young's convolution inequality we have
\begin{align*}
    A+B \lesssim \|\mathbb{P}_{D_1}u\|_{X^{s-1+\frac{\alpha}{2}-\epsilon,\frac12}}\|u\|_{W^s}+\|\mathbb{P}_{D_2}u\|_{X^{s,\frac12-\epsilon}}\|u\|_{W^s}.
\end{align*}

Next we estimate the $\|\mathbb{P}_{D_2}B_2(u)\|_{X^{s,\frac{1}{2}-\epsilon}}$ portion of the $Y^s$ norm:
\begin{align}\label{eq: B Xsb 4}
   \| \mathbb{P}_{D_2} B_{2}(u) \|_{X^{s,\frac{1}{2}-\epsilon}} 
   \leq N^{-(\alpha-1)\epsilon} \Bigg \| \mathbbm{1}_{D_2}(\xi,\tau)\underset{\substack{\tau_{1}+\tau_{2}=\tau \\ \xi_{1}+\xi_{2}=\xi}}{\int\sum^{\ast}}\frac{\lb\xi\rb^{s+1+(\alpha-1)\epsilon}\lb \tau-\omega(\xi)\rb^{\frac{1}{2}-\epsilon}}{|\Omega_{2}(\xi_1,\xi_{2})|} 
   |\hat{u}(\xi_1,\tau_1)\hat{u}(\xi_{2},\tau_{2})|\Bigg \|_{L^2_{\xi,\tau}}.
\end{align}
 If $\max_{i=1,2}\lb \tau_i-\omega(\xi_i)\rb \leq \frac{1}{4} \lb \tau-\omega(\xi)\rb$, then by (\ref{eq: case 1}), we must have $\lb \tau-\omega(\xi)\rb \gtrsim |\xi|^{\alpha}$, hence the right-hand side of (\ref{eq: B Xsb 4}) vanishes. Therefore we may only consider the case $\lb \tau_1-\omega(\xi_1)\rb \geq \frac{1}{4} \lb \tau-\omega(\xi)\rb$. 
 
 If $|\xi_1|\gg|\xi_2|$, then we have 
 \begin{align*}
    \left|\frac{\lb\xi\rb^{s+1+(\alpha-1)\epsilon}\lb \tau-\omega(\xi)\rb^{\frac12-\epsilon}}{\Omega_{2}(\xi_1,\xi_{2})}\right|
    &\lesssim |\xi_1|^{s+1-\alpha+(\alpha-1)\epsilon}\lb \tau_1-\omega(\xi_1)\rb^{\frac12-\epsilon}|\xi_2|^{-1} \\
    &\begin{multlined}
    \lesssim \mathbbm{1}_{D_1}(\xi_1,\tau_1)|\xi_1|^{s-1+\frac{\alpha}{2}-\epsilon}\lb \tau_1-\omega(\xi_1)\rb^{\frac12} |\xi_2|^{1-\frac{3\alpha}{2}}\\
    +\mathbbm{1}_{D_2}(\xi_1,\tau_1)|\xi_1|^{s}\lb \tau_1-\omega(\xi_1)\rb^{\frac12-\epsilon} |\xi_2|^{-\alpha+(\alpha-1)\epsilon}.
    \end{multlined}
\end{align*}
For the second inequality, we need $\alpha\geq\frac43$. By Young's convolution inequality, right-hand side of (\ref{eq: B Xsb 4}) is less than
\begin{align*}
    \|\mathbb{P}_{D_1}u\|_{X^{s-1+\frac{\alpha}{2}-\epsilon,\frac12}}\|u\|_{W^s}+\|\mathbb{P}_{D_2}u\|_{X^{s,\frac12-\epsilon}}\|u\|_{W^s}.
\end{align*}

If $|\xi_2|\gg|\xi_1|$, then we have 
\begin{align*}
    \left|\frac{\lb\xi\rb^{s+1+(\alpha-1)\epsilon}\lb \tau-\omega(\xi)\rb^{\frac12-\epsilon}}{\Omega_{2}(\xi_1,\xi_{2})}\right|
    \lesssim |\xi_1|^{-1}\lb \tau_1-\omega(\xi_1)\rb^{\frac12-\epsilon}|\xi_2|^{s+1-\alpha+(\alpha-1)\epsilon}.
\end{align*}
Hence we can dominate the right-hand side of (\ref{eq: B Xsb 4}) by
\begin{align} \label{eq: l1l2l2l1}
    \left\| \lb\xi\rb^{-1}\lb \tau-\omega(\xi)\rb^{\frac12-\epsilon}\hat{u}\right\|_{\ell^{1}_{\xi}L^{2}_{\tau}}\left\| \lb\xi\rb^{s+1-\alpha+(\alpha-1)\epsilon}\hat{u} \right\|_{\ell^{2}_{\xi}L^{1}_{\tau}} 
    \lesssim \left\| u\right\|_{X^{-\frac12+,\frac12-\epsilon}}\left\| u \right\|_{W^{s+1-\alpha+(\alpha-1)\epsilon}}.
\end{align}
Since 
\begin{align*}
     \left\| u\right\|_{X^{-\frac12+,\frac12-\epsilon}} \lesssim \left\|  \mathbb{P}_{D_1}u\right\|_{X^{-\frac12-\alpha \epsilon+,\frac12}}+\left\| \mathbb{P}_{D_2}u\right\|_{X^{-\frac12+,\frac12-\epsilon}},
\end{align*}
the right-hand side of (\ref{eq: l1l2l2l1}) is less than $\|u\|_{Y^s}^2$ for $s\geq\frac12-\frac{\alpha}{2}+\epsilon$.

If $|\xi_2|\approx|\xi_1|$, then we have 
\begin{align*}
    \left|\frac{\lb\xi\rb^{s+1+(\alpha-1)\epsilon}\lb \tau-\omega(\xi)\rb^{\frac12-\epsilon}}{\Omega_{2}(\xi_1,\xi_{2})}\right|
    &\lesssim |\xi_1|^{-\alpha}|\xi|^{s+(\alpha-1)\epsilon}\lb \tau_1-\omega(\xi_1)\rb^{\frac12-\epsilon} \\
    &\begin{multlined}
    \lesssim \mathbbm{1}_{D_1}(\xi_1,\tau_1)|\xi_1|^{s-1+\frac{\alpha}{2}-\epsilon}\lb \tau_1-\omega(\xi_1)\rb^{\frac12} |\xi_2|^{-s+1-\frac{3\alpha}{2}}|\xi|^{s+(\alpha-1)\epsilon} \\
    +\mathbbm{1}_{D_2}(\xi_1,\tau_1)|\xi_1|^{s}\lb \tau_1-\omega(\xi_1)\rb^{\frac12-\epsilon} |\xi_2|^{-s-\alpha}|\xi|^{s+(\alpha-1)\epsilon}.
    \end{multlined}
\end{align*}
By Young's convolution inequality we have
\begin{align*}
(\ref{eq: B Xsb 4}) \lesssim \|\mathbb{P}_{D_1}u\|_{X^{s-1+\frac{\alpha}{2}-\epsilon,\frac12}}\|u\|_{W^s}+\|\mathbb{P}_{D_2}u\|_{X^{s,\frac12-\epsilon}}\|u\|_{W^s}.
\end{align*}

Finally we estimate the $W^s$ part of the $Y^s$ norm:
\begin{multline} \label{eq: Y^s second}
    \Bigg \| \underset{\substack{\tau_{1}+\tau_{2}=\tau \\ \xi_{1}+\xi_{2}=\xi}}{\int\sum^{\ast}}\frac{\lb\xi\rb^{s}\xi}{\Omega_{2}(\xi_1,\xi_{2})} 
   \hat{u}_{1}(\xi_1,\tau_1)\hat{u}_{2}(\xi_{2},\tau_{2})\Bigg \|_{\ell^2_{\xi}L^1_{\tau}} \\
   \lesssim N^{-\epsilon}\Bigg\| \sum_{\xi_{1}+\xi_{2}=\xi}^{\ast}\frac{\lb\xi \rb^{s+1+\epsilon}}{|\Omega_{2}(\xi_1,\xi_{2})|} 
   \|\hat{u}_{1}(\xi_1,\,\cdot\,)\|_{L^1_{\tau}}\|\hat{u}_{2}(\xi_{2},\,\cdot\,)\|_{L^1_{\tau}} \Bigg\|_{\ell^2_{\xi}}.
\end{multline}
Assume by symmetry that $|\xi_1| \geq |\xi_2|$. If $|\xi_1| \gg |\xi_2|$, then we have
\begin{align*}
    (\ref{eq: Y^s second})
    &\lesssim \Bigg\| \sum_{\xi_{1}+\xi_{2}=\xi}^{\ast}|\xi_1|^{s+1-\alpha+\epsilon}|\xi_2|^{-1}
   \|\hat{u}(\xi_1,\,\cdot\,)\|_{L^1_{\tau}}\|\hat{u}(\xi_{2},\,\cdot\,)\|_{L^1_{\tau}} \Bigg\|_{\ell^2_{\xi}} \\
   &\lesssim \||\xi|^{s}\hat{u}\|_{\ell^2_{\xi}L^1_{\tau}}\||\xi|^{-\alpha+\epsilon}\hat{u}\|_{\ell^1_{\xi}L^1_{\tau}} \\
   &\lesssim \|u\|_{W^s}\|u\|_{W^{\frac12-\alpha+\epsilon+}}.
\end{align*}
If $|\xi_1| \approx |\xi_2|$, then similarly we have
\begin{align*}
    (\ref{eq: Y^s second})
    \lesssim \Bigg\| \sum_{\xi_{1}+\xi_{2}=\xi}^{\ast}|\xi|^{s+\epsilon}|\xi_1|^{s}|\xi_2|^{-s-\alpha}
   \|\hat{u}(\xi_1,\,\cdot\,)\|_{L^1_{\tau}}\|\hat{u}(\xi_{2},\,\cdot\,)\|_{L^1_{\tau}} \Bigg\|_{\ell^2_{\xi}} 
   \lesssim \|u\|_{W^s}^2.
\end{align*}
This completes the proof.
\end{proof}

\begin{lem}\label{lem: B Xsb 2} Let $\frac43 \leq \alpha < 2$ and $s > \frac12$. Then for $m,n\geq 0$ and $0<T<1$, we have
 \begin{align*}
        &\|\mathbb{P}_{\{|\xi|\geq N\}}B_{k_0,\dots,k_n}(u)\|_{Y^{s}_{T}} \lesssim_{\epsilon} N^{-(\alpha-1)\epsilon}\|u\|_{Y^{s}_{T}}^{\nu_n}, \\
        &\|\mathbb{P}_{\{|\xi|\geq N\}}\mathfrak{B}^{l_1,\dots,l_m}_{k_0,\dots,k_n}(u)\|_{Y^{s}_{T}} \lesssim_{\epsilon} N^{-(\alpha-1)\epsilon}\|u\|_{Y^{s}_{T}}^{\nu_{n,m}},
    \end{align*}
    for all $0<\epsilon\ll_{\alpha,s}1$, where the implicit constant does not depend on $T$.
\end{lem}
\begin{proof}
     As in the proof of Lemma \ref{lem: B Xsb}, we separately consider the cases $\max_{1\leq i \leq \nu_n}\lb \tau_i-\omega(\xi_i)\rb \leq \frac{1}{2\nu_n} \lb \tau-\omega(\xi)\rb$ and $\max_{1\leq i \leq \nu_n}\lb \tau_i-\omega(\xi_i)\rb > \frac{1}{2\nu_n} \lb \tau-\omega(\xi)\rb$. The argument after this is indeed simpler because we are assuming $s>\frac12$ here.\footnote{At the place we used the $L^4$-Strichartz estimate in the proof of Lemma \ref{lem: B Xsb}, we simply apply Young's convolution inequality.} We omit details.
\end{proof}

\subsubsection{A general estimate}
In the case the symbol mainly depends on the first three dominant frequencies, we find the following lemma useful:
\begin{lem} \label{lem: general Xsb}
For $m:\Z^{n}_{\ast}\to \C$, let $\Lambda_{m}$ be the Fourier multiplier operator given by
\begin{align*}
    \mathcal{F}\left[\Lambda_{m}(u_1,\dots,u_n)\right](\xi):=\sum_{\xi_1+\cdots+\xi_{n}=\xi}m(\xi_1,\dots,\xi_n)\hat{u}_1(\xi_1)\cdots\hat{u}_n(\xi_n).
\end{align*}
    Suppose that $\alpha>\frac{2+\sqrt{67}}{7}$ and
    \begin{align} \label{eq: general multiplier bound}
    \sup_{(\xi_1,\dots,\xi_n) \in \Z^{n}_{\ast}}\frac{|m(\xi_1,\dots,\xi_n)|\lb\xi_1+\cdots+\xi_n\rb^{s+a+}}{\prod_{i=1}^{3}\lb\xi_i^{*}\rb^{s}\prod_{i=4}^{n}\lb\xi_i^{*}\rb^{s-\frac12-}}<\infty.
\end{align}
Then we have
\begin{align*}
    \|\Lambda_{m}(u_1,\dots,u_n)\|_{X^{s+a,-\frac12+}_{T}}\lesssim_{\epsilon} \prod_{i=1}^{n}\|u_i\|_{Y^{s}_{T}}
\end{align*}
for all $0<\epsilon \ll_{\alpha,s} 1$.
\end{lem}

\begin{proof}
Since
\begin{align*}
    \|u\|_{X^{s,1-\frac{1}{\alpha}-}} \lesssim \|\mathbb{P}_{D_1}u\|_{X^{s-1+\frac{\alpha}{2}-,\frac12}}+\|\mathbb{P}_{D_2}u\|_{X^{s,\frac12-}},
\end{align*}
it suffices to show the inequality 
\begin{align} \label{eq: Z ineq 1}
    \|\Lambda_{m}(u_1,\dots,u_n)\|_{X^{s+a,-\frac{1}{2}+}} \lesssim  \prod_{i=1}^{3}\|u_i\|_{X^{s,1-\frac{1}{\alpha}-}}\prod_{i=4}^{n}\|u_i\|_{W^{s}}.
\end{align}
By possibly rearranging the frequencies, we may assume that $$\supp(\hat{u}_1,\dots,\hat{u}_{n}) \subseteq \{(\xi_1,\dots,\xi_n) \in\Z^n_{\ast}:|\xi_1|\geq|\xi_2|\geq \dots \geq |\xi_n|\}.$$ We may also assume that $\hat{u}_{i}$ is real-valued and non-negative for each $1 \leq i \leq n+1$. 

For $1\leq i \leq n+1$, let $\sigma_i:=\tau_{i}-\omega(\xi_{i})$. To show (\ref{eq: Z ineq 1}), it suffices to dominate the quantity
\begin{align}\label{eq: Xsb dual 1}
    \underset{\substack{\tau_1+\cdots+\tau_{n+1}=0 \\ \xi_1+\cdots+\xi_{n+1}=0}}{\int\sum}\frac{|m(\xi_1,\dots,\xi_n)|\lb\xi_{n+1}\rb^{s+a}}{\prod_{i=1}^{n}\lb\xi_i\rb^{s}\prod_{i=1}^{3}\lb\sigma_i\rb^{1-\frac{1}{\alpha}-}\lb\sigma_{n+1}\rb^{\frac12-}}\prod_{i=1}^{n+1}\hat{u}_{i}(\xi_i,\tau_i)
\end{align}
by $\prod_{i=1}^{3}\|u_i\|_{L^2_{x,t}}\prod_{i=4}^{n}\|u_i\|_{W^0}\|u_{n+1}\|_{L^2_{x,t}}$. 

By interpolating the $L^4$-Strichartz estimate (\ref{eq: L4 Strichartz}) and the Plancherel identity, we have
\begin{align} \label{eq: L 18/5 Strichartz}
    \left\| \left[\frac{f}{\lb \tau-\omega(\xi)\rb^{\frac{2\alpha+4}{9(\alpha+1)}+}}\right]^{\vee}\right\|_{L^{\frac{18}{5}+}_{x,t}} \lesssim \|f\|_{L^2_{\xi,\tau}}.
\end{align}
For $\alpha>\frac{2+\sqrt{67}}{7}$, we have $1-\frac{1}{\alpha} > \frac{2\alpha+4}{9(\alpha+1)}$. Hence by (\ref{eq: L 18/5 Strichartz}),
\begin{align*}
    (\ref{eq: Xsb dual 1}) 
    &\lesssim \prod_{i=1}^{3}\left\| \left[\frac{\hat{u}_i}{\lb \tau-\omega(\xi)\rb^{1-\frac{1}{\alpha}-}}\right]^{\vee}\right\|_{L^{\frac{18}{5}+}_{x,t}}
    \prod_{i=4}^{n}\left\| \left[\frac{\hat{u}_i}{\lb \xi \rb^{\frac12+}}\right]^{\vee}\right\|_{L^{\infty}_{x,t}}
    \left\| \left[\frac{\lb \xi \rb^{0-}\hat{u}_{n+1}}{\lb \tau-\omega(\xi)\rb^{\frac12-}}\right]^{\vee}\right\|_{L^{6-}_{x,t}} \\
    &\lesssim \prod_{i=1}^{3}\|u_i\|_{L^2_{x,t}}\prod_{i=4}^{n}\|u_i\|_{W^0}\|u_{n+1}\|_{L^2_{x,t}}
\end{align*}
as desired.
\end{proof}

\begin{cor} \label{cor: general}
For $\alpha>\frac{2+\sqrt{67}}{7}$ and $0<T<1$, the followings hold:
    \begin{enumerate}[label=(\alph*)] 
    \item For $s>\frac{1}{2}-\frac{\alpha}{2}$, we have
    \begin{align*}
    \|R^{1}_{2,2}(u)\|_{X^{s+a,-\frac12+}_{T}}+\|R^{2}_{2,2}(u)\|_{X^{s+a,-\frac12+}_{T}} \lesssim \|u\|_{Y^s_T}^{3}
    \end{align*}
    for some $a>0$.
    \item For $N \geq 2$, fix $\mathcal{K}:=\{2^{n_1},3^{n_2}\}$ with $n_1+n_2=N$. Let $M:=1+\sum_{j=1}^{2}n_{j}(k_{j}-1)$ and suppose that $M \geq 4$. Then for $s>\frac12$, we have
    \begin{align*}
        \bigg\| \sum_{\theta \in \textnormal{Perm}(\mathcal{K})}R^{1}_{\theta(1),\dots,\theta(N)}(u)\bigg\|_{X^{s+a,-\frac12+}_{T}} \lesssim \|u\|_{Y^s_T}^{M}.
    \end{align*}
    for some $a>0$.
    \item Let $n \geq 1$ and $(k_0,\dots,k_n) \in [3]^{n+1}$. Then for $s>1-\frac{\alpha}{2}$, we have
    \begin{align*}
        \|R^{2}_{k_0,\dots,k_n}(u)\|_{X^{s+a,-\frac12+}_{T}} \lesssim \|u\|_{Y^s_T}^{\nu_n}
    \end{align*}
    for some $a>0$.
    \item Let $m,n \geq 0$, $d \geq 2$, $(k_0,\dots,k_n) \in [d]^{n+1}$, and $(l_1,\dots,l_m) \in [d]^{m}$. Then for $s>\frac12$, we have
    \begin{align*}
        \|\mathfrak{R}_{k_0,\dots,k_n}^{l_1,\dots,l_m}(u)\|_{X^{s+a,-\frac12+}_{T}} \lesssim \|u\|_{Y^s_T}^{\nu_{n,m}}
    \end{align*}
    for some $a>0$.
     \item  Let $n \geq 1$ and assume further that $\alpha>\frac{n+1}{n}$. For $s>\frac12$, we have
    \begin{align*}
    \|N_{k_0,\dots,k_n}(u) \|_{X^{s+a,-\frac12+}_{T}} \lesssim \|u\|_{Y^s_T}^{\nu_n}
    \end{align*}
    for some $a>0$.
    \item Let $m+n \geq 1$ and assume further that $\alpha>\frac{m+n+1}{m+n}$. For $s>\frac12$, we have
    \begin{align*}
    \|\mathfrak{N}_{k_0,\dots,k_n}^{l_1,\dots,l_m}(u) \|_{X^{s+a,-\frac12+}_{T}} \lesssim \|u\|_{Y^s_T}^{\nu_{n,m}}
    \end{align*}
    for some $a>0$.
\end{enumerate}
\end{cor}
\begin{proof}
    All these estimates easily follow from Lemma \ref{lem: general Xsb} with the pointwise bounds in Section \ref{section: Pointwise estimates}. For (d), we use the fact that the symbol is supported on the set $\mathcal{A} \cup \mathcal{B}$ defined in the proof of Lemma \ref{lem: mathfrak R energy estimate}.
\end{proof}
\subsubsection{Low-frequency terms}

\begin{lem}
    Let $\alpha>\frac{1+\sqrt{31}}{5}$ and $s>\frac{1}{2}-\frac{\alpha}{2}$. Then we have
    \begin{align*}
        \left\|\mathbb{P}_{\{|\xi|< N\}}\partial_{x}u^2\right\|_{X^{s,-\frac12+}_{T}}\lesssim_{\epsilon} N^{\frac{1}{\alpha}+|s|} \|u\|_{Y^s_T}^2
    \end{align*}
    for all $0<\epsilon\ll_{\alpha,s}1$.
\end{lem}
\begin{proof}
    As in the proof of Lemma \ref{lem: general Xsb}, we demonstrate the inequality
    \begin{align} \label{eq: low freq dual}
        \underset{\substack{\tau_{1}+\tau_{2}+\tau_{3}=0 \\ \xi_{1}+\xi_{2}+\xi_{3}=0}}{\int\sum}\frac{\mathbbm{1}_{\{|\xi_3| \leq N\}}|\xi_3|\lb\xi_3\rb^{s}}{\lb\xi_1\rb^{s}\lb\xi_2\rb^{s}\lb\sigma_1\rb^{1-\frac{1}{\alpha}-}\lb\sigma_2\rb^{1-\frac{1}{\alpha}-}\lb\sigma_3\rb^{\frac{1}{2}-}} \prod_{i=1}^{3}f_i(\xi_i,\tau_i) \lesssim \prod_{i=1}^{3}\|f_i\|_{L^2_{x,t}},
    \end{align}
    where $\sigma_i:=\tau_{i}-\omega(\xi_{i})$. Since $|\xi_1+\xi_2| \leq N$, we only need to consider the case $|\xi_1|\approx|\xi_2|$. 
    
    \noindent\textbf{Case 1: $\max_{1\leq i \leq 3}\lb \sigma_i\rb=\lb \sigma_1 \rb$ or $\lb \sigma_2 \rb$.}
    
    Let $\max_{1\leq i \leq 3}\lb \sigma_i\rb=\lb \sigma_1 \rb$. In this case, we have $\lb \sigma_1 \rb^{1-\frac{1}{\alpha}-}\gtrsim |\xi_1|^{\alpha-1-}|\xi_3|^{1-\frac{1}{\alpha}-}$ by Lemma \ref{lem: phase function asymptotics 1}. Using this, we may apply the $(2,3+,6-)$-H\"older inequality to the left-hand side of (\ref{eq: low freq dual}). The $L^{3+}$-Strichartz estimate requires modulation $\frac{\alpha+2}{6(\alpha+1)}+$. Hence we need $1-\frac{1}{\alpha}>\frac{\alpha+2}{6(\alpha+1)}$, that is, $\alpha>\frac{1+\sqrt{31}}{5}$.

    \noindent\textbf{Case 2: $\max_{1\leq i \leq 3}\lb \sigma_i\rb=\lb \sigma_3 \rb$.} 
    
    In this case, we have $\lb \sigma_3 \rb^{1-\frac{1}{\alpha}-}\gtrsim |\xi_1|^{\alpha-1-}|\xi_3|^{1-\frac{1}{\alpha}-}$. Hence
    \begin{align*}
        \frac{\mathbbm{1}_{\{|\xi_3| \leq N\}}|\xi_3|\lb\xi_3\rb^{s}}{\lb\xi_1\rb^{s}\lb\xi_2\rb^{s}\lb\sigma_1\rb^{1-\frac{1}{\alpha}-}\lb\sigma_2\rb^{1-\frac{1}{\alpha}-}\lb\sigma_3\rb^{\frac{1}{2}-}}
        \lesssim \frac{N^{\frac{1}{\alpha}+|s|}}{\lb\xi_1\rb^{2s+\alpha-1-}\lb\sigma_1\rb^{1-\frac{1}{\alpha}-}\lb\sigma_2\rb^{1-\frac{1}{\alpha}-}\lb\sigma_3\rb^{\frac{1}{\alpha}-\frac{1}{2}}}.
    \end{align*}
    Therefore, if $\frac{2\alpha^2+\alpha-2}{4\alpha(\alpha+1)}>1-\frac{1}{\alpha}$, i.e. $\alpha>\frac{1+\sqrt{17}}{4}$, we can use the estimate
    \begin{align*}
    \left\|\left[\frac{f}{\lb \tau-\omega(\xi) \rb^{\frac{2\alpha^2+\alpha-2}{4\alpha(\alpha+1)}}}\right]^{\vee}\right\|_{L^{\frac{4\alpha(\alpha+2)}{3\alpha+2}}_{x,t}}+\left\|\left[\frac{f}{\lb \tau-\omega(\xi) \rb^{\frac{1}{\alpha}-\frac12}}\right]^{\vee}\right\|_{L^{\frac{2\alpha(\alpha+2)}{2\alpha^2+\alpha-2}}_{x,t}}
    \lesssim \|f\|_{L^2_{x,t}}
    \end{align*}
    and the $\big(\frac{2\alpha(\alpha+2)}{2\alpha^2+\alpha-2},\frac{2\alpha(\alpha+2)}{2\alpha^2+\alpha-2},\frac{4\alpha(\alpha+2)}{3\alpha+2}\big)$-H\"older inequality to the left-hand side of (\ref{eq: low freq dual}).
\end{proof}

\begin{lem}
    Let $s>\frac12$. Then for $m \geq 2$, we have
    \begin{align*}
        \left\|\mathbb{P}_{\{|\xi|< N\}}\partial_{x}(u^{m})\right\|_{X^{s,-\frac12+}_{T}}\lesssim N^{1+s} \|u\|_{Y^s_T}^m
    \end{align*}
    for all $0<\epsilon\ll_{\alpha,s}1$.
\end{lem}
\begin{proof}
    This follows easily from Young's convolution inequality and Bernstein's inequality.
\end{proof}

\subsubsection{The non-resonant term \texorpdfstring{$N_{2,2}$}{N_{2,2}}}

\begin{lem} \label{lem: N22 Xsb}
Let $\frac32< \alpha < 2$ and $s>\frac12-\frac{\alpha}{2}$. Then for $0<T<1$, we have
\begin{align*}
    \|N_{2,2}(u)\|_{X^{s+a,-\frac12+}_{T}} \lesssim_{\epsilon} \|u\|_{Y^{s}_{T}}^3
\end{align*}
for some $a>0$ and all $0<\epsilon\ll_{\alpha,s}1$.
\end{lem}
\begin{proof}
Let
\begin{align*}
    m(\xi_1,\xi_2,\xi_3):=\mathbbm{1}_{\mathcal{N}_{3}}(\xi_1,\xi_2,\xi_3)\frac{i (\xi_1+\xi_2+\xi_3) (\xi_2+\xi_3)}{4\Omega_{2}(\xi_1,\xi_2+\xi_3)}.
\end{align*}
Direct computation gives
\begin{align*}
    \mathcal{F}[N_{2,2}(u)](\xi)=\sum_{\xi_1+\xi_2+\xi_3=\xi}m(\xi_1,\xi_2,\xi_3)\hat{u}(\xi_1)\hat{u}(\xi_2)\hat{u}(\xi_3).
\end{align*}
Let $|\xi|_{\max}:=\max(|\xi_1+\xi_2+\xi_3|,|\xi_1|,|\xi_2+\xi_3|)$ and define $|\xi|_{\text{med}},|\xi|_{\min}$ similarly. Then by Lemma \ref{lem: phase function asymptotics 1},
\begin{align*}
    |m(\xi_1,\xi_2,\xi_3)| \lesssim \frac{|\xi_1+\xi_2+\xi_3|||\xi_2+\xi_3|}{|\xi|_{\max}^{\alpha-1}|\xi|_{\text{med}}|\xi|_{\min}} \lesssim |\xi|_{\max}^{2-\alpha}|\xi_1|^{-1}.
\end{align*}
The following table shows the upper bound of $m$ and the lower bound of $|\Omega_{3}(\xi_1,\xi_2,\xi_3)|$ (up to the $\xi_2\leftrightarrow\xi_3$ symmetry) on each partition of the set
    $$\left\{(\xi_1,\xi_2,\xi_3,\xi_4) \in \Z^4_{\ast}: (\xi_1,\xi_2,\xi_3) \in \mathcal{N}_3,\ \xi_1+\xi_2+\xi_3+\xi_4=0\right\}.$$
\begin{center}
\begin{tabular}{ |c|c|c| } 
 \hline
 Frequency region & Upper bound of $m$ & Lower bound of $|\Omega_{3}|$ \\ 
 \hline
 $\mathcal{P}_1:=\{|\xi_1|\approx|\xi_2|\approx|\xi_4|\gg|\xi_3|\}$ & $|\xi_1|^{1-\alpha}$ & $|\xi_1|^{\alpha+1}$ \\
 $\mathcal{P}_2:=\{|\xi_2|\approx|\xi_3|\approx|\xi_4|\gg|\xi_1|\}$ & $|\xi_1|^{-1}|\xi_2|^{2-\alpha}$ & $|\xi_2|^{\alpha+1}$ \\ 
 $\mathcal{P}_3:=\{|\xi_1|\approx|\xi_2|\gg \max(|\xi_3|,|\xi_4|)\}$ & $|\xi_1|^{1-\alpha}$ & $|\xi_1|^{\alpha}|\xi_3+\xi_4|$ \\ 
 $\mathcal{P}_4:=\{|\xi_1|\approx|\xi_4| \gg \max(|\xi_2|,|\xi_3|)\}$ & $|\xi_1|^{1-\alpha}$ & $|\xi_1|^\alpha|\xi_2+\xi_3|$ \\ 
 $\mathcal{P}_5:=\{|\xi_2|\approx|\xi_3|\gg \max(|\xi_1|,|\xi_4|)\}$ & $|\xi_1|^{-1}\max(|\xi_1|,|\xi_4|)^{2-\alpha}$ & $|\xi_2|^{\alpha}|\xi_1+\xi_4|$ \\ 
 $\mathcal{P}_6:=\{|\xi_2|\approx|\xi_4| \gg \max(|\xi_1|,|\xi_3|)\}$ & $|\xi_1|^{-1}|\xi_2|^{2-\alpha}$ & $|\xi_2|^\alpha|\xi_1+\xi_3|$ \\ 
 \hline
\end{tabular}
\end{center}
Define the off-diagonal extension of $N_{2,2}$ by
\begin{align*}
    \mathcal{F}[N_{2,2}(u_1,u_2,u_3)](\xi):=\sum_{\xi_1+\xi_2+\xi_3=\xi}m(\xi_1,\xi_2,\xi_3)\hat{u}_1(\xi_1)\hat{u}_2(\xi_2)\hat{u}_3(\xi_3).
\end{align*}
Below we estimate 
    \begin{align*}
        \left\|N_{2,2}(\mathbb{P}_{D_{i_1}}u_1,\mathbb{P}_{D_{i_2}}u_2,\mathbb{P}_{D_{i_3}}u_3)\right\|_{X^{s+a,-\frac12+}}
    \end{align*}
for $i_1,i_2,i_3\in\{1,2\}$. For $1 \leq i \leq 4$, we write $\sigma_i$ for $\tau_i-\omega(\xi_i)$.

\noindent \textbf{Case 1: $(\xi_1,\xi_2,\xi_3,\xi_4)\in \bigcup_{i=1}^{4}\mathcal{P}_i$.}

It is apparent from the above table that this case is better than the case $(\xi_1,\xi_2,\xi_3,\xi_4)\in\mathcal{P}_6$. We only need to analyze $\mathcal{P}_5$ and $\mathcal{P}_6$.

\noindent \textbf{Case 2: $(\xi_1,\xi_2,\xi_3,\xi_4)\in \mathcal{P}_5$.}

\noindent \underline{Subcase 2a: $i_1=i_2=i_3=1$.} We only look at the worst case $|\xi_4| \gg |\xi_1|$. The $|\xi_4| \lesssim |\xi_1|$ case can be handled similarly. By duality, it suffices to demonstrate that the quantity
\begin{align}\label{eq: N22 dual case 1}
    \underset{\substack{\tau_{1}+\tau_{2}+\tau_3+\tau_4=0 \\ \xi_{1}+\xi_{2}+\xi_3+\xi_4=0}}{\int\sum^{\ast}}\frac{|m(\xi_1,\xi_2,\xi_3)||\xi_4|^{s+a}}{|\xi_1\xi_2\xi_3|^{s-1+\frac{\alpha}{2}-}}\prod_{i=1}^{4} \frac{f_{i}(\xi_{i},\tau_{i})}{\lb \sigma_i\rb^{\frac{1}{2}-}}
\end{align}
is dominated by $\prod_{i=1}^{4}\|f_i\|_{L^2_{\xi,\tau}}$. By interpolating (\ref{eq: L4 Strichartz}) with the Plancherel identity, we have
\begin{align*}
    \|u\|_{L^{3}_{\T \times \R}} \lesssim \|u\|_{X^{0,\frac{\alpha+2}{6(\alpha+1)}}}.
\end{align*}
Notice that we have
\begin{align}\label{eq: 111}
    \frac{|m(\xi_1,\xi_2,\xi_3)||\xi_4|^{s+a}}{|\xi_1\xi_2\xi_3|^{s-1+\frac{\alpha}{2}-}\lb \sigma_1^* \rb^{\frac12-}\min\left(\lb \sigma_2 \rb,\lb \sigma_3 \rb\right)^{\frac{2\alpha+1}{6(\alpha+1)}-}} \lesssim \frac{1}{|\xi_1|^{\frac12+}|\xi_2|^{0+}}
\end{align}
for some $a>0$. Here we used $\lb \sigma_1^* \rb \gtrsim |\xi_2|^{\alpha}|\xi_4|$ and $\min(\lb \sigma_2 \rb,\lb \sigma_3 \rb) \gtrsim |\xi_2|^{\alpha}$.

Suppose that $\lb \sigma_1^* \rb = \lb \sigma_1 \rb$. Then by (\ref{eq: 111}), we can dominate (\ref{eq: N22 dual case 1}) by
\begin{align*}
    &\underset{\substack{\tau_{1}+\tau_{2}+\tau_3+\tau_4=0 \\ \xi_{1}+\xi_{2}+\xi_3+\xi_4=0}}{\int\sum^{\ast}}\frac{\prod_{i=1}^{4}f_{i}(\xi_{i},\tau_{i})}{\lb\xi_1\rb^{\frac12+}\lb \xi_3\rb^{0+}\lb \sigma_2 \rb^{\frac{\alpha+1}{6(\alpha+2)}}\lb \sigma_3 \rb^{\frac12+}\lb \sigma_4 \rb^{\frac12+}} \\
    &\lesssim 
    \left\| \left[\frac{f_1}{\lb\xi\rb^{\frac12+}}\right]^{\vee}\right\|_{L^{\infty}_{x}L^{2}_{t}}
    \left\|\left[\frac{f_2}{\lb\tau-\omega(\xi)\rb^{\frac{\alpha+2}{6(\alpha+1)}}}\right]^{\vee} \right\|_{L^{3}_{x,t}}
    \left\|\left[\frac{\lb \xi \rb^{0-}f_3}{\lb\tau-\omega(\xi)\rb^{\frac12+}}\right]^{\vee} \right\|_{L^{6}_{x,t}}
    \left\|\left[\frac{f_4}{\lb\tau-\omega(\xi)\rb^{\frac12+}}\right]^{\vee} \right\|_{L^{2}_{x}L^{\infty}_{t}} \\
    &\lesssim \prod_{i=1}^{4}\|f_i\|_{L^2_{\xi,\tau}}.
\end{align*}

Suppose that $\lb \sigma_1^* \rb = \lb \sigma_2 \rb$. Then using the $(\infty,2,3,6)$-H\"older inequality with (\ref{eq: 111}), we have
\begin{align*}
    (\ref{eq: N22 dual case 1})
    &\lesssim 
    \left\| \left[\frac{\lb\xi\rb^{-\frac12-}f_1}{\lb\tau-\omega(\xi)\rb^{\frac12+}}\right]^{\vee}\right\|_{L^{\infty}_{x,t}}
    \left\|f_2 \right\|_{L^{2}_{x,t}}
    \left\|\left[\frac{f_3}{\lb\tau-\omega(\xi)\rb^{\frac{\alpha+2}{6(\alpha+1)}}}\right]^{\vee} \right\|_{L^{3}_{x,t}}
    \left\|\left[\frac{f_4}{\lb\tau-\omega(\xi)\rb^{\frac12+}}\right]^{\vee} \right\|_{L^{6}_{x,t}} \\
    &\lesssim \prod_{i=1}^{4}\|f_i\|_{L^2_{\xi,\tau}}.
\end{align*}
If $\lb \sigma_1^{*} \rb = \lb \sigma_4 \rb$, then we may apply the $(\infty,3,6,2)$-H\"older inequality.

\noindent \underline{Subcase 2b: $i_2=2$ or $i_3=2$.} We only show the inequality
\begin{align*}
    \left\|N_{2,2}(\mathbb{P}_{D_1}u_1,\mathbb{P}_{D_1}u_2,\mathbb{P}_{D_2}u_3)\right\|_{X^{s+a,-\frac12+}} \lesssim \|\mathbb{P}_{D_2}u_3\|_{X^{s,\frac12-}}\prod_{i=1,2}\|\mathbb{P}_{D_1}u_i\|_{X^{s-1+\frac{\alpha}{2}-,\frac12}},
\end{align*}
since the remaining cases can be handled in a similar way. It suffices to show that
\begin{align}\label{eq: N22 dual case 1b}
    \underset{\substack{\tau_{1}+\tau_{2}+\tau_3+\tau_4=0 \\ \xi_{1}+\xi_{2}+\xi_3+\xi_4=0}}{\int\sum^{\ast}}\frac{|m(\xi_1,\xi_2,\xi_3)||\xi_4|^{s+a}}{|\xi_1\xi_2|^{s-1+\frac{\alpha}{2}-}|\xi_3|^{s}}\prod_{i=1}^{4} \frac{f_{i}(\xi_{i},\tau_{i})}{\lb \sigma_i\rb^{\frac{1}{2}-}}\lesssim \prod_{i=1}^{4}\|f_i\|_{L^2_{\xi,\tau}}.
\end{align}
If $s>\frac12-\frac{\alpha}{2}$, then 
\begin{align*}
    \frac{|m(\xi_1,\xi_2,\xi_3)||\xi_4|^{s+a+}}{|\xi_1\xi_2|^{s-1+\frac{\alpha}{2}-}|\xi_3|^{s-}\lb\Omega_{3}(\xi_1,\xi_2,\xi_3)\rb^{\frac12-}}=O(1)
\end{align*}
for some $a>0$. Hence for $j\in\{1,2,3,4\}$ with $\max_{1 \leq i \leq 4}\lb \sigma_i\rb=\lb \sigma_j\rb$, the left-hand side of (\ref{eq: N22 dual case 1b}) is dominated by
\begin{align} \label{eq: case ab ineq}
    \left\| f_j\right\|_{L^{2}_{x,t}}\prod_{i\neq j}\left\|\left[\frac{\lb\xi\rb^{0-}f_i}{\lb\tau-\omega(\xi)\rb^{\frac12+}}\right]^{\vee} \right\|_{L^{6}_{x,t}}
    \lesssim \prod_{i=1}^{4}\|f_i\|_{L^2_{\xi,\tau}}
\end{align}
as desired.

\noindent \textbf{Case 4: $(\xi_1,\xi_2,\xi_3,\xi_4)\in \mathcal{P}_6$.}

We only investigate the harder case
\begin{align*}
    \left\|N_{2,2}(\mathbb{P}_{D_1}u_1,\mathbb{P}_{D_1}u_2,\mathbb{P}_{D_1}u_3)\right\|_{X^{s+a,-\frac12+}} \lesssim \prod_{i=1}^{3}\|\mathbb{P}_{D_1}u_i\|_{X^{s-1+\frac{\alpha}{2}-,\frac12}}.
\end{align*}
We have to show that the quantity (\ref{eq: N22 dual case 1}) is less than $\prod_{i=1}^{4}\|f_i\|_{L^2_{\xi,\tau}}$. Assuming $\alpha>\frac32$ and $s>\frac12-\frac{\alpha}{2}$, we have
\begin{align*}
    \frac{|m(\xi_1,\xi_2,\xi_3)||\xi_4|^{s+a+}}{|\xi_1\xi_2\xi_3|^{s-1+\frac{\alpha}{2}-}\lb\Omega_{3}(\xi_1,\xi_2,\xi_3)\rb^{\frac12-}}=O(1)
\end{align*}
for some $a>0$. Hence the estimate (\ref{eq: case ab ineq}) applies in this case. This completes the proof.
\end{proof}

\subsubsection{Near-resonant terms in the \texorpdfstring{$\textnormal{deg}(P)\geq 4$}{deg(P)>=4} case} 

In the case $\textnormal{deg}(P)\geq 4$, we modify the definition of the sets $\mathcal{N}_{n}$, $\mathcal{R}^{1}_{n}$, and $\mathcal{R}^{2}_{n}$ for $n \geq 4$ as follows, which have been originally defined in Section \ref{section: Normal form reduction}:
\begin{align*}
    &\mathcal{N}_{n} = \left\{(\xi_1,\dots,\xi_{n}) \in \mathcal{X}_{n}: |\Omega_{n}(\xi_1,\dots,\xi_{n})| \gtrsim |\xi_{1}^{*}|^{\alpha}\right\}, \\
    &\mathcal{R}^{1}_{n} \subseteq \left\{(\xi_1,\dots,\xi_{n}) \in \mathcal{X}_{n}:\xi_{1}^{*}=\xi_1+\dots+\xi_{n} \textnormal{ and } |\xi_{3}^*|^{\alpha}|\xi_4^*| \ll |\xi_1^*|^{\alpha}\right\}\setminus \mathcal{N}_{n},\\
    &\mathcal{R}^{2}_{n} \subseteq \left\{(\xi_1,\dots,\xi_{n}) \in \mathcal{X}_{n}:|\xi_{3}^*|^{\alpha}|\xi_4^*| \gtrsim |\xi_1^*|^{\alpha}\right\}\setminus \mathcal{N}_{n}.
\end{align*}
The purpose of this modification is to satisfy the condition (\ref{eq: near resonance}) in Lemma \ref{lem: near resonant} below. We remark that the proof of Proposition \ref{lem: Rk multiplier} remains valid under this modification if we replace the original definition of the set $\mathcal{B}_{\theta,n}$ in that proof with
\begin{align*}
    \mathcal{B}_{\theta,n}:=\left\{ (\xi_1,\dots,\xi_{M}) \in \Z^{M}_{\ast} :  
    \left|\Omega_{\theta,n}\right| \gtrsim |\xi_1|^{\alpha} \text{ or } \xi_{\iota_{\theta,n}+1}+\cdots+\xi_M=0\right\},
\end{align*}
and then replace the condition (\ref{eq: original condition}) with 
\begin{align*}
    \xi_{\sigma(\iota_{\theta,n}+1)}+\cdots+\xi_{\sigma(M)} \neq 0 \textnormal{ and } \sigma \cdot \left|\Omega_{\theta,n}\right| \ll |\xi_1|^{\alpha}.
\end{align*}
All the other estimates in Section \ref{section: Pointwise estimates} remain unchanged, except for Lemma \ref{lem: mu pointwise bound} (c) and Lemma \ref{lem: R2 pointwise bound}, which we do not use in the case $\textnormal{deg}(P)\geq 4$.

\begin{lem} \label{lem: near resonant}
For $m:\Z^{n}_{\ast}\to \C$, let $\Lambda_{m}$ be the Fourier multiplier operator given by
\begin{align*}
    \mathcal{F}\left[\Lambda_{m}(u_1,\cdots,u_n)\right](\xi):=\sum_{\xi_1+\cdots+\xi_{n}=\xi}m(\xi_1,\dots,\xi_n)\hat{u}_1(\xi_1)\cdots\hat{u}_n(\xi_n).
\end{align*}
Assume further that
\begin{align}\label{eq: near resonance}
    \supp(m)\subseteq \{(\xi_1,\dots,\xi_n)\in\Z^{n}_{\ast}: |\Omega_n(\xi_1,\dots,\xi_n)| \ll |\xi_1^*|^{\alpha}\}.
\end{align}
    Suppose that $\alpha>\frac32$, $s>\frac12$, $\epsilon \ll 1$, 
    \begin{align} \label{eq: near resonant multiplier bound 0}
    \sup_{(\xi_1,\dots,\xi_n) \in \Z^{n}_{\ast}}\frac{|m(\xi_1,\dots,\xi_n)|\lb\xi_1+\cdots+\xi_n\rb^{s+a}}{\lb\xi_1^*\rb^{s+\frac{8\alpha^2+\alpha-6}{6(\alpha+1)}-}\prod_{i=2}^{4}\lb\xi_i^{*}\rb^{s-\frac12-}} < \infty,
    \end{align}
    and
    \begin{align} \label{eq: near resonant multiplier bound}
    \sup_{(\xi_1,\dots,\xi_n) \in \Z^{n}_{\ast}}\frac{|m(\xi_1,\dots,\xi_n)|\lb\xi_1+\cdots+\xi_n\rb^{s+a}}{\prod_{i=1}^{4}\lb\xi_i^{*}\rb^{s-}}<\infty.
\end{align}
Then we have
\begin{align*}
    \|\Lambda_{m}(u_1,\cdots,u_n)\|_{X^{s+a,-\frac12+}_{T}}\lesssim \prod_{i=1}^{n}\|u_i\|_{Y^{s}_{T}}.
\end{align*}
\end{lem}
\begin{proof}
Assume as before that $$\supp(\hat{u}_1,\dots,\hat{u}_{n}) \subseteq \{(\xi_1,\dots,\xi_n) \in\Z^n_{\ast}:|\xi_1|\geq|\xi_2|\geq \dots \geq |\xi_n|\},$$ and that $\hat{u}_{i}$ is real-valued and non-negative for each $1 \leq i \leq n+1$.
We need show the inequality
\begin{align*}
    \|\Lambda_{m}(\mathbb{P}_{D_{k_1}}u_1,\cdots,\mathbb{P}_{D_{k_n}}u_n)\|_{X^{s+a,-\frac12+}}\lesssim \prod_{i=1}^{n}\|u_i\|_{Y^{s}}
\end{align*}
for $k_1, \dots, k_{n} \in \{1,2\}$. By duality, we can reformulate this as
\begin{align} \label{eq: near resonance dual}
\begin{split}
    &\underset{\substack{\tau_1+\cdots+\tau_{n+1}=0 \\ \xi_1+\cdots+\xi_{n+1}=0}}{\int\sum^{\ast}}\frac{|m(\xi_1,\dots,\xi_n)||\xi_{n+1}|^{s+a}}{\lb \tau_{n+1}-\omega(\xi_{n+1}) \rb^{\frac12-}}\prod_{i=1}^{n}\mathbbm{1}_{D_{k_i}}\hat{u}_{i}(\xi_i,\tau_i)\hat{u}_{n+1}(\xi_{n+1},\tau_{n+1}) \\
    &\lesssim \prod_{i=1}^{n}\|u_i\|_{Y^s}\|u_{n+1}\|_{L^2_{x,t}}.
\end{split}
\end{align}
For $1 \leq i \leq n+1$, let $\sigma_i:=\tau_{i}-\omega(\xi_{i})$.

\noindent\textbf{Case 1: $k_1=1$.}

By (\ref{eq: near resonance}), we have
\begin{align*}
    \left| \sum_{i=1}^{n+1}\sigma_i\right|=|\Omega_n(\xi_1,\dots,\xi_n)| \ll |\xi_1|^{\alpha}
\end{align*}
on the support of $m$. Since $|\sigma_1| \gtrsim |\xi_1|^{\alpha}$, there exists some $2 \leq j \leq n+1$ with $|\sigma_j| \gtrsim |\xi_1|^{\alpha}$.

\noindent\underline{Subcase 1a: $|\sigma_j| \gtrsim |\xi_1|^{\alpha}$ for some $2 \leq j \leq n$.} Fix $2 \leq j \leq n$ with $|\sigma_j| \gtrsim |\xi_1|^{\alpha}$. To show (\ref{eq: near resonance dual}), it suffices to demonstrate that
\begin{align}\label{eq: subcase 1a}
    \underset{\substack{\tau_1+\cdots+\tau_{n+1}=0 \\ \xi_1+\cdots+\xi_{n+1}=0}}{\int\sum^{\ast}}\frac{|m(\xi_1,\dots,\xi_n)||\xi_{n+1}|^{s+a}\prod_{i=1}^{n}\mathbbm{1}_{D_{k_i}}\hat{u}_{i}(\xi_i,\tau_i)\hat{u}_{n+1}(\xi_{n+1},\tau_{n+1})}{|\xi_1\xi_j|^{s-1+\frac{\alpha}{2}-}\prod_{\substack{2 \leq i \leq n \\ i \neq j}}|\xi_i|^{s}\lb\sigma_1\rb^{\frac12-}\lb\sigma_5\rb^{\frac12-}\lb\sigma_{n+1}\rb^{\frac12-}}
\end{align}
is dominated by $$\prod_{i=1,j}\|\mathbb{P}_{D_{k_i}}u_i\|_{L^2_{x,t}}\prod_{\substack{2 \leq i \leq n \\ i \neq j}}\|u_i\|_{W^0}\|u_{n+1}\|_{L^2_{x,t}}.$$ By interpolating (\ref{eq: L4 Strichartz}) with the Plancherel identity or (\ref{eq: L6 Strichartz}), we have
\begin{align} \label{eq: 3+ 6- Strichartz}
    \left\|\left[\frac{f}{\lb \tau-\omega(\xi) \rb^{\frac{\alpha+2}{6(\alpha+1)}+}}\right]^{\vee}\right\|_{L^{3+}_{x,t}}
    +\left\|\left[\frac{\lb\xi\rb^{0-}f}{\lb \tau-\omega(\xi) \rb^{\frac12-}}\right]^{\vee}\right\|_{L^{6-}_{x,t}}
    \lesssim \|f\|_{L^2_{x,t}}.
\end{align}
Therefore, if
\begin{align} \label{eq: O(1)}
    \frac{|m(\xi_1,\dots,\xi_n)||\xi_{n+1}|^{s+a+}}{|\xi_1\xi_j|^{s-1+\frac{\alpha}{2}-}\lb \sigma_1 \rb^{\frac12-\frac{\alpha+2}{6(\alpha+1)}-}\lb \sigma_j \rb^{\frac12-}\prod_{\substack{2 \leq i \leq n \\ i \neq j}}|\xi_i|^{s-\frac12-}}=O(1),
\end{align}
then the desired bound follows by applying the $(3+,\infty,\dots,\infty,2,\infty,\dots,\infty,6-)$-H\"older inequality ($L^2$ norm to $u_j$) to (\ref{eq: subcase 1a}), and then using (\ref{eq: 3+ 6- Strichartz}). Now, (\ref{eq: O(1)}) can be verified by (\ref{eq: near resonant multiplier bound 0}) with the fact that $\lb \sigma_1 \rb, \lb \sigma_j \rb \gtrsim |\xi_1|^{\alpha}$.

\noindent\underline{Subcase 1b: $|\sigma_{n+1}| \gtrsim |\xi_1|^{\alpha}$.} In this case, we use the $(2+,\infty-,\infty-,\infty-,\dots,\infty,2)$-H\"older inequality to dominate the quantity
\begin{align*}
    \underset{\substack{\tau_1+\cdots+\tau_{n+1}=0 \\ \xi_1+\cdots+\xi_{n+1}=0}}{\int\sum^{\ast}}\frac{|m(\xi_1,\dots,\xi_n)||\xi_{n+1}|^{s+a}\prod_{i=1}^{n}\mathbbm{1}_{D_{k_i}}\hat{u}_{i}(\xi_i,\tau_i)\hat{u}_{n+1}(\xi_{n+1},\tau_{n+1})}{|\xi_1|^{s-1+\frac{\alpha}{2}-}\prod_{i=2}^{n}|\xi_i|^{s}\lb\sigma_1\rb^{\frac12-}\lb\sigma_{n+1}\rb^{\frac12-}}
\end{align*}
by $\|\mathbb{P}_{D_{1}}u_1\|_{L^2_{x,t}}\prod_{i=2}^{n}\|u_i\|_{W^0}\|u_{n+1}\|_{L^2_{x,t}}$. This requires
\begin{align*}
    \frac{|m(\xi_1,\dots,\xi_n)||\xi_{n+1}|^{s+a}}{|\xi_1|^{s-1+\frac{3\alpha}{2}-}\prod_{i=2}^{n}|\xi_i|^{s-\frac12-}}=O(1),
\end{align*}
which follows from (\ref{eq: near resonant multiplier bound 0}).

\noindent\textbf{Case 2: $k_1=2$.}

\noindent\underline{Subcase 2a: $k_2=1$.}
To show (\ref{eq: near resonance dual}) in this case, we consider
\begin{align} \label{eq: near resonant dual 2}
    \underset{\substack{\tau_1+\cdots+\tau_{n+1}=0 \\ \xi_1+\cdots+\xi_{n+1}=0}}{\int\sum^{\ast}}\frac{|m(\xi_1,\dots,\xi_n)||\xi_{n+1}|^{s+a}\prod_{i=1}^{n}\mathbbm{1}_{D_{k_i}}\hat{u}_{i}(\xi_i,\tau_i)\hat{u}_{n+1}(\xi_{n+1},\tau_{n+1})}{|\xi_1|^{s}|\xi_2|^{s-1+\frac{\alpha}{2}-}\prod_{i=3}^{n}|\xi_i|^{s}\prod_{i=1}^{3}\lb\sigma_{i}\rb^{\frac12-}\lb\sigma_{n+1}\rb^{\frac12-}},
\end{align}
and claim that this quantity is dominated by $\prod_{i=1}^{3}\|\mathbb{P}_{D_{k_i}}u_i\|_{L^2_{x,t}}\prod_{i=4}^{n}\|u_i\|_{W^0}\|u_{n+1}\|_{L^2_{x,t}}$. Since $\alpha>\frac32$, using $\lb\sigma_{2}\rb^{\frac12-} \gtrsim |\xi_1|^{\frac{\alpha}{2}-}$ and (\ref{eq: near resonant multiplier bound}), we have
\begin{align*}
    \frac{|m(\xi_1,\dots,\xi_n)||\xi_{n+1}|^{s+a+}}{|\xi_1|^{s-}|\xi_2|^{s-1+\frac{\alpha}{2}-}|\xi_3|^{s-}|\xi_4|^{s-\frac12-}\lb\sigma_{2}\rb^{\frac12-}}=O(1).
\end{align*}
Hence, by applying the $(6-,2+,6-,\infty,\dots,\infty,6-)$-H\"older inequality to (\ref{eq: near resonant dual 2}), we obtain the desired bound.

\noindent\underline{Subcase 2b: $k_2=2$.} We only look at the harder case $(k_2,k_3)=(2,1)$. In this case, we may apply the $(6-,6-,2+,\infty,\dots,\infty,6-)$-H\"older inequality to dominate
\begin{align*}
    \underset{\substack{\tau_1+\cdots+\tau_{n+1}=0 \\ \xi_1+\cdots+\xi_{n+1}=0}}{\int\sum^{\ast}}\frac{|m(\xi_1,\dots,\xi_n)||\xi_{n+1}|^{s+a}\prod_{i=1}^{n}\mathbbm{1}_{D_{k_i}}\hat{u}_{i}(\xi_i,\tau_i)\hat{u}_{n+1}(\xi_{n+1},\tau_{n+1})}{|\xi_1 \xi_2|^{s}|\xi_3|^{s-1+\frac{\alpha}{2}-}\prod_{i=4}^{n}|\xi_i|^{s}\prod_{i=1}^{3}\lb\sigma_{i}\rb^{\frac12-}\lb\sigma_{n+1}\rb^{\frac12-}}.
\end{align*}
by $\prod_{i=1}^{3}\|\mathbb{P}_{D_{k_i}}u_i\|_{L^2_{x,t}}\prod_{i=4}^{n}\|u_i\|_{W^0}\|u_{n+1}\|_{L^2_{x,t}}$. For this, we need (\ref{eq: near resonant multiplier bound}) with $\alpha>\frac32$.
\end{proof}

\begin{cor}
Let $0<T<1$ and $\frac32<\alpha < 2$. Then the followings hold:
    \begin{enumerate}[label=(\alph*)] 
    \item For $N \geq 2$, fix $\mathcal{K}:=\{k_1^{n_1},\dots,k_m^{n_m}\}$ with $\sum_{j=1}^{m}n_{j}=N$. Let $M:=1+\sum_{j=1}^{m}n_{j}(k_{j}-1)$ and suppose that $M \geq 4$. Then for $s>\frac12$, we have
    \begin{align*}
        \bigg\| \sum_{\theta \in \textnormal{Perm}(\mathcal{K})}R^{1}_{\theta(1),\dots,\theta(N)}(u)\bigg\|_{X^{s+a,-\frac12+}_{T}} \lesssim \|u\|_{Y^s_T}^{M}
    \end{align*}
    for some $a>0$.
    \item Let $n \geq 0$, $d \geq 2$, and $(k_0,\dots,k_n) \in [d]^{n+1}$. Then for $s>\frac{1}{3}+\frac{1}{3\alpha}$, we have
    \begin{align*}
        \|R^{2}_{k_0,\dots,k_n}(u)\|_{X^{s+a,-\frac12+}_{T}} \lesssim \|u\|_{Y^s_T}^{\nu_n}
    \end{align*}
    for some $a>0$.
\end{enumerate}
\end{cor}

\begin{proof}
We only need to check the conditions in Lemma \ref{lem: near resonant} for the associated symbols. We omit the elementary verification.

\end{proof}

\subsection{Proof of local well-posedness and smoothing} We close this section by proving Theorem \ref{thm: lwp}. Notice first that using the identity
\begin{align*}
\prod_{i=1}^{n}x_i-\prod_{i=1}^{n}y_i = \sum_{i=1}^{n}\left(\prod_{j=1}^{i-1}y_j\right)\left(\prod_{j=i+1}^{n}x_j\right)(x_i-y_i),
\end{align*}
all the estimates in the previous subsection can be extended to the estimates for the differences.

\begin{proof}[Proof of Theorem \ref{thm: lwp}]

We first consider (a). For $N>0$ and $u\in Y^{s}_{T}$, define $\Gamma_{N}(u)$ as the right-hand side of (\ref{eq: contraction}). Then by Lemma \ref{lem: Ys space} and the estimates in subsection \ref{subsection: Bourgain space estimates for the nonlinear terms}, we have
\begin{align*}
    \|\Gamma_N(u)\|_{Y^{s}_{T}} 
    &\lesssim \|g\|_{H^s}+\|g\|_{H^s}^2+\left\|\mathbb{P}_{\{|\xi|\geq N\}}B_2(u)(t)\right\|_{Y^s_T}+T^{0+}\left\|\mathbb{P}_{\{|\xi|< N\}}\partial_{x}u^2(t')\right\|_{X^{s,-\frac12+}_T} \\
    &\quad+T^{0+}\left\|R^{1}_{2,2}(u)(t')+R^{2}_{2,2}(u)(t')+N_{2,2}(u)\right\|_{X^{s,-\frac12+}_T}
    \\
    & \lesssim \|g\|_{H^s}+\|g\|_{H^s}^2+N^{-(\alpha-1)\epsilon}\|u\|_{Y^{s}_T}^2+T^{0+}\left(\|u\|_{Y^s_T}^2+\|u\|_{Y^s_T}^3\right).
\end{align*}
The estimate for the difference 
\begin{multline*}
    \|\Gamma_N(u)-\Gamma_N(v)\|_{Y^{s}_{T}}
    \lesssim N^{-(\alpha-1)\epsilon}\left(\|u\|_{Y^{s}_T}+\|v\|_{Y^{s}_T}\right)\|u-v\|_{Y^{s}_T} \\
    +T^{0+}\left(\|u\|_{Y^{s}_T}+\|u\|_{Y^{s}_T}^2+\|v\|_{Y^{s}_T}+\|v\|_{Y^{s}_T}^2\right)\|u-v\|_{Y^{s}_T}
\end{multline*}
can be similarly deduced. Therefore, for $N \gg 1$ and $T \ll 1$, the map $\Gamma_{N}$ is a contraction on $\{u \in Y^{s}_T:\|u\|_{Y^{s}_T} \leq C(\|g\|_{H^s}+\|g\|_{H^s}^2)\}$. Hence there exists a unique solution to $u=\Gamma_{N}(u)$. Also, the solutions obtained by this argument satisfy the estimate
\begin{align*}
    \|u\|_{Y^{s}_T} \lesssim \|g\|_{H^s}+\|g\|_{H^s}^2.
\end{align*}
Using this with Lemma \ref{lem: B2 deg 2 estimate}, Corollary \ref{cor: general}, and Lemma \ref{lem: N22 Xsb}, we find the smoothing estimate (\ref{eq: nonlinear smoothing}).

For (b) and (c), we define the operator $\Gamma_N(u)$ as the right-hand side of (\ref{eq: gkdv normal form}). Now the remaining argument is similar to that of (a).
\end{proof}

\appendix

\section{Upper bound of the smoothing exponent} \label{appendix: Upper bound of the smoothing exponent}

In Theorem \ref{thm: smoothing}, in the $P(x)=x^2$ case, we gain $\alpha-1$ derivatives for $s>\frac12$:
    \begin{align*}
        u(t)-e^{t\partial_{x}D^{\alpha}_{x}}g \in C([-T,T],H^{s+\alpha-1}).
    \end{align*}
The following result suggests that the order of smoothing $\alpha-1$ is likely optimal. It should be compared with the result in \cite[Theorem 2]{IMT}.
\begin{prop}
    If $\alpha \geq 1$ and $a>\alpha-1$, then for any $s \in \R$ and $T>0$, there does not exist $C(\|g\|_{H^{s}})$ such that the estimate
\begin{align*}
     \left\| \int_{0}^{t}S(t-t')\left[ S(t')g\partial_{x} S(t')g\right] dt' \right\|_{C([-T,T],H^{s+a})} \lesssim C(\|g\|_{H^{s}})
\end{align*}
holds true. 
\end{prop}

\begin{proof}
    Let $a>\alpha-1$. For $N \gg 1$ let
    \begin{align*}
    \hat{g}(\xi)=
    \begin{cases}
        N^{-s}, &\,\text{if } \xi=N-1, \\
        1, &\,\text{if } \xi=1, \\
        0, &\, \text{otherwise}.
    \end{cases}
\end{align*}
Then $\|g\|_{H^{s}(\T)}\approx 1$. Since $\Omega_{2}(N-1,1) \approx N^{\alpha}$, we have for $t_N:=\pi/\Omega_{2}(N-1,1)$,
\begin{align*}
    &\left \| \int_{0}^{t_{N}}S(t_{N}-t) \left [ \partial_{x} \left ( S(t)g \right )^2\right ] dt \right \|_{H^{s+a}} \\
    &\quad=\left \| |\xi|^{s+a} \sum_{\xi_1+\xi_2=\xi} \frac{2\xi}{\Omega_{2}(\xi_1,\xi_2)} \left(e^{i\Omega_{2}(\xi_1,\xi_2)t_{N}} -1\right)\hat{g}_{\xi_1}\hat{g}_{\xi_2}\right \|_{\ell^{2}_{\xi}}
    \\
    &\quad\geq \left \| \mathbbm{1}_{\{\xi=N\}}|\xi|^{s+a} \sum_{\xi_1+\xi_2=\xi} \frac{2\xi}{\Omega_{2}(\xi_1,\xi_2)} \left(e^{i\Omega_{2}(\xi_1,\xi_2)t_{N}} -1\right)\hat{g}_{\xi_1}\hat{g}_{\xi_2}\right \|_{\ell^{2}_{\xi}} \\
    &\quad\gtrsim N^{a+1-\alpha}.
\end{align*}
Since $t_N \to 0$, and $N^{a+1-\alpha} \to \infty$ as $N \to \infty$, the desired result follows. 
\end{proof}

\section{The fifth-order KdV equation} \label{appendix: The fifth-order KdV equation}

In this appendix, we illustrate how our method improves the well-posedness results for the (non-integrable) fifth-order KdV equation 
\begin{equation}\label{eq: 5KdV}
        \begin{cases}
        \partial_{t} u - \partial_{x}^{5} u + \alpha\partial_{x}(u^3)+\beta \partial_{x}(\partial_{x}u)^2+\gamma \partial_{x}(u\partial_{x}^{2}u)=0, \\
        u(x,0) = g \in H^{s}(\T),
    \end{cases}
    (x,t) \in \T \times [-T,T], 
\end{equation}
where $\alpha,\beta,\gamma \in \R$, and $g \in H^{s}(\T)$ is a mean-zero, real-valued function. Our main result in this appendix is:
\begin{thm} \label{thm: 5kdv lwp}
For any $\alpha,\beta,\gamma \in \R$, the equation (\ref{eq: 5KdV}) is locally well-posed for $s>\frac12$.
\end{thm}
\begin{remark}
    If $\gamma=2\beta$, the flow map in Theorem \ref{thm: 5kdv lwp} is real-analytic on the level set $\mathscr{M}_{0,c}^{s}:=\{u \in H^{s}(\T): \hat{u}(0)=0, \ \|u\|_{L^{2}}=c\}$ for each $c>0$. This is because the solutions are obtained via the Piccard iteration after a gauge transform, and the gauge transform used in \cite{Mc2} is a simple translation $u \mapsto u(x-4ct/5,t)$ on $\mathscr{M}_{0,c}^{s}$ (by the conservation of the $L^{2}$-norm). On the other hand, in the completely integrable case, the flow map is nowhere locally uniformly continuous on $\mathscr{M}_{0,c}^{s}$ for any $c>0$ and $0 \leq s <\frac12$; see \cite{KM}. Therefore, it is not likely that the method we used in the proof of \ref{thm: 5kdv lwp} extends to the $s <\frac12$ case.
\end{remark}
In \cite{Mc2}, McConnell proved that (\ref{eq: 5KdV}) is locally well-posed for $s>\frac{35}{64}$. He performed normal form reduction to (\ref{eq: 5KdV}), and then estimated the resulting terms in the function spaces defined by the norms
\begin{align} \label{eq: 5KdV YZ}
\begin{split}
    &\|u\|_{Y^{s}}:=\|u\|_{X^{s,\frac14}}+\|u\|_{W^s}, \\
    &\|u\|_{Z^{s}}:=\left\|\mathcal{F}^{-1}\left[\lb\tau-\xi^5\rb^{-1}\hat u\right] \right\|_{Y^{s}},
\end{split}
\end{align}
where we define $\|u\|_{X^{s,b}}:=\|\lb \xi \rb^{s}\lb \tau-\xi^5 \rb^{b}\hat{u}\|_{L^{2}_{\xi,\tau}}$ and $\|u\|_{W^s}:=\|\lb \xi \rb^{s}\hat{u}\|_{\ell^{2}_{\xi}L^{1}_{\tau}}$. There, he used the estimate
\begin{align} \label{eq: 5KdV L8}
    \big\|e^{t\partial_{x}^{5}}g\big\|_{L^{8}_{\T \times [-T,T]}} \lesssim \|g\|_{H^{a+}}
\end{align}
to prove the inequality 
\begin{align*}
    \|\mathcal{NR}_2(u)\|_{Z^{s}_T}+\|\mathcal{NR}_3(u)\|_{Z^{s}_T} \lesssim T^{\theta} \|u\|_{Y^s_T}^{4}
\end{align*}
for $s>\frac{1+a}{2}$, where $\mathcal{NR}_2(u)$ and $\mathcal{NR}_3(u)$ are the near-resonant terms associated with (\ref{eq: 5KdV}). Since all the other relevant terms are estimated in the $s>\frac12$ level, this leads to local well-posedness for (\ref{eq: 5KdV}) for $s>\frac{1+a}{2}$. It is conjectured that the estimate (\ref{eq: 5KdV L8}) holds for $a=0$, however, the currently best known value of $a$ so far is $\frac{3}{32}$.

The estimate (\ref{eq: 5KdV L8}) was essential in \cite{Mc2} because the modulation $\frac14$ in the space (\ref{eq: 5KdV YZ}) was too low to handle some necessary Strichartz estimates. However, we can avoid this difficulty by simply replacing the spaces (\ref{eq: 5KdV YZ}) by
\begin{align}\label{eq: 5KdV tYZ}
\begin{split}
    &\|u\|_{\t Y^{s}}:=\|\mathbb{P}_{D_1}u\|_{X^{s-1-\epsilon,\frac12}}+\|\mathbb{P}_{D_2}u\|_{X^{s,\frac12-\epsilon}}+\|u\|_{W^s}, \\
    &\|u\|_{\t Z^{s}}:=\|u\|_{X^{s,-\frac12+}},
\end{split}
\end{align}
where
\begin{align*}
    & D_1:=\left \{(\xi,\tau) \in \Z_{\ast}\times\R:  \lb \tau-\xi^5 \rb \gtrsim |\xi|^{4} \right\}, \\
    & D_2:=\left(\Z_{\ast}\times\R\right)\setminus D_1.
\end{align*}
Following the proof of Lemma \ref{lem: B Xsb}, we can see that the boundary terms $B_1,B_2$ in \cite{Mc2} are controlled in $Y^s$ for $s>\frac12$:
\begin{align*}
    \|\mathbb{P}_{\{|\xi|\geq N\}}B_1(u)\|_{\t Y^s_T}+\|\mathbb{P}_{\{|\xi|\geq N\}}B_2(u)\|_{\t Y^s_T} \lesssim N^{-\epsilon}\big(\|u\|_{\t Y^s_T}^2+\|u\|_{\t Y^s_T}^3\big).
\end{align*}

\begin{lem}
Let $0<T<1$. For $0<\epsilon \ll 1$ and $s>\frac12$, we have
   \begin{align*}
    \|\mathcal{NR}_2(u)\|_{\t Z^{s}_T}+\|\mathcal{NR}_3(u)\|_{\t Z^{s}_T} \lesssim T^{\theta} \|u\|_{\t Y^s_T}^{4}
\end{align*}
for some $\theta>0$.
\end{lem}
\begin{proof}
Since there is no distinction between $\mathcal{NR}_2$ and $\mathcal{NR}_3$ in this proof, we drop the subscripts and simply write them $\mathcal{NR}$. Assume by possibly rearranging the frequencies that that $$\supp(\hat{u}_1,\dots,\hat{u}_{n}) \subseteq \{(\xi_1,\dots,\xi_n) \in\Z^n_{\ast}:|\xi_1|\geq|\xi_2|\geq \dots \geq |\xi_n|\}.$$ For $k_1,\dots,k_4 \in \{1,2\}$, we need to show
\begin{align*}
    \|\mathcal{NR}(\mathbb{P}_{D_{k_1}}u_1,\mathbb{P}_{D_{k_2}}u_2,\mathbb{P}_{D_{k_3}}u_3,\mathbb{P}_{D_{k_4}}u_4)\|_{X^{s,-\frac12+}} \lesssim \prod_{i=1}^{4}\|u_i\|_{Y^s}.
\end{align*}
Let $m$ be the symbol associated with $\mathcal{NR}$. By duality, it suffices to show that
\begin{align} \label{eq: 5kdv dual 0}
    \underset{\substack{\tau_1+\cdots+\tau_{5}=0 \\ \xi_1+\cdots+\xi_{5}=0}}{\int\sum^{\ast}}\frac{|m(\xi_1,\xi_2,\xi_3,\xi_4)||\xi_{5}|^{s+a}}{\lb \tau_{5}-\omega(\xi_{5}) \rb^{\frac12-}}\prod_{i=1}^{4}\mathbbm{1}_{D_{k_i}}\hat{u}_{i}(\xi_i,\tau_i)\hat{u}_{5}(\xi_5,\tau_5)
    \lesssim \prod_{i=1}^{4}\|u_i\|_{Y^s}\|u_{5}\|_{L^2_{x,t}}.
\end{align}
Recall that we have
\begin{align*}
    |m(\xi_1,\xi_2,\xi_3,\xi_4)| \lesssim |\xi_2\xi_3|^{\frac12}|\xi_4|^{-\frac12}   
\end{align*}
and
\begin{align*}
    \supp(m) \subseteq \left\{|\Phi_4(\xi_1,\xi_2,\xi_3,\xi_4)| \ll |\xi_1^*|^4\right\}\cap\left\{|\xi_3^*|^{4}|\xi_4^*|\gtrsim |\xi_1^*|^4 \textnormal{ or } \xi_1^*+\xi_2^*=0\right\},
\end{align*}
where $\Phi_4(\xi_1,\xi_2,\xi_3,\xi_4):=(\xi_1+\xi_2+\xi_3+\xi_4)^5-\xi_1^5-\xi_2^5-\xi_3^5-\xi_4^5$.

\noindent\textbf{Case 1: $k_1=1$.}

For $1 \leq i \leq 5$, let $\sigma_i:=\tau_i-\xi_i^5$. Then on the support of $m$, we have
\begin{align*}
    \left| \sum_{i=1}^{5}\sigma_i\right|=|\Phi_4(\xi_1,\xi_2,\xi_3,\xi_4)| \ll |\xi_1|^{4}.
\end{align*}
Since $|\sigma_1| \gtrsim |\xi_1|^{4}$, there exists some $2 \leq j \leq 5$ with $|\sigma_j| \gtrsim |\xi_1|^4$.

\noindent\underline{Subcase 1a: $|\sigma_2| \gtrsim |\xi_1|^4$.} To show (\ref{eq: 5kdv dual 0}), it suffices to dominate the quantity
\begin{align} \label{eq: 5kdv dual}
    \underset{\substack{\tau_1+\cdots+\tau_{5}=0 \\ \xi_1+\cdots+\xi_{5}=0}}{\int\sum^{\ast}}\frac{|m(\xi_1,\xi_2,\xi_3,\xi_4)||\xi_{5}|^{s}\prod_{i=1}^{4}\mathbbm{1}_{D_{k_i}}\hat{u}_{i}(\xi_i,\tau_i)\hat{u}_{5}(\xi_5,\tau_5)}{|\xi_1\xi_2|^{s-1-}|\xi_3\xi_4|^{s}\lb\sigma_1\rb^{\frac12-}\lb\sigma_2\rb^{\frac12-}\lb\sigma_{5}\rb^{\frac12-}}
\end{align}
by $\prod_{i=1}^{2}\|\mathbb{P}_{D_{k_i}}u_i\|_{L^2_{x,t}}\prod_{i=3}^{4}\|u_i\|_{W^0}\|u_{5}\|_{L^2_{x,t}}$. Note that using $\lb\sigma_1\rb^{\frac12-}\lb\sigma_2\rb^{\frac12-} \gtrsim |\xi_1|^{4-}$, we have for $s>\frac12$,
\begin{align*}
    \frac{|m(\xi_1,\xi_2,\xi_3,\xi_4)||\xi_{5}|^{s}}{|\xi_1\xi_2|^{s-1-}|\xi_3\xi_4|^{s}\lb\sigma_1\rb^{\frac12-}\lb\sigma_2\rb^{\frac12-}\lb\sigma_{5}\rb^{\frac12-}}
    \lesssim \frac{1}{|\xi_3\xi_4\xi_5|^{\frac12+}\lb\sigma_{2}\rb^{0+}\lb\sigma_{5}\rb^{\frac12-}}.
\end{align*}
Now we may apply the $(2,2+,\infty,\infty,\infty-)$-H\"older inequality to (\ref{eq: 5kdv dual}) to obtain the desired bound.

\noindent\underline{Subcase 1b: $|\sigma_3| \gtrsim |\xi_1|^4$ or $|\sigma_4| \gtrsim |\xi_1|^4$.} We can repeat the same analysis as in the former subcase. For example, if $|\sigma_3| \gtrsim |\xi_1|^4$, then by the $(2,\infty,2+,\infty,\infty-)$-H\"older inequality we have
\begin{align*}
    &\underset{\substack{\tau_1+\cdots+\tau_{5}=0 \\ \xi_1+\cdots+\xi_{5}=0}}{\int\sum^{\ast}}\frac{|m(\xi_1,\xi_2,\xi_3,\xi_4)||\xi_{5}|^{s}\prod_{i=1}^{4}\mathbbm{1}_{D_{k_i}}\hat{u}_{i}(\xi_i,\tau_i)\hat{u}_{5}(\xi_5,\tau_5)}{|\xi_1\xi_3|^{s-1-}|\xi_2\xi_4|^{s}\lb\sigma_1\rb^{\frac12-}\lb\sigma_3\rb^{\frac12-}\lb\sigma_{5}\rb^{\frac12-}} \\
    &\lesssim \prod_{i=1,3}\|\mathbb{P}_{D_{k_i}}u_i\|_{L^2_{x,t}}\prod_{i=2,4}\|u_i\|_{W^0}\|u_{5}\|_{L^2_{x,t}}.
\end{align*}

\noindent\underline{Subcase 1c: $|\sigma_5| \gtrsim |\xi_1|^4$.} In this case, we use the $(2,\infty,\infty,\infty,2)$-H\"older inequality to dominate the quantity
\begin{align*}
    \underset{\substack{\tau_1+\cdots+\tau_{5}=0 \\ \xi_1+\cdots+\xi_{5}=0}}{\int\sum^{\ast}}\frac{|m(\xi_1,\xi_2,\xi_3,\xi_4)||\xi_{5}|^{s}\prod_{i=1}^{4}\mathbbm{1}_{D_{k_i}}\hat{u}_{i}(\xi_i,\tau_i)\hat{u}_{5}(\xi_5,\tau_5)}{|\xi_1|^{s-1}|\xi_2\xi_3\xi_4|^{s}\lb\sigma_1\rb^{\frac12-}\lb\sigma_{5}\rb^{\frac12-}}
\end{align*}
by $\|\mathbb{P}_{D_{1}}u_1\|_{L^2_{x,t}}\prod_{i=2}^{4}\|\hat{u}_i\|_{L^1_{\tau}\ell^{2}_{\xi}}\|u_{5}\|_{L^2_{x,t}}$. Indeed this subcase is more favorable than subcase 1a.

\noindent\textbf{Case 2: $k_1=2$.}

Invoking the inequality
\begin{align*}
    \|u\|_{X^{s,\frac14-}} \lesssim \|\mathbb{P}_{D_1}u\|_{X^{s-1-,\frac12}}+\|\mathbb{P}_{D_2}u\|_{X^{s,\frac12-}},
\end{align*}
it suffices to dominate the quantity
\begin{align} \label{eq: NR dual}
    \underset{\substack{\tau_1+\cdots+\tau_{5}=0 \\ \xi_1+\cdots+\xi_{5}=0}}{\int\sum^{\ast}}\frac{|m(\xi_1,\xi_2,\xi_3,\xi_4)||\xi_{5}|^{s}\prod_{i=1}^{4}\mathbbm{1}_{D_{k_i}}\hat{u}_{i}(\xi_i,\tau_i)\hat{u}_{5}(\xi_5,\tau_5)}{|\xi_1\xi_2\xi_3\xi_4|^{s}\lb\sigma_1\rb^{\frac12-}\lb\sigma_2\rb^{\frac14-}\lb\sigma_3\rb^{\frac14-}\lb\sigma_{5}\rb^{\frac12-}}
\end{align}
by $\prod_{i=1}^{3}\|\mathbb{P}_{D_{k_i}}u_i\|_{L^2_{x,t}}\|\hat{u}_4\|_{W^0}\|u_5\|_{L^2_{x,t}}$. Note that for $s>\frac12$, we have
\begin{align*}
    \frac{|m(\xi_1,\xi_2,\xi_3,\xi_4)||\xi_5|^{s+}}{|\xi_1|^{s-}|\xi_2\xi_3|^{s}|\xi_4|^{s-\frac12-}}=O(1).
\end{align*}
Using this, we may apply the $(6-,3+,3+,\infty,6-)$-H\"older inequality and to (\ref{eq: NR dual}). Then we obtained the desired bound by the Strichartz estimate
\begin{align*}
    \left\| \left[\frac{f}{\lb \tau-\xi^5\rb^{\frac15+}}\right]^{\vee}\right\|_{L^{3+}_{x,t}}+\left\| \left[\frac{\lb \xi \rb^{0-}f}{\lb \tau-\xi^5\rb^{\frac12-}}\right]^{\vee}\right\|_{L^{6-}_{x,t}} \lesssim \|f\|_{L^2_{\xi,\tau}}.
\end{align*}
\end{proof}
In \cite{Mc2}, the proofs of the nonlinear estimates there do not take advantage of the $-\frac34$ modulation in the $Z^s$ norm. Rather, they are actually done for the $X^{s,-\frac12+}$ norm; generically, these estimates are like
\begin{align*}
    \|\Lambda(u)\|_{Z^s} \lesssim \|\Lambda(u)\|_{X^{s,-\frac12+}} \lesssim \|u\|_{Y^s}^{n}
\end{align*}
for some Fourier multiplier operator $\Lambda$. Since we have
\begin{align*}
    \|u\|_{X^{s,\frac14-\frac{\epsilon}{4}}}+\|u\|_{W^s} \lesssim \|u\|_{\t Y^{s}},
\end{align*}
and the analysis in \cite{Mc2} is not sensitive to $\epsilon>0$, all the nonlinear estimates stated there for the spaces (\ref{eq: 5KdV YZ}) can also be verified for (\ref{eq: 5KdV tYZ}). This with a contraction argument proves Theorem \ref{thm: 5kdv lwp}.

\section*{Acknowledgements}
The author is grateful to Professor Nikolaos Tzirakis for many helpful comments and discussions. Also, Ryan McConnell provided helpful feedback on the first draft of the paper. In particular, his suggestion was essential in improving Lemma \ref{lem: N22 Xsb}.

\end{document}